\documentclass[reqno,12pt]{amsart} 
\usepackage[colorlinks=true,linkcolor=blue,citecolor=red]{hyperref} 
\usepackage{esint} 
\usepackage{amssymb} 
\usepackage{enumitem} 
\usepackage{mathrsfs} 

\setenumerate{label=(\alph*)} 
\setenumerate[2]{label=(\Alph*)} 

\allowdisplaybreaks[1] 

\newtheorem{theorem}{Theorem}[section]
\newtheorem{proposition}[theorem]{Proposition}
\newtheorem{lemma}[theorem]{Lemma}

\theoremstyle{definition}
\newtheorem{definition}[theorem]{Definition}
\newtheorem{example}[theorem]{Example}
\theoremstyle{remark}
\newtheorem{remark}[theorem]{Remark}

\numberwithin{equation}{section} 

\def\diam{\operatorname{diam}}
\def\div{\operatorname{div}}
\def\dist{\operatorname{dist}}

\def\BMO{\operatorname{BMO}}

\def\inf{\operatornamewithlimits{inf\vphantom{p}}}
\def\length{\operatorname{length}}

\newcommand{\norm}[1]{\left\lVert#1\right\rVert}

\calclayout

\begin{document}

\title[Characterizations of Lyapunov domains]
{Characterizations of Lyapunov domains in terms of Riesz transforms and the Plemelj-Privalov theorem}

\author{Juan Jos\'{e} Mar\'{i}n}
\address{Juan Jos\'{e} Mar\'{i}n
\\
Departamento de Matem{\'a}tica Aplicada y Estad{\'i}stica
\\
Universidad Polit{\'e}cnica de Cartagena
\\
Cartagena, Spain.} \email{juanjose.marin@upct.es}

\author{Jos\'e Mar{\'\i}a Martell}
\address{Jos\'e Mar{\'\i}a Martell
\\
Instituto de Ciencias Matem\'aticas CSIC-UAM-UC3M-UCM
\\
Consejo Superior de Investigaciones Cient{\'\i}ficas
\\
C/ Nicol\'as Cabrera, 13-15
\\
E-28049 Madrid, Spain} \email{chema.martell@icmat.es}

\author{Dorina Mitrea}
\address{Dorina Mitrea
\\
Department of Mathematics
\\
Baylor University 
\\
Waco, TX 76706, USA} \email{Dorina\_Mitrea@baylor.edu}

\author{Marius Mitrea}
\address{Marius Mitrea
\\
Department of Mathematics
\\
Baylor University 
\\
Waco, TX 76706, USA} \email{Marius\_Mitrea@baylor.edu}

\date{\today}



\begin{abstract}
We prove several characterizations of $\mathscr{C}^{1,\omega}$-domains (aka Lyapunov domains), where 
$\omega$ is a growth function satisfying natural assumptions. For example, given an Ahlfors regular 
domain $\Omega\subseteq{\mathbb{R}}^n$, we show that the modulus of continuity of the geometric measure 
theoretic outward unit normal $\nu$ to $\Omega$ is dominated by (a multiple of) $\omega$ if and only if 
the action of each Riesz transform $R_j$ associated with $\partial\Omega$ on the constant function $1$ 
has a modulus of continuity dominated by (a multiple of) $\omega$. The proof of this result requires that we 
establish a higher-dimensional generalization of the classical Plemelj-Privalov theorem, identifying a large class 
of singular integral operators that are bounded on generalized H\"older spaces. This class includes the 
Cauchy–Clifford operator and the harmonic double layer operator, among others.
\end{abstract}

\maketitle
\tableofcontents

\section{Introduction and statement of main result}\label{sec:intro}

The principal aim of this paper is to characterize Lyapunov $\mathscr{C}^{1,\omega}$-domains in 
${\mathbb{R}}^n$, $n\geq 2$. One way to think of such a domain $\Omega\subset{\mathbb{R}}^n$ is as an open 
set of locally finite perimeter whose geometric measure theoretic outward unit normal $\nu$, after possibly 
being redefined on a set of $\sigma$-measure zero, belongs to ${\mathscr{C}}^{\omega}(\partial\Omega)$.
Here, ${\mathscr{C}}^{\omega}(\partial\Omega)$ is a generalized H\"{o}lder space, quantifying 
continuity in terms of the modulus, or ``growth'' function, $\omega$ (cf. Definition~\ref{tFF}), and 
$\sigma:={\mathcal{H}}^{n-1}\lfloor\partial\Omega$ is the ``surface measure'' on $\partial\Omega$
(with ${\mathcal{H}}^{n-1}$ denoting the $(n-1)$-dimensional Hausdorff measure in ${\mathbb{R}}^{n}$). 
The above class of domains may be equivalently described as the 
collection of all open subsets of ${\mathbb{R}}^n$ which locally coincide (up to a rigid transformation 
of the space) with the upper-graph of a real-valued continuously differentiable function defined in 
${\mathbb{R}}^{n-1}$ whose first-order partial derivatives belong to ${\mathscr{C}}^{\omega}({\mathbb{R}}^{n-1})$;
see Remark~\ref{YrafV.b}.

The characterizations of the class of Lyapunov domains we presently seek are in terms of the 
boundedness properties of certain classes of singular integral operators acting on generalized H\"{o}lder spaces.
The most prominent examples of such singular integral operators are offered by the Riesz transforms.
On the boundary of an Ahlfors regular domain $\Omega\subseteq{\mathbb{R}}^n$ (cf. Definition~\ref{utRF-he42}) 
with compact boundary there are two versions we wish to consider. First, for each $j\in\{1,\dots,n\}$, we define 
the action of the $j$-th principal-value Riesz transform on $\partial\Omega$ on each function $f\in L^1(\partial\Omega,\sigma)$ as
\begin{equation}\label{Rieszpv}
(R_jf)(x):=\lim_{\varepsilon\to 0^{+}}\frac{1}{\varpi_{n-1}}\int\limits_{\substack{y\in\partial\Omega\\ |x-y|>\varepsilon}}
\frac{x_j-y_j}{|x-y|^n}f(y)\,d\sigma(y)\,\,\text{ for $\sigma$-a.e. }\,\,x\in\partial\Omega,
\end{equation}
where $\varpi_{n-1}$ stands for the area of the unit sphere $S^{n-1}$ in ${\mathbb{R}}^n$ (for the fact that the above limit
exists see, e.g., \cite[Proposition~5.6.7, p.\,378]{GHA.I}, whose proof makes use of results from \cite{Mas} and \cite{Mattila}).

Second, for each $j\in\{1,\dots,n\}$, we define the $j$-th distributional Riesz transform on $\partial\Omega$ as the operator  
\begin{equation}\label{Rjdistributional1}
\begin{array}{c}
\text{$R_j:{\mathscr{C}}^{\omega}(\partial\Omega)\longrightarrow\big({\mathscr{C}}^{\omega}(\partial\Omega)\big)^{*}$
satisfying, for every $f,g\in{\mathscr{C}}^{\omega}(\partial\Omega)$,}
\\[8pt]
\displaystyle
\langle R_j f,g\rangle=\frac{1}{2\varpi_{n-1}}\int_{\partial\Omega}
\int_{\partial\Omega}\frac{x_j-y_j}{|x-y|^n}\big(f(y)g(x)-f(x)g(y)\big)\,d\sigma(y)\,d\sigma(x),
\end{array}
\end{equation}
where $\langle\cdot,\cdot\rangle$ denotes the natural duality pairing between 
$\big({\mathscr{C}}^{\omega}(\partial\Omega)\big)^{*}$ and ${\mathscr{C}}^{\omega}(\partial\Omega)$. 
From the $T(1)$ theorem for spaces of homogeneous type (cf., e.g., \cite[Theorem~13, p.\,94]{Chr90}) we know that for each 
$j\in\{1,\dots,n\}$ the operator $R_j$ from \eqref{Rjdistributional1} extends to a bounded linear mapping on 
$L^2(\partial\Omega,\sigma)$ if and only if $R_j 1$, again considered in the sense of \eqref{Rjdistributional1}, 
belongs to $\BMO(\partial\Omega,\sigma)$, the John-Nirenberg space 
of functions with bounded mean oscillation on $\partial\Omega$ (with respect to the measure $\sigma$). 
Moreover, said membership is actually equivalent to $\partial\Omega$ being uniformly rectifiable 
(cf. \cite[p.\,239]{NazTolVol14}). In turn, the latter geometric property ensures that the principal-value Riesz transforms
on $\partial\Omega$ (cf. \eqref{Rieszpv}) induce well-defined and bounded operators on $L^2(\partial\Omega,\sigma)$
(cf. \cite{DavSem93}, \cite[Theorem~5.10.3, pp.\,458-459]{GHA.I}). Ultimately this goes to show that the aforementioned extension 
of $R_j$ from \eqref{Rjdistributional1} to a bounded mapping on $L^2(\partial\Omega,\sigma)$ is given by the principal-value 
integral operator acting on each $f\in L^2(\partial\Omega,\sigma)$ according to \eqref{Rieszpv}.

Our main result in this regard is the following theorem. For all relevant definitions, the reader is referred to 
Sections~\ref{sec:growth}, \ref{sec:holder}, and \ref{sec:gmt}.

\begin{theorem}\label{theor:uewmp}
Let $n\in{\mathbb{N}}$, $n\geq 2$ and suppose $\Omega\subset\mathbb{R}^n$ is an Ahlfors regular domain whose boundary is compact. 
Abbreviate $\sigma:={\mathcal{H}}^{n-1}\lfloor\partial\Omega$ 
and denote by $\nu$ the geometric measure theoretic outward unit normal to $\Omega$. Also, define 
$\Omega_{+}:=\Omega$ and $\Omega_{-}:={\mathbb{R}}^n\setminus\overline{\Omega}$.
Finally, let $\omega:\big(0,\diam(\partial\Omega)\big)\to(0,\infty)$ be a bounded, non-decreasing function,  
whose limit at the origin vanishes, and satisfying
\begin{equation}\label{54332e90}
\sup_{0<t<\diam(\partial\Omega)}\Bigg\{\frac{1}{\omega(t)}\Big(
\int_0^{t}\omega(s)\frac{ds}{s}+t\,\int_t^{\diam(\partial\Omega)}\frac{\omega(s)}{s}\,\frac{ds}{s}\Big)\Bigg\}<+\infty.
\end{equation} 

Then the following statements are equivalent:

\begin{list}{$(\theenumi)$}{\usecounter{enumi}\leftmargin=.8cm
\labelwidth=.8cm\itemsep=0.2cm\topsep=.1cm
\renewcommand{\theenumi}{\alph{enumi}}}
\item After possibly being altered on a set of $\sigma$-measure zero, the outward unit normal 
$\nu$ to $\Omega$ belongs to the generalized H\"older space ${\mathscr{C}}^{\omega}(\partial\Omega)$
{\rm (}i.e. the set $\Omega$ is a $\mathscr{C}^{1,\omega}$-domain; cf. Remark~\ref{YrafV.b}{\rm )}.

\vskip 0.08in
\item The Riesz transforms on $\partial\Omega$ defined in \eqref{Rjdistributional1} satisfy 
\begin{equation}\label{mainthmeq3}
R_j 1\in{\mathscr{C}}^{\omega}(\partial\Omega)\,\,\text{ for each }\,\,j\in\{1,\dots,n\}.
\end{equation}

\item The set $\Omega$ is a {\rm UR} domain {\rm (}in the sense of Definition~\ref{tFCa97}{\rm )} and the principal-value 
Riesz transforms \eqref{Rieszpv} satisfy 
\begin{equation}\label{mainthmeq3-pv}
R_j 1\in{\mathscr{C}}^{\omega}(\partial\Omega)\,\,\text{ for each }\,\,j\in\{1,\dots,n\}.
\end{equation}

\item The set $\Omega$ is a {\rm UR} domain 
and given any odd homogenous polynomial $P$ of degree $\ell\geq 1$ in $\mathbb{R}^n$ the singular 
integral operator acting on each function $f\in{\mathscr{C}}^{\omega}(\partial\Omega)$ according to 
\begin{equation}\label{trdcc-765}
(Tf)(x):=\lim_{\varepsilon\to 0^{+}}\int\limits_{\substack{y\in\partial\Omega\\ |x-y|>\varepsilon}}
\frac{P(x-y)}{|x-y|^{n-1+\ell}}\,f(y)\,d\sigma(y)
\,\,\text{ for $\sigma$-a.e. }\,\,x\in\partial\Omega
\end{equation}
is well-defined and maps the generalized H\"older space ${\mathscr{C}}^{\omega}(\partial\Omega)$ boundedly into itself.

\item The set $\Omega$ is a {\rm UR} domain, and the boundary-to-domain versions of the Riesz 
transforms defined for each $j\in\{1,\dots,n\}$ and each $f\in L^1(\partial\Omega,\sigma)$ as 
\begin{equation}\label{trfad-ytR}
\big({\mathscr{R}}_j^{\pm}f\big)(x):=\frac{1}{\varpi_{n-1}}
\int_{\partial\Omega}\frac{x_j-y_j}{|x-y|^{n}}\,f(y)\,d\sigma(y),\qquad\forall\,x\in\Omega_{\pm},
\end{equation}
satisfy 
\begin{equation}\label{mainthmRinCw}
{\mathscr{R}}_j^{\pm} 1\in{\mathscr{C}}^{\omega}(\Omega_{\pm})\,\,\text{ for each }\,\,j\in\{1,\dots,n\}.
\end{equation}

\item The set $\Omega$ is a {\rm UR} domain, and given any odd homogenous polynomial $P$ of degree $\ell\ge 1$ 
in $\mathbb{R}^n$, the integral operators acting on each function $f\in\mathscr{C}^{\omega}(\partial\Omega)$
according to 
\begin{equation}\label{defT}
\mathbb{T}_{\pm}f(x):=\int_{\partial\Omega}\frac{P(x-y)}{|x-y|^{n-1+\ell}}\,f(y)\,d\sigma(y),
\qquad\forall\,x\in\Omega_{\pm},
\end{equation}
map the generalized H\"older space 
${\mathscr{C}}^{\omega}(\partial\Omega)$ continuously into ${\mathscr{C}}^{\omega}(\Omega_{\pm})$.
\end{list}

In addition, if $\Omega$ is a ${\mathscr{C}}^{1,\omega}$-domain, there exists a constant $C\in(0,\infty)$, depending only on 
$n$, $\omega$, and $\Omega$, with the property that the operators in \eqref{defT} and \eqref{trdcc-765} satisfy
\begin{equation}\label{mainthmeq1}
\|{\mathbb{T}}_{\pm}f\|_{\mathscr{C}^{\omega}(\Omega_{\pm})}\leq C^\ell\,2^{\ell^2} 
\|P\|_{L^2(S^{n-1},{\mathcal{H}}^{n-1})}\|f\|_{{\mathscr{C}}^{\omega}(\partial\Omega)},
\qquad\forall\,f\in{\mathscr{C}}^{\omega}(\partial\Omega),
\end{equation}
and
\begin{equation}\label{mainthmeq2}
\|Tf\|_{{\mathscr{C}}^{\omega}(\partial\Omega)}\leq
C^\ell\,2^{\ell^2}\|P\|_{L^2(S^{n-1},{\mathcal{H}}^{n-1})}\|f\|_{{\mathscr{C}}^{\omega}(\partial\Omega)},
\qquad\forall\,f\in{\mathscr{C}}^{\omega}(\partial\Omega).
\end{equation}
\end{theorem}

This generalizes earlier work in \cite[Theorem~1.1, pp.\,959--960]{MitMitVer16} where similar characterizations for 
domains of class ${\mathscr{C}}^{1+\alpha}$, with $\alpha\in(0,1)$, have been obtained. 
The latter scenario presently corresponds to the particular choice $\omega(t):=t^{\alpha}$ for each $t>0$
in Theorem~\ref{theor:uewmp}. Our present work adds further credence to the heuristic principle that 
the action of the distributional Riesz transforms \eqref{Rjdistributional1} on the 
constant function $1$ encapsulates much information, both of analytic and geometric flavor, about the 
underlying Ahlfors regular domain $\Omega\subseteq{\mathbb{R}}^n$ (with compact boundary). 
At the most basic level, the main result of F.~Nazarov, X.~Tolsa, and A.~Volberg in \cite[p.\,239]{NazTolVol14} states that 
\begin{equation}\label{Mabb88}
\partial\Omega\,\text{ is a {\rm UR} set}\,\Longleftrightarrow\,
R_j 1\in{\rm BMO}(\partial\Omega,\sigma)\,\,\,\text{ for each }\,\,j\in\{1,\dots,n\},
\end{equation}
and it has been noted in \cite[Theorem~1.5, p.\,962]{MitMitVer16} that 
\begin{equation}\label{eq:iugf.6r4}
\left.
\begin{array}{r}
\nu\in{\rm VMO}(\partial\Omega,\sigma)
\\[4pt]
\text{and $\partial\Omega$ is a {\rm UR} set}
\end{array}
\right\}
\Longleftrightarrow R_j1\in{\rm VMO}(\partial\Omega,\sigma)\,\,\text{ for all }\,\,j\in\{1,\dots,n\},
\end{equation}
where ${\rm VMO}(\partial\Omega,\sigma)$ stands for the Sarason space of functions with 
vanishing mean oscillation on $\partial\Omega$, with respect to the measure $\sigma$.
By further assigning additional regularity to the functionals $\{R_j1\}_{1\leq j\leq n}$ one obtains the 
following result (proved in \cite[Theorem~1.1, pp.\,959--960]{MitMitVer16}) 
\begin{equation}\label{eq:iugf.6r4.xxx}
\left.
\begin{array}{r}
\text{$\Omega$ is a domain}
\\[4pt]
\text{of class ${\mathscr{C}}^{1+\alpha}$}
\end{array}
\right\}
\Longleftrightarrow R_j1\in{\mathscr{C}}^{\alpha}(\partial\Omega)\,\,\text{ for all }\,\,j\in\{1,\dots,n\},
\end{equation}
where $\alpha\in(0,1)$ and ${\mathscr{C}}^{\alpha}(\partial\Omega)$ is the classical H\"older 
space of order $\alpha$ on $\partial\Omega$. 

Theorem~\ref{theor:uewmp} provides a satisfactory generalization of \eqref{eq:iugf.6r4.xxx} by allowing 
considerably more flexible scales of spaces measuring H\"older regularity (see the discussions in 
Examples~\ref{trfdfd-1}-\ref{trfdfd-3} in this regard). To place this in perspective, observe that the operators described 
in \eqref{trdcc-765} may be thought of as generalized Riesz transforms since they correspond to 
\eqref{Rieszpv} in the case when 
\begin{equation}\label{eq:Pjah}
P(x):=x_j/\varpi_{n-1}\,\,\text{ for }\,\,x=(x_1,\dots,x_n)\in{\mathbb{R}}^n,\quad 1\leq j\leq n.
\end{equation}
As such, our result may be interpreted as roughly saying that 
{\it the classical Riesz transforms are bounded on a generalized H\"older space
if and only if all generalized Riesz transforms are bounded on a generalized H\"older space
if and only if the underlying domain is Lyapunov}. 

Owing to their significant role in partial differential equations, Lyapunov domains have received 
considerable attention in the literature, leading to the development of several alternative characterizations. 
It is therefore natural to compare Theorem~\ref{theor:uewmp} with other geometric characterizations of Lyapunov domains 
that have been established. For example, for Jordan planar domains, E.~Casey has identified in \cite[Theorem~1]{Casey2024} 
a sufficient condition in terms of the rate of vanishing of the $\varepsilon$-Carleson function. 
Theorem~\ref{theor:uewmp} should also be compared with the following 
purely geometric characterization of Lyapunov domains obtained in \cite[Theorem~1.3, p.\,285]{ABMMZ}, which amounts to 
the ability of threading the boundary of the said domain in between the two rounded components of 
an ``hour-glass'' shaped configuration.

\begin{proposition}\label{MBV-33}
Fix $D\in(0,\infty)$ and let $\omega:(0,D]\to[0,\infty)$ be a continuous, strictly increasing function,  
with the property that 
\begin{equation}\label{OmOmOmX2}
\lim_{\lambda\to 0^{+}}\Big(\,\,\sup_{t\in(0,\min\{D,D/\lambda\}]}
\frac{\omega(\lambda\,t)}{\omega(t)}\Big)=0.
\end{equation}
Define the pseudo-ball associated with $\omega$ having apex at a point $x\in{\mathbb{R}}^n$, 
axis of symmetry along some vector $h\in S^{n-1}$, height $b>0$, and aperture $a>0$, as the set
\begin{equation}\label{DCC-PK.2}
\mathscr{G}^{\omega}_{a,b}(x,h):=\big\{y\in B(x,D)\setminus\{x\}:\,a|y-x|\,\omega(|y-x|)< h\cdot(y-x)<b\big\}.
\end{equation}

Then a given nonempty, open, proper subset $\Omega$ of ${\mathbb{R}}^n$, with compact boundary 
is a $\mathscr{C}^{1,\omega}$-domain if and only if there exist $a>0$, $b>0$ and a function 
$h:\partial\Omega\to S^{n-1}$ with the property that 
\begin{equation}\label{Bbvbv-44}
\mathscr{G}^{\omega}_{a,b}(x,h(x))\subseteq\Omega\,\,\text{ and }\,\,
\mathscr{G}^{\omega}_{a,b}(x,-h(x))\subseteq{\mathbb{R}}^n\setminus\Omega
\,\,\text{ for each }\,\,x\in\partial\Omega.
\end{equation}
\end{proposition}

\noindent In \cite[Theorem~4.4, pp.\,330--331]{ABMMZ} this was used to prove a sharper version of the Hopf-Oleinik Boundary
Point Principle in domains satisfying a ``pseudo-ball condition'' (in place of the classical interior ball condition).


Theorem~\ref{theor:uewmp} should be contrasted with the following result, which specifically focuses on the issue 
of boundedness of singular integral operators (SIO) on generalized H\"older spaces. The class of SIO's presently considered 
has been introduced in \cite[Chapter~5, pp.\,531--672]{GHA.IV}, under the label ``generalized double layers.''
The key feature of these operators is that the kernel is given by the inner product of the outward unit
normal $\nu$ with a divergence-free vector-valued kernel.
This class of operators includes the Cauchy–Clifford operator (see Section~\ref{SEc:5rff})
and the harmonic double layer potential operator, among others.
For the proof, the reader is referred to Section~\ref{KJB.D+M}.

\begin{theorem}\label{theor:uewmp.222}
Let $n\in{\mathbb{N}}$, $n\geq 2$ and suppose $\Omega\subset\mathbb{R}^n$ is an Ahlfors regular domain with compact boundary. 
Abbreviate $\sigma:={\mathcal{H}}^{n-1}\lfloor\partial\Omega$ and denote by $\nu$ the geometric measure theoretic outward unit 
normal to $\Omega$. Let $\omega:\big(0,\diam(\partial\Omega)\big)\to(0,\infty)$ be a bounded, doubling, non-decreasing, Dini function, 
whose limit at the origin vanishes. Associate with such a Dini growth function $\omega$ its Zygmund pair, 
$\omega_Z:\big(0,\diam(\partial\Omega)\big)\to(0,\infty)$ defined as
\begin{equation}\label{omega-cond:main.222b}
\omega_Z(t):=\int_0^{t}\omega(s)\frac{ds}{s}+t\,\int_t^{\diam(\partial\Omega)}\frac{\omega(s)}{s}\,\frac{ds}{s}
\,\,\text{ for each }\,\,t\in\big(0,\diam(\partial\Omega)\big).
\end{equation}
Finally, consider a vector-valued function 
\begin{align}\label{GDL-TH.it.000}
\begin{array}{c}
\vec{k}=(k_j)_{1\leq j\leq n}\in\big[{\mathscr{C}}^2({\mathbb{R}}^n\setminus\{0\})\big]^n
\\[6pt]
\text{odd, positive homogeneous of degree $1-n$,} 
\\[6pt]
\text{satisfying }\,\,{\rm div}\,\vec{k}=\sum_{j=1}^n\partial_jk_j=0\,\,\text{ in }\,\,{\mathbb{R}}^n\setminus\{0\}.
\end{array}
\end{align}

Then the following statements are true:

\begin{enumerate}
\item[(i)] The {\rm (}boundary-to-boundary{\rm )} generalized double layer potential operator
acts in a meaningful fashion on each function $f\in{\mathscr{C}}^{\omega}(\partial\Omega)$ according to 
\begin{align}\label{16778.bbb.222.GDL}
\qquad\quad Tf(x):=\lim_{\varepsilon\to 0^{+}}\int\limits_{\substack{y\in\partial\Omega\\ |x-y|>\varepsilon}}
\langle\nu(y),\vec{k}(x-y)\rangle f(y)\,d\sigma(y)\,\,\text{ for $\sigma$-a.e. }\,\,x\in\partial\Omega.
\end{align}
Defined as such, $Tf$ coincides outside of a $\sigma$-nullset with a function in ${\mathscr{C}}^{\omega_Z}(\partial\Omega)$.
Thus interpreted, $T$ induces a well-defined, linear, and bounded mapping 
\begin{equation}\label{Jgsag.TTT.222}
T:{\mathscr{C}}^{\omega}(\partial\Omega)\longrightarrow{\mathscr{C}}^{\omega_Z}(\partial\Omega).
\end{equation}
In particular, under the additional assumption that there exists $C\in(0,\infty)$ for which 
\begin{equation}\label{omega-cond:main.333}
\qquad\quad\int_0^{t}\omega(s)\frac{ds}{s}+t\,\int_t^{\diam(\partial\Omega)}\frac{\omega(s)}{s}\,\frac{ds}{s}
\leq C\omega(t)\,\text{ for all }\,t\in\big(0,\diam(\partial\Omega)\big),
\end{equation}
it follows that 
\begin{equation}\label{Jgsag.TTT.333}
T:{\mathscr{C}}^{\omega}(\partial\Omega)\longrightarrow{\mathscr{C}}^{\omega}(\partial\Omega).
\end{equation}

\item[(ii)] Strengthen the original geometric hypotheses by also assuming that $\Omega$ is a uniform domain
(cf. Definition~\ref{tFCa-u6fr}). 
Then the {\rm (}boundary-to-domain{\rm )} generalized double layer potential operator, originally defined for each function 
$f\in L^1(\partial\Omega,\sigma)$ according to 
\begin{align}\label{GDL-TH.it.1}
{\mathcal{T}}f(x):=\int_{\partial\Omega}\langle\nu(y),\vec{k}(x-y)\rangle f(y)\,d\sigma(y)\,\,\text{ for all }\,\,x\in\Omega,
\end{align}
induces a well-defined, linear, and bounded mapping 
\begin{equation}\label{Jgsag.TTT}
{\mathcal{T}}:{\mathscr{C}}^{\omega}(\partial\Omega)\longrightarrow{\mathscr{C}}^{\omega_Z}(\Omega).
\end{equation}

Moreover, given any $f\in{\mathscr{C}}^{\omega}(\partial\Omega)$, under the identification
${\mathscr{C}}^{\omega_Z}(\Omega)\equiv{\mathscr{C}}^{\omega_Z}(\overline{\Omega})$, 
one may extend ${\mathcal{T}}f$ uniquely to a function in ${\mathscr{C}}^{\omega_Z}(\overline{\Omega})$.
Regarded as such, one then has the boundary trace formula
\begin{equation}\label{eq:-knbmdnjm.D+M}
({\mathcal{T}}f)\big|_{\partial\Omega}=-\tfrac{\vartheta}{2}f+Tf\,\,\text{ everywhere on }\,\,\partial\Omega,
\end{equation} 
where 
\begin{equation}\label{kjshgh.FF.D+M}
\qquad\quad\vartheta:=\int_{S^{n-1}}\langle x,\vec{k}(x)\rangle\,d\mathcal{H}^{n-1}(x)
=\sum_{j=1}^n\int_{S^{n-1}}x_jk_j(x)\,d\mathcal{H}^{n-1}(x)\in\mathbb{C}.
\end{equation}

Finally, under the additional assumption that \eqref{omega-cond:main.333} holds for some constant $C\in(0,\infty)$, it follows that 
\begin{equation}\label{Jgsag.TTT.D+M}
{\mathcal{T}}:{\mathscr{C}}^{\omega}(\partial\Omega)\longrightarrow{\mathscr{C}}^{\omega}(\Omega)
\end{equation}
is a well-defined, linear, and bounded mapping.
\end{enumerate}
\end{theorem}

The Cauchy integral operator in the plane is a prime example of an operator to which Theorem~\ref{theor:uewmp.222} applies. 
To make this transparent, work in the two-dimensional setting, and identify ${\mathbb{R}}^2\equiv{\mathbb{C}}$. 
Let $\Omega\subset\mathbb{R}^2$ be an Ahlfors regular domain with compact boundary. 
Abbreviate $\sigma:={\mathcal{H}}^{1}\lfloor\partial\Omega$ and denote by $\nu=(\nu_1,\nu_2)\equiv\nu_1+i\nu_2$ 
the geometric measure theoretic outward unit normal to $\Omega$. Also, consider  
\begin{equation}\label{jhDSc-uH.a}
\vec{k}(z):=-\frac{1}{2\pi}\Big(\frac{1}{z},\frac{i}{z}\Big)\,\,\text{ for each }\,\,z\in{\mathbb{C}}\setminus\{0\}.
\end{equation}
This is a smooth vector-valued function (with complex-valued components) which is odd, positive homogeneous of degree $-1$, and satisfies
\begin{align}\label{jhDSc-uH.b}
\big({\rm div}\,\vec{k}\,\big)(z) &=-\frac{1}{2\pi}\partial_x\Big(\frac{1}{z}\Big)
-\frac{1}{2\pi}\partial_y\Big(\frac{i}{z}\Big)
=-\frac{1}{2\pi}(\partial_x+i\,\partial_y)\Big(\frac{1}{z}\Big)
\nonumber\\[6pt]
&=-\frac{1}{\pi}\overline{\partial}\Big(\frac{1}{z}\Big)
=0\,\,\text{ for each }\,\,z\in{\mathbb{C}}\setminus\{0\},
\end{align}
where $\overline{\partial}:=\tfrac{1}{2}(\partial_{x}+i\,\partial_{y})$ is the Cauchy-Riemann operator in the plane.
Hence all hypotheses of Theorem~\ref{theor:uewmp.222} are satisfied in this case, and \eqref{kjshgh.FF.D+M} presently becomes
\begin{align}\label{KD-jmp-rr.Gttt.FD}
\vartheta &=\int_{S^{1}}\langle\xi,\vec{k}(\xi)\rangle\,d{\mathcal{H}}^{1}(\xi)
=-\frac{1}{2\pi}\int_{S^{1}}\Big\{\xi_1\Big(\frac{1}{\xi}\Big)+\xi_2\Big(\frac{i}{\xi}\Big)\Big\}\,d{\mathcal{H}}^{1}(\xi)
\nonumber\\[6pt]
&=-\frac{1}{2\pi}\int_{S^{1}}(\xi_1+i\xi_2)\Big(\frac{1}{\xi}\Big)\,d{\mathcal{H}}^{1}(\xi)
=-\frac{1}{2\pi}\int_{S^{1}}1\,d{\mathcal{H}}^{1}(\xi)=-1.
\end{align}
Since for each $z\in\Omega$ we have 
\begin{align}\label{jhDSc-uH.c}
\big\langle\nu(\zeta)\,,\,\vec{k}(z-\zeta)\big\rangle &=-\frac{1}{2\pi}\nu_1(\zeta)\Big(\frac{1}{z-\zeta}\Big)
-\frac{1}{2\pi}\nu_2(\zeta)\Big(\frac{i}{z-\zeta}\Big)
\nonumber\\[6pt]
&=-\frac{1}{2\pi}\big(\nu_1(\zeta)+i\nu_2(\zeta)\big)\Big(\frac{1}{z-\zeta}\Big)
\nonumber\\[6pt]
&=\frac{1}{2\pi}\frac{\nu(\zeta)}{\zeta-z}\,\,\text{ for $\sigma$-a.e. }\,\,\zeta\in\partial\Omega,
\end{align}
and since $i\nu(\zeta)\,d\sigma(\zeta)=d\zeta$ on $\partial\Omega$, the operator ${\mathcal{T}}$ 
constructed as in \eqref{GDL-TH.it.1} for $\vec{k}$ as in \eqref{jhDSc-uH.a} becomes precisely the boundary-to-domain 
Cauchy integral operator, acting on each function $f\in L^1(\partial\Omega,\sigma)$ according to 
\begin{equation}\label{jhDSc-uH.d}
{\mathscr{C}}f(z)=\frac{1}{2\pi i}\int_{\partial\Omega}\frac{f(\zeta)}{\zeta-z}\,d\zeta\,\,\text{ for all }\,\,z\in\Omega.
\end{equation}
Also, the principal-value singular integral operator $T$ from \eqref{16778.bbb.222.GDL} presently becomes
\begin{equation}\label{jhDSc-uH.e}
Cf(z):=\lim_{\varepsilon\to 0^{+}}\frac{1}{2\pi i}\int\limits_{\substack{\zeta\in\partial\Omega\\ |z-\zeta|>\varepsilon}}
\frac{f(\zeta)}{\zeta-z}\,d\zeta\,\,\text{ for $\sigma$-a.e. }\,\,z\in\partial\Omega,
\end{equation}
i.e., the principal-value Cauchy integral operator on $\partial\Omega$. As such, all conclusions of Theorem~\ref{theor:uewmp.222} are valid
for the operators \eqref{jhDSc-uH.d}-\eqref{jhDSc-uH.e}. In particular, \eqref{Jgsag.TTT.333} becomes the statement that the 
principal-value Cauchy singular integral operator induces a continuous mapping 
\begin{equation}\label{Jgsag.TTT.333.uuu}
C:{\mathscr{C}}^{\omega}(\partial\Omega)\longrightarrow{\mathscr{C}}^{\omega}(\partial\Omega)
\end{equation}
whenever $\Omega\subset\mathbb{R}^2\equiv{\mathbb{C}}$ is an arbitrary Ahlfors regular domain with compact boundary and 
$\omega:\big(0,\diam(\partial\Omega)\big)\to(0,\infty)$ is a bounded, doubling, non-decreasing, Dini function, 
whose limit at the origin vanishes, and which satisfies \eqref{omega-cond:main.333}. 

From this point of view, Theorem~\ref{theor:uewmp.222} may be regarded as a higher-dimensional generalization of the classical 
Plemelj-Privalov theorem regarding the boundedness of the Cauchy integral operator on standard H\"older spaces on the unit circle 
in the complex plane (cf. \cite{Plemelj1908}, \cite{Privalov1916}). Subsequent improvements focused on relaxing the geometric assumptions. 
For example, Privalov treated piecewise smooth curves without cusps in \cite{Privalov1939}
(note that the piecewise smoothness assumption and the absence of cusps guarantee that the domain encircled by said curve 
is an Ahlfors regular domain in our terminology). Additional classical contributions in this regard are due to Muskhelishvili \cite{Mush1953} 
and Davydov \cite{Dav1949}. A generalization of the Plemelj-Privalov theorem for moduli of continuity has been given by Zygmund \cite{Zyg1923} 
in the case of a circle. In \cite{BaSt1956}, Bari and Stechkin have shown that the class of moduli of continuity considered by Zygmund
is optimal in the case of a circle. Extensions to more general geometric settings have been obtained by Magnaradze \cite{Magda1947}, 
Babaev and Salaev \cite{BaSa1965}, Dyn'kin \cite{Dyn1979}, and Guseinov in \cite{Gus1992}. In particular, 
Salaev \cite{Salaev1976} has shown the optimality class of moduli of continuity introduced by Zygmund in 
the class of Ahlfors regular curves. 

Theorem~\ref{theor:uewmp.222} also applies to the {\it Clifford-Cauchy integral operator} considered on Ahlfors regular domains in ${\mathbb{R}}^n$. 
This may be thought of as a higher-dimensional version of the `standard' Cauchy operator in the plane, constructed in relation to a 
higher-dimensional (and, this time, highly noncommutative) version of the field of complex numbers, dubbed Clifford algebra. 
We shall elaborate on this topic later (see the discussion in \S\ref{SEc:5rff}). For now, we make the following observation 
paving the way towards such a generalization. 

\begin{remark}\label{KJag.GFF.D+M}
All results in Theorem~\ref{theor:uewmp.222} remain valid when the components $(k_j)_{1\leq j\leq n}$ of the vector field $\vec{k}$ are 
{\it matrix}-{\it valued} functions, i.e., when for some $M\in{\mathbb{N}}$ in place of the first line in \eqref{GDL-TH.it.000} we have  
\begin{equation}\label{GDL-TH.it.000.D+M}
\vec{k}=(k_j)_{1\leq j\leq n}\in\Big[{\mathscr{C}}^2({\mathbb{R}}^n\setminus\{0\})^{M\times M}\Big]^n,
\end{equation}
and when all generalized H\"older spaces consist of vector-valued functions, i.e., 
${\mathscr{C}}^{\omega}(\partial\Omega)$ is replaced by ${\mathscr{C}}^{\omega}(\partial\Omega)^M$, et cetera
{\rm (}note that in this scenario, $\vartheta$ defined as in \eqref{kjshgh.FF.D+M} now belongs to $\mathbb{C}^{M\times M}${\rm )}.
Indeed, this ``matrix-vector'' version of Theorem~\ref{theor:uewmp.222} is readily obtained by applying the current scalar version to each vector field 
\begin{equation}\label{GDL-TH.it.000.D+M.2}
\vec{k}^{\alpha\beta}:=(k_j^{\alpha\beta})_{1\leq j\leq n}\in\big[{\mathscr{C}}^2({\mathbb{R}}^n\setminus\{0\})\big]^n,
\,\,\text{ with }\,\,1\leq\alpha,\beta\leq M,
\end{equation}
where each $k_j^{\alpha\beta}$ denotes the $(\alpha,\beta)$-entry in the ${\mathbb{C}}^{M\times M}$-valued function $k_j$.
\end{remark}

Finally, we note that Theorem~\ref{theor:uewmp.222}  is applicable to double layer potential operators associated with 
second-order elliptic systems of partial differential operators (of the sort considered in \cite{GHA.IV}-\cite{GHA.V}), 
and this opens the door for considering boundary value problems with data in generalized H\"older spaces. 

The layout of the paper is as follows. In \S\ref{sec:growth} we introduce and study 
growth functions and we provide an example illustrating that generalized H\"older 
spaces may not be contained within the scale of classical H\"older spaces.
Growth functions are then used to define generalized H\"{o}lder spaces in \S\ref{sec:holder}.
Essentially, \S\ref{sec:gmt} amounts to a compendium on the subject of geometric measure theory.
We then proceed to study singular integrals on generalized H\"older spaces in \S\ref{Sec:6ff}.
The tools presented up to that point are employed to prove Theorem~\ref{theor:uewmp.222} in \S\ref{KJB.D+M}.
In \S\ref{SEc:5rff}, we develop the Clifford algebra framework to specialize
Theorem~\ref{theor:uewmp.222} to Cauchy–Clifford operators.
Finally, \S\ref{treED} contains the proof of our main result, Theorem~\ref{theor:uewmp}, 
which proceeds by induction on the degree $\ell$ of the polynomial $P$ involved
in \eqref{trdcc-765} and \eqref{defT}.
Following \cite[pp.\,998-1005]{MitMitVer16}, our approach relies on the Clifford algebra techniques mentioned earlier.

\section{Growth functions}\label{sec:growth}

\begin{definition}\label{gf-Def-1}
Given a number $D\in(0,\infty)$, a function $\omega:(0,D)\to(0,\infty)$ is called a {\tt growth} {\tt function} 
{\tt on} $(0,D)$ if $\omega$ is non-decreasing and
\begin{equation}\label{eq:wD}
\lim_{t\to 0^{+}}\omega(t)=0,\qquad\omega(D):=\lim_{t\to D^{-}}\omega(t)<\infty.
\end{equation}
Corresponding to the case when $D=\infty$, a function $\omega:(0,\infty)\to(0,\infty)$ 
is called a {\tt growth} {\tt function} {\tt on} $(0,\infty)$ if $\omega$ is non-decreasing and 
\begin{equation}\label{terdds-tr}
\lim_{t\to 0^{+}}\omega(t)= 0.
\end{equation}

If $\omega$ is a growth function on $(0,D)$ with $D\in(0,\infty]$, call $\omega$ {\tt doubling} provided 
\begin{equation}\label{wnonin-hGFF}
C_\omega^{\rm dou}:=
\sup_{0<t<D/2}\frac{\omega(2t)}{\omega(t)}<\infty,
\end{equation}
and refer to $C_\omega^{\rm dou}$ as the doubling constant of $\omega$. 
Also, call $\omega$ {\tt Dini} provided 
\begin{equation}\label{wnonin-hGFF.DDD}
\int_0^D\frac{\omega(t)}{\max\{t,t^2\}}\,dt<\infty.
\end{equation}
\end{definition}

It turns out that any growth function on an interval extends to a growth function on $(0,\infty)$. 
For ease of reference, we state this formally in the remark below. 

\begin{remark}\label{rem:wextension}
Let $D\in(0,\infty]$. If $\omega$ is a growth function on $(0,D)$ then 
\begin{equation}\label{om-EXT}
\widetilde{\omega}(t):=\omega\big(\min\{t,D\}\big)\,\,\text{ for each }\,\,t\in(0,\infty)
\end{equation}
is a growth function on $(0,\infty)$ with the property that $\widetilde{\omega}=\omega$ on $(0,D)$. 
Also, it is immediate from definitions that 
\begin{equation}\label{wnonin-hGFF.cz}
\parbox{9.00cm}{if $\omega$ is a doubling growth function on $(0,D)$, with $D\in(0,\infty]$, 
then $\widetilde{\omega}$ is a doubling growth function on the interval $(0,\infty)$
satisfying $C_{\widetilde{\omega}}^{\rm dou}=C_\omega^{\rm dou}$,}
\end{equation}
and
\begin{equation}\label{wnonin-hGFF.cz.EF}
\parbox{10.30cm}{if $\omega$ is a Dini growth function on $(0,D)$, with $D\in(0,\infty]$, 
then $\widetilde{\omega}$ is a Dini growth function on the interval $(0,\infty)$.}
\end{equation}
\end{remark}

Following Zygmund, with any given Dini growth function $\omega:(0,D)\to(0,\infty)$ with $D\in(0,\infty]$
we shall associate 
\begin{equation}\label{omega-cond:main}
\omega_Z(t):=\int_0^{t}\omega(s)\frac{ds}{s}+t\,\int_t^{D}\frac{\omega(s)}{s}\,\frac{ds}{s}
\,\,\text{ for each }\,\,t\in(0,D).
\end{equation}
For this we have:
\begin{equation}\label{wnonin-hGFF.cz.UU}
\parbox{11.00cm}{if $\omega$ is some Dini growth function on $(0,D)$, with $D\in(0,\infty]$, 
then $\omega_Z:(0,D)\to(0,\infty)$ is a growth function on $(0,D)$; 
when $\omega$ is also doubling then $\omega_Z$ is doubling and there 
exists $c\in(0,\infty)$ which depends only on the doubling constant of $\omega$ with the 
property that $\omega_Z(t)\geq c\,\omega(t)$ for each $t\in(0,D)$.}
\end{equation}
Indeed, the fact that $\omega$ is Dini implies that $\omega_Z:(0,D)\to(0,\infty)$ is well-defined, 
and whenever $0<t_1<t_2<D$ we may estimate 
\begin{equation}\label{omega-cond:main.A}
\omega_Z(t_2)-\omega_Z(t_1)\geq(t_2-t_1)\int_{t_2}^D\frac{\omega(s)}{s}\,\frac{ds}{s}\geq 0,
\end{equation}
which shows that $\omega_Z$ is strictly increasing. Also, for each $t\in(0,D/2)$ we may write
(making a change of variables and using the fact that $\omega$ is doubling) 
\begin{align}\label{omega-cond:main.B}
\omega_Z(2t) &=\int_0^{2t}\omega(s)\frac{ds}{s}+2t\,\int_{2t}^{D}\frac{\omega(s)}{s}\,\frac{ds}{s}
\nonumber\\[6pt]
&\leq\int_0^{t}\omega(2\tau)\frac{d\tau}{\tau}+2t\,\int_{t}^{D}\frac{\omega(s)}{s}\,\frac{ds}{s}
\nonumber\\[6pt]
&\leq C_\omega^{\rm dou}\int_0^{t}\omega(s)\frac{ds}{s}+2t\,\int_{t}^{D}\frac{\omega(s)}{s}\,\frac{ds}{s}
\leq\max\big\{2,C_\omega^{\rm dou}\big\}\omega_Z(t),
\end{align}
proving that $\omega_Z$ is doubling. Furthermore, when $D/2\leq t<D$ the fact that $\omega$ is non-decreasing and doubling permits us to estimate 
\begin{align}\label{dewdead-ttt.xxx.2.aaa}
\omega(t) &\leq\omega(D)\leq(C_\omega^{\rm dou})^2\omega(D/4)
\leq\frac{(C_\omega^{\rm dou})^2}{\ln 2}\int_{D/4}^{D/2}\omega(s)\,\frac{ds}{s}
\nonumber\\[6pt]
&\leq\frac{(C_\omega^{\rm dou})^2}{\ln 2}\int_0^{t}\omega(s)\,\frac{ds}{s}
\leq\frac{(C_\omega^{\rm dou})^2}{\ln 2}\omega_Z(t).
\end{align}
Finally, when $0<t<D/2$ the fact that $\omega$ is non-decreasing allows us to estimate 
\begin{align}\label{dewdead-ttt.xxx.3.bbb}
\frac{1}{2}\omega(t) &\leq\omega(t)\Big(1-\frac{t}{D}\Big)=t\omega(t)\Big(\frac{1}{t}-\frac{1}{D}\Big)
\nonumber\\[6pt]
&=t\omega(t)\int_{t}^D\frac{ds}{s^2}
\leq t\int_{t}^D\frac{\omega(s)}{s}\,\frac{ds}{s}
\leq\omega_Z(t),
\end{align}
and this completes the proof of \eqref{wnonin-hGFF.cz.UU}. 

Moving on, another assumption (which plays a natural role in a variety of contexts) imposed on a given 
Dini growth function $\omega:(0,D)\to(0,\infty)$, with $D\in(0,\infty]$, is that  
\begin{equation}\label{omega-cond:main.111}
C_\omega^{\rm Zyg}:=\sup_{0<t<D}\frac{\omega_Z(t)}{\omega(t)}<+\infty.
\end{equation}
Whenever $\omega:(0,D)\to(0,\infty)$, with $D\in(0,\infty]$, is a Dini growth function for which \eqref{omega-cond:main.111} holds, 
we shall say that $\omega$ is a {\tt Zygmund} {\tt modulus} {\tt of} {\tt continuity}.

Hence, a Dini growth function $\omega:(0,D)\to(0,\infty)$, with $D\in(0,\infty]$, is a Zygmund modulus of continuity if and only if 
there exists $C\in(0,\infty)$ with the property that
\begin{equation}\label{omega-cond:main.222}
\int_0^{t}\omega(s)\frac{ds}{s}+t\,\int_t^{D}\frac{\omega(s)}{s}\,\frac{ds}{s}\leq C\omega(t)\,\,\text{ for each }\,\,t\in(0,D),
\end{equation}
and the best constant doing the job in \eqref{omega-cond:main.222} is $C_\omega^{\rm Zyg}$. We remark that when $D<\infty$ the
inequality in \eqref{omega-cond:main.222} extends (simply by passing to limit) to $t=D$, if $\omega(D)$ is interpreted as in \eqref{eq:wD}. 

Let us note that for any Dini growth function $\omega$ on $(0,D)$, with $D\in(0,\infty]$, the demand in 
\eqref{omega-cond:main.111} implies that $\omega$ is doubling. Indeed, for each $t\in(0,D/4)$ we may write 
\begin{align}\label{omega-cond:main.3b}
(\ln 2)\omega(2t)\leq\int_{2t}^{4t}\omega(s)\frac{ds}{s}
\leq 4t\int_{t}^{D}\frac{\omega(s)}{s}\frac{ds}{s}\leq 4C_\omega^{\rm Zyg}\cdot\omega(t)
\end{align}
which already proves that $\omega$ is doubling in the case when $D=\infty$, whereas if $D<\infty$ then 
for each $t\in(D/4,D/2)$ we have $\omega(2t)\leq C\omega(t)$ with $C:=\omega(D)/\omega(D/4)$.

\begin{remark}\label{rem-jb}
Condition \eqref{omega-cond:main.111} is closely related to the dilation indices of Orlicz spaces, which are useful in the 
theory of interpolation of Orlicz spaces (cf. \cite[Chapter~4.8, pp.\,265-280]{BenSha88}). 
Given a growth function $\omega:(0,D)\to(0,\infty)$, where $D\in(0,\infty]$, we set 
\begin{equation}\label{ytrF}
h_{\omega}(s):=\sup_{0<t<\min\{D,D/s\}}\frac{\omega(ts)}{\omega(t)}\,\,\text{ for each }\,\,s\in(0,\infty),
\end{equation}
and define the lower and upper dilation indices of $\omega$, respectively, as
\begin{equation}\label{OrliczIndices}
i_{\omega}:=\lim_{s\to 0^{+}}\frac{\ln h_{\omega}(s)}{\ln s}=\sup_{0<s<1}\frac{\ln h_{\omega}(s)}{\ln s},
\qquad 
I_{\omega}:=\lim_{s\to\infty}\frac{\ln h_{\omega}(s)}{\ln s}=\inf_{s>1}\frac{\ln h_{\omega}(s)}{\ln s}. 
\end{equation} 
Note that for every $r,s\in(0,\infty)$ and $0<t<\min\{D,D/s,D/sr\}$ we have
\begin{equation}\label{eq:dilation-tsr}
\frac{\omega(tsr)}{\omega(t)}=\frac{\omega(tsr)}{\omega(ts)} \cdot \frac{\omega(ts)}{\omega(t)}\leq h_\omega(r)h_\omega(s).
\end{equation}
We can assume that \eqref{eq:dilation-tsr} actually holds for each
$0<t<\min\{D,D/sr\}$. Indeed, the case $r\geq 1$ or $s\leq 1$ is straightforward and otherwise
we interchange the roles of $r$ and $s$. Taking the supremum over $t\in\big(0,\min\{D,D/sr\}\big)$, it follows that 
$h_\omega(rs)\leq h_\omega(r)h_\omega(s)$, that is, $h_\omega$ is submultiplicative.
Then, whenever $h_\omega(s)$ is finite for $s\in(0,\infty)$, the equalities in 
\eqref{OrliczIndices} and the existence of the limits therein are immediate 
consequences of Fekete’s lemma for real functions (cf. \cite[Theorem~7.6.2, p.\,244]{HilPhi57}),
applied to the finite subadditive function $f(t)=\ln (h_\omega(e^t))$. 

Next, we will check that 
\begin{equation}\label{eq:ZZ-dd}
\text{\eqref{omega-cond:main.111} holds whenever $0<i_\omega\leq I_\omega<1$}. 
\end{equation}
Indeed, from \eqref{OrliczIndices} we see that if $0<\varepsilon<\min\{i_\omega,1-I_\omega\}$ then
there exists some $C_\varepsilon\in(0,\infty)$ such that $h_\omega(s)\leq C_\varepsilon s^{i_\omega-\varepsilon}$ 
for every $s\in(0,1)$, and $h_\omega(s)\leq C_\varepsilon s^{I_\omega+\varepsilon}$ for every $s\in(1,\infty)$. 
Hence, in view of these estimates and \eqref{ytrF}, we conclude that there there exists $C\in(0,\infty)$ such that if $t\in(0,D)$ then
\begin{equation}\label{trfdF-yG}
i_\omega>0\,\Rightarrow\,\int_0^t\frac{\omega(s)}{\omega(t)}\frac{ds}{s}=\int_0^1\frac{\omega(ts)}{\omega(t)}\frac{ds}{s} 
\leq C_\varepsilon\int_0^1 s^{i_\omega-\varepsilon}\frac{ds}{s}\leq C,
\end{equation}
\begin{equation}\label{yffc-yr}
I_\omega<1\,\Rightarrow\,t\int_t^D\frac{\omega(s)}{s\omega(t)}\frac{ds}{s}
=\int_1^{D/t}\frac{\omega(ts)}{s\omega(t)}\frac{ds}{s} 
\leq C_\varepsilon\int_1^{\infty}s^{I_\omega+\varepsilon-1}\frac{ds}{s}\leq C.
\end{equation}
This finishes the proof of \eqref{eq:ZZ-dd}. 

Let us note that the triple inequality in \eqref{eq:ZZ-dd} is a condition often imposed in relation to the Young function 
employed in the definition of Orlicz spaces in order to ensure that these enjoy a wealth of desirable properties
(such as the boundedness of the Hardy-Littlewood operator and modular extrapolation theorems of the sort discussed in 
\cite[Theorem~5.3.18, pp.\,206-287]{GHA.II}, \cite[Chapter 4]{CUMP}, \cite{CMM22}). 

As such, the tables may be turned and one can use these reasonable Young functions 
as a source of producing examples of growth functions satisfying \eqref{omega-cond:main.111}. Specifically, 
if $\Phi$ is a Young function and $\omega:=\Phi^{-1}$ then $i_\omega$ and $I_\omega$ become, respectively, the lower and 
upper Boyd indices of $\Phi$ (see \cite[Theorem~8.18, p.\,277]{BenSha88}). In summary, 
\begin{equation}\label{eq:ytfr-DM}
\parbox{10.60cm}{whenever $\Phi$ is a Young function whose Boyd indices belong to $(0,1)$
it follows that $\omega:=\Phi^{-1}$ is a Zygmund modulus of continuity (i.e., 
a growth function satisfying \eqref{omega-cond:main.111}).}
\end{equation}
\end{remark}

Thus, \eqref{eq:ytfr-DM} becomes an excellent source for producing relevant examples of growth functions satisfying \eqref{omega-cond:main.111}.

\begin{example}\label{trfdfd-1}
Fix $D\in(0,\infty]$. Given $\alpha\in(0,1)$, if $\omega(t):=t^\alpha$ for each $t\in(0,D)$ then 
$i_\omega=I_\omega=\alpha$, hence \eqref{omega-cond:main.111} holds. The particular version of 
Theorem~\ref{theor:uewmp} corresponding to this scenario has been established in \cite[Theorem~1.1, pp.\,959--960]{MitMitVer16}.
\end{example}

There are many examples of interest that are treated here for the first time.

\begin{example}\label{trfdfd-2}
Fix $D\in(0,\infty]$ and an arbitrary $\alpha\in(0,1)$ along with $\theta\in\mathbb{R}$. For each $t\in(0,\infty)$, 
define $\log_{+}t:=\max\{0,\ln t\}$. Then such examples include  
\begin{equation}\label{yrF.1}
\omega(t):=t^\alpha\,(A+\log_{+}t)^\theta\,\,\text{ for all }\,\,t\in(0,D),\,\,\text{ where }\,\,
A:=\max\{1,-\theta/\alpha\},
\end{equation}
\begin{equation}\label{yrF.2}
\omega(t):=t^\alpha\,(B+\log_{+}(1/t))^\theta\,\,\text{ for all }\,\,t\in(0,D),\,\,\text{ where }\,\,
B:=\max\{1,\theta/\alpha\},
\end{equation}
and, under the additional assumption that $D<\infty$,
\begin{equation}\label{yrF.5}
\omega(t):=t^\alpha\,(B+\ln(D/t))^\theta\,\,\text{ for all }\,\,t\in(0,D),\,\,\text{ where }\,\,
B:=\max\{1,\theta/\alpha\}.
\end{equation}
It is easy to verify that these functions are indeed growth functions.
In these situations $i_\omega=I_\omega=\alpha$ which, as noted earlier, guarantees that \eqref{omega-cond:main.111} holds.
For completeness, we explicitly compute the dilation indices for \eqref{yrF.5}, 
noting that the computations for the remaining cases are analogous.
In this case, for every $t,s\in(0,\infty)$, we have
\begin{equation}\label{eq:dilation-yrF.5-1}
\frac{\omega(ts)}{\omega(t)}=s^{\alpha}\left(\frac{B+\ln(D/ts)}{B+\ln(D/t)}\right)^{\theta}
=s^{\alpha}\left(1+\frac{\ln(1/s)}{B+\ln(D/t)}\right)^{\theta}.
\end{equation}
Assume first that $\theta\geq 0$. Then
\begin{align}\label{eq:dilation-yrF.5-2}
i_\omega&=\alpha+\lim_{s\to 0^+}\frac{\ln \left[\displaystyle\sup_{0<t<\min\{D,D/s\}}
\left(1+\frac{\ln(1/s)}{B+\ln(D/t)}\right)^{\theta}\right]}{\ln s}
\nonumber\\[4pt]
&=\alpha+\theta\lim_{s\to 0^+}\frac{\ln \left[
1+\frac{\ln(1/s)}{B} \right]}{\ln s}=\alpha,
\end{align}
and similarly,
\begin{equation}\label{eq:dilation-yrF.5-3}
I_\omega=\alpha+\theta\lim_{s\to\infty}\frac{\ln \left[
1+\frac{\ln(1/s)}{B+\ln s} \right]}{\ln s}=\alpha.
\end{equation}
On the other hand, if $\theta<0$, one has
\begin{align}\label{eq:dilation-yrF.5-4}
i_\omega=\alpha+\lim_{s\to 0^+}\frac{\ln \left[\displaystyle\lim_{t\to 0^+}
\left(1+\frac{\ln(1/s)}{B+\ln(D/t)}\right)^{\theta}\right]}{\ln s}=\alpha,
\end{align}
and
\begin{equation}\label{eq:dilation-yrF.5-5}
I_\omega=\alpha+\lim_{s\to\infty}\frac{\ln \left[\displaystyle\lim_{t\to 0^+}
\left(1+\frac{\ln(1/s)}{B+\ln(D/t)}\right)^{\theta}\right]}{\ln s}=\alpha.
\end{equation}
\end{example}

\begin{example}\label{trfdfd-3}
Another relevant example is offered by 
\begin{equation}\label{yrF.3}
\omega(t):=\max\{t^\alpha,t^\beta\}\,\,\text{ for all }\,\,t\in(0,D),\,\,\text{ where }\,\,0<\alpha<\beta<1.
\end{equation}
Assume first that $D=\infty$. In this case, we have
\begin{equation}
h_\omega(s)=
\max\bigg\{ \sup_{0<t\leq 1}s^\alpha,\,\,
\sup_{1<t\leq 1/s}t^{\alpha-\beta} s^\alpha,\,\,
\sup_{t> 1/s}s^{\beta}
\bigg\}=s^\alpha,
\quad\text{for every }s\in(0,1),
\end{equation}
and
\begin{equation}
h_\omega(s)=
\max\bigg\{\sup_{0<t\leq 1/s} s^\alpha,\,\,
\sup_{1/s<t\leq 1} t^{\beta-\alpha}s^\beta,\,\,
\sup_{t>1}s^\beta\bigg\}=s^\beta,
\quad\text{for every }s\in(1,\infty),
\end{equation}
so that $i_\omega=\alpha$ and $I_\omega=\beta$.
However, if we assume that $D<\infty$, then for every $s\in(0,\min\{1,1/D\})$ we have $D<1/s$ and hence
\begin{equation}
h_\omega(s)= \sup_{0<t<D}\frac{t^\alpha s^\alpha}{\omega(t)}=
\max\bigg\{ \sup_{\substack{0<t\leq 1 \\[1pt] t<D}}s^\alpha,\,\,
\sup_{1<t< D}t^{\alpha-\beta} s^\alpha\bigg\}=s^\alpha,
\end{equation}
where we agree that the supremum over an empty set is $0$.
If $s\in(\max\{1,D\},\infty)$ then $D/s<1$ and
\begin{equation}
h_\omega(s)= \sup_{0<t<D/s}\frac{\omega(ts)}{t^{\alpha}}
\max\bigg\{ \sup_{\substack{0<t\leq 1/s \\[1pt] t<D/s}}s^\alpha,\,\,
\sup_{1/s<t<D/s}t^{\beta-\alpha} s^\beta\bigg\}=\max\{1,D^{\beta-\alpha}\}s^\alpha.
\end{equation}
We conclude that $i_\omega=I_\omega=\alpha$ in this case.
A similar example is given by
\begin{equation}\label{yrF.4}
\omega(t):=\min\{t^\alpha,t^\beta\}\,\,\text{ for all }\,\,t\in(0,D),\,\,\text{ where }\,\,0<\alpha<\beta<1.
\end{equation}
Following the same reasoning as above, one readily sees that 
$i_\omega=\alpha$ and $I_\omega=\beta$ if $D=\infty$ and 
$i_\omega=I_\omega=\beta$ if $D<\infty$. Accordingly, in both examples, condition 
\eqref{omega-cond:main.111} is verified.
\end{example}

Given a growth function $\omega$, in the lemma below we study some of the implications of \eqref{omega-cond:main.111} 
for the function $\widetilde{\omega}$ associated with $\omega$ as in Remark~\ref{rem:wextension}.

\begin{lemma}\label{omega-cond:extensionlemma}
Suppose $D\in(0,\infty]$ and let $\omega$ be a growth function on $(0,D)$ satisfying \eqref{omega-cond:main.111}. 
If $\widetilde{\omega}$ is associated with $\omega$ as in Remark~\ref{rem:wextension}, then the following 
statements are true.

\begin{enumerate}
\item For each $N\in[1,\infty)$ one has
\begin{equation}\label{eq:extensionlemma-a}
\int_0^{t}\widetilde{\omega}(s)\frac{ds}{s}\leq C_1\widetilde{\omega}(t)\,\,\text{ for all }\,\,t\in(0,ND),
\end{equation}
with $C_1:=C_\omega^{\rm Zyg}+\ln N$ if $D<\infty$ and $C_1:=C_\omega^{\rm Zyg}$ if $D=\infty$.

\vskip 0.08in
\item One has
\begin{equation}\label{omega-cond:extensionlemma-b}
t\int_t^{\infty}\frac{\widetilde{\omega}(s)}{s}\frac{ds}{s}\leq C_2\widetilde{\omega}(t)\,\,\text{ for all }\,\,t\in(0,\infty),
\end{equation}
with $C_2:=C_\omega^{\rm Zyg}+\max\{1,C_\omega^{\rm Zyg}\}\frac{\omega(D)}{\omega(D/2)}$ if $D<\infty$ and $C_2:=C_\omega^{\rm Zyg}$ if $D=\infty$.

\vskip 0.08in
\item As a consequence of Remark~\ref{rem:wextension} and {\it (a)}-{\it (b)}, one concludes that 
$\widetilde{\omega}$ is a growth function on $(0,\infty)$ which satisfies \eqref{omega-cond:main.111} with 
\begin{equation}\label{omega-cond:main.111.tilde}
C_{\widetilde{\omega}}^{\rm Zyg}\leq C_1+C_2,
\end{equation}
where $C_1,C_2\in(0,\infty)$ are as above. 

\vskip 0.08in
\item Whenever $0<t_1\leq t_2<\infty$ one has
\begin{equation}\label{eq:extensionlemma-c}
\frac{\widetilde{\omega}(t_2)}{t_2}\leq C_2\frac{\widetilde{\omega}(t_1)}{t_1},
\end{equation} 
with $C_2$ as in part {\it (b)}. In particular, $\widetilde{\omega}$ is doubling with constant 
\begin{equation}\label{wnonincreasing}
C_{\widetilde{\omega}}^{\rm dou}=
\sup_{0<t<\infty}\frac{\widetilde{\omega}(2t)}{\widetilde{\omega}(t)}\leq 2C_2. 
\end{equation}
\end{enumerate}
\end{lemma}

\begin{proof}
If $D=\infty$, then the claims in $(a)$ and $(b)$ are direct consequences of \eqref{omega-cond:main.222}, since 
$\widetilde{\omega}=\omega$ on $(0,\infty)$ in this case. Assume next that $D<\infty$. In such a scenario, if 
$t\in(0,D)$ then the claim in $(a)$ follows at once from \eqref{omega-cond:main.222} since $\widetilde{\omega}=\omega$ 
on $(0,D)$. On the other hand, if $D\leq t<ND$ for some $N\in[1,\infty)$, then
\begin{align}\label{6r33r}
\int_0^{t}\widetilde{\omega}(s)\frac{ds}{s} 
&=\int_0^{D}\omega(s)\frac{ds}{s}+\int_D^{ND}\omega(D)\frac{ds}{s}\leq C_\omega^{\rm Zyg}\omega(D)+\omega(D)\ln N 
\nonumber\\[4pt] 
&=(C_\omega^{\rm Zyg}+\ln N)\widetilde{\omega}(t),
\end{align}
keeping in mind that, as noted earlier, \eqref{omega-cond:main.222} extends to $t=D$ in this case. 

As regards item $(b)$, we first claim that 
\begin{equation}\label{eq:claimlem}
t\leq\frac{\max\{1,C_\omega^{\rm Zyg}\}\,D}{\omega(D/2)}\,\omega(t)\,\,\text{ for each }\,\,t\in(0,D).
\end{equation}
Indeed, if $0<t\leq D/2$ we may estimate 
\begin{equation}\label{y5d3ssD}
t\frac{\omega(D/2)}{D}\leq t\int_{D/2}^D\frac{\omega(s)}{s}\frac{ds}{s} 
\leq t\int_t^D\frac{\omega(s)}{s}\frac{ds}{s}\leq C_\omega^{\rm Zyg}\omega(t),
\end{equation}
which suits our purposes. If $D/2<t<D$, the fact hat $\omega$ is a non-decreasing function entails 
$\omega(D/2)/D\leq\omega(t)/t$, finishing the proof of \eqref{eq:claimlem}. In turn, on account of 
\eqref{eq:claimlem} we see that for each $t\in(0,D)$ we have
\begin{align}\label{t6fDXy}
t\int_t^{\infty}\frac{\widetilde{\omega}(s)}{s}\frac{ds}{s} 
&=t\int_t^{D}\frac{\omega(s)}{s}\frac{ds}{s}+t\int_D^{\infty}\frac{\omega(D)}{s}\frac{ds}{s} 
\leq C_\omega^{\rm Zyg}\omega(t)+t\frac{\omega(D)}{D} 
\nonumber\\[4pt] 
&\leq\Big(C_\omega^{\rm Zyg}+\max\{1,C_\omega^{\rm Zyg}\}\frac{\omega(D)}{\omega(D/2)}\Big)\omega(t).
\end{align}
To finish the proof of \eqref{omega-cond:extensionlemma-b} there remains to observe that, if $t\geq D$,
\begin{equation}\label{tddG-9j}
t\int_t^{\infty}\frac{\widetilde{\omega}(s)}{s}\frac{ds}{s}=\omega(D)=\widetilde{\omega}(t).
\end{equation}

Next, the claims in item $(c)$ are direct consequences of Remark~\ref{rem:wextension} and parts {\it (a)}-{\it (b)}.
Turning our attention to item $(d)$, if $0<t_1\leq t_2<\infty$ then \eqref{omega-cond:extensionlemma-b} implies
\begin{equation}\label{643D} 
\frac{\widetilde{\omega}(t_2)}{t_2}\leq\int_{t_2}^{\infty}\frac{\widetilde{\omega}(s)}{s}\frac{ds}{s}
\leq\int_{t_1}^{\infty}\frac{\widetilde{\omega}(s)}{s}\frac{ds}{s}\leq C_2\frac{\widetilde{\omega}(t_1)}{t_1},
\end{equation} 
and the assertion in \eqref{wnonincreasing} follows by specializing this to the case when $t_1:=t$ and $t_2:=2t$. 
\end{proof}

In the classical case when, for some $\alpha\in(0,1)$, the growth function is defined as 
$\omega(t):=t^\alpha$ for each $t\in(0,\infty)$, the function $W(t):=t^{\alpha-1}$ for each 
$t\in(0,\infty)$ plays a significant role in the ensuing analysis. Below we identify the general 
format of the latter function associated with general growth functions. 

\begin{lemma}\label{6gfdc}
Given $D\in(0,\infty]$, let $\omega$ be a growth function on $(0,D)$. Set 
\begin{equation}\label{W+:general}
W_\omega(t):=\int_t^{D}\frac{\omega(s)}{s}\,\frac{ds}{s}\,\,\text{ for each }\,\,t\in(0,D).
\end{equation}

Then $W_\omega:(0,D)\to(0,\infty]$ is a non-increasing function which satisfies
\begin{equation}\label{w+:acdszc.WWW}
\omega(t)\Big(\frac{1}{t}-\frac{1}{D}\Big)\leq W_\omega(t)\,\,\text{ for each }\,\,t\in(0,D),
\end{equation} 
as well as 
\begin{align}\label{dewdead-ttt}
\int_0^{\tau}W_\omega(t)\,dt=\omega_Z(\tau)\,\,\text{ for each }\,\,\tau\in(0,D),
\end{align}
and which takes only finite values if $D<\infty$.

Moreover, if the growth function $\widetilde{\omega}$ on $(0,\infty)$ is associated with $\omega$ as in Remark~\ref{rem:wextension} 
and $W_{\widetilde\omega}$ is associated with $\widetilde{\omega}$ as in \eqref{W+:general}, then 
\begin{equation}\label{w+:acdszc}
\frac{\widetilde{\omega}(t)}{t}\leq W_{\widetilde{\omega}}(t)\,\,\text{ for each }\,\,t\in(0,\infty),
\end{equation} 
and if $\omega$ is doubling then there exists a constant $C\in(0,\infty)$ with the property that 
\begin{align}\label{dewdead-ttt.xxx}
\int_0^{\tau}W_{\widetilde\omega}(t)\,dt\leq C\omega_Z(\tau)\,\,\text{ for each }\,\,\tau\in(0,D).
\end{align}
\end{lemma}

\begin{proof}
By design, $W_\omega$ is a non-increasing function. Also, that $W_\omega$ only takes finite values 
when $D<\infty$ is clear from \eqref{W+:general} and the fact that $\omega$ is bounded in this case.

Next, bearing in mind that $\omega$ is non-decreasing, for every $t\in(0,D)$ we may write 
\begin{equation}\label{tfVV}
\omega(t)\Big(\frac{1}{t}-\frac{1}{D}\Big)=\int_t^D\frac{\omega(t)}{s}\frac{ds}{s} 
\leq\int_t^D\frac{\omega(s)}{s}\frac{ds}{s}=W_\omega(t),
\end{equation}
proving \eqref{w+:acdszc.WWW}. Note that \eqref{w+:acdszc} is a direct consequence of \eqref{w+:acdszc.WWW}
written for $\widetilde\omega$ in place of $\omega$ and with $D=\infty$. 
Next, for each $\tau\in(0,D)$ we may compute 
\begin{align}\label{dewdead-ytee}
\int_0^{\tau}W_\omega(t)\,dt &=\int_0^{\tau}\Big(\int_{t}^{\tau}\frac{\omega(s)}{s}\,\frac{ds}{s}\Big)\,dt
+\int_0^{\tau}\Big(\int_{\tau}^D\frac{\omega(s)}{s}\,\frac{ds}{s}\Big)\,dt
\nonumber\\[4pt]
&=\int_0^{\tau}\omega(s)\,\frac{ds}{s}
+\tau\int_{\tau}^D\frac{\omega(s)}{s}\,\frac{ds}{s}
=\omega_Z(\tau),
\end{align}
which establishes \eqref{dewdead-ttt}. Finally, consider the claim made in \eqref{dewdead-ttt.xxx}, 
under the additional assumption that $\omega$ is doubling. When $D=\infty$ this is already contained 
in \eqref{dewdead-ttt}. When $D<\infty$, directly from definitions we see that 
\begin{equation}\label{W+:geTT}
W_{\widetilde\omega}(t)=\int_t^\infty\frac{\widetilde{\omega}(s)}{s}\frac{ds}{s}
=\left\{
\begin{array}{ll}
W_{\omega}(t)+\frac{\omega(D)}{D} &\text{ if }\,\,t\in(0,D),
\\[4pt]
\frac{\omega(D)}{D} &\text{ if }\,\,t\in[D,\infty),
\end{array}
\right.
\end{equation}
so we may use \eqref{dewdead-ttt} to conclude that for each $\tau\in(0,D)$ we have
\begin{align}\label{dewdead-ttt.xxx.1}
\int_0^{\tau}W_{\widetilde\omega}(t)\,dt=\int_0^{\tau}W_{\omega}(t)\,dt+\tau\frac{\omega(D)}{D}
=\omega_Z(\tau)+\tau\frac{\omega(D)}{D}.
\end{align}
On the one hand, when $D/2\leq\tau<D$ the fact that $\omega$ is non-decreasing and doubling permits us to estimate 
\begin{align}\label{dewdead-ttt.xxx.2}
\tau\frac{\omega(D)}{D}\leq(C_\omega^{\rm dou})^2\omega(D/4)
&\leq\frac{(C_\omega^{\rm dou})^2}{\ln 2}\int_{D/4}^{D/2}\omega(s)\,\frac{ds}{s}
\nonumber\\[6pt]
&\leq\frac{(C_\omega^{\rm dou})^2}{\ln 2}\int_0^{\tau}\omega(s)\,\frac{ds}{s}
\leq\frac{(C_\omega^{\rm dou})^2}{\ln 2}\omega_Z(\tau).
\end{align}
On the other hand, when $0<\tau<D/2$, once again the fact that $\omega$ is non-decreasing and doubling permits us to estimate 
\begin{align}\label{dewdead-ttt.xxx.3}
\tau\frac{\omega(D)}{D} &\leq C_\omega^{\rm dou}\tau\frac{\omega(D/2)}{D}
=C_\omega^{\rm dou}\tau\int_{D/2}^D\frac{\omega(D/2)}{s}\,\frac{ds}{s}
\nonumber\\[6pt]
&\leq C_\omega^{\rm dou}\tau\int_{D/2}^D\frac{\omega(s)}{s}\,\frac{ds}{s}
\leq C_\omega^{\rm dou}\tau\int_{\tau}^D\frac{\omega(s)}{s}\,\frac{ds}{s}
\leq C_\omega^{\rm dou}\omega_Z(\tau).
\end{align}
At this stage, \eqref{dewdead-ttt.xxx} follows from \eqref{dewdead-ttt.xxx.1}-\eqref{dewdead-ttt.xxx.3}.
\end{proof}

\section{Generalized H\"{o}lder spaces}\label{sec:holder}

Here we introduce generalized {H}\"{o}lder spaces, consisting of functions 
whose continuity is quantified using growth functions of the sort discussed in the previous section. 

\begin{definition}\label{tFF}
Let $U\subseteq\mathbb{R}^{n}$ be an arbitrary set. 

\begin{enumerate}
\item Given a growth function $\omega$ on $(0,\infty)$, for each vector-valued function $u$ defined 
in $U$ consider $[u]_{\dot{\mathscr{C}}^\omega(U)}$ to be zero if $U$ is a singleton, and 
\begin{equation}\label{21356}
[u]_{\dot{\mathscr{C}}^\omega(U)}:=\sup_{\substack{x,y\in U\\ x\neq y}}\frac{|u(x)-u(y)|}{\omega(|x-y|)}\in[0,\infty]
\end{equation}
if the cardinality of the set $U$ is at least two. 
Then the \texttt{homogeneous $\omega$-H\"older space} on $U$ is introduced as
\begin{equation}\label{56edC}
\dot{\mathscr{C}}^\omega(U):=\big\{u:U\to\mathbb{C}:\,[u]_{\dot{\mathscr{C}}^\omega(U)}<\infty\big\}.
\end{equation}

\item If $D\in(0,\infty]$ and $\omega$ is a growth function on $(0,D)$, define 
the {\tt inhomogeneous} $\omega$-{\tt H\"older} {\tt space} on $U$ as
\begin{equation}\label{tedX}
\mathscr{C}^\omega(U):=\big\{u\in\dot{\mathscr{C}}^{\widetilde{\omega}}(U):\,\text{$u$ bounded on $U$}\big\},
\end{equation}
and is equipped with the norm
\begin{equation}\label{treD}
\mathscr{C}^\omega(U)\ni u\longmapsto
\|u\|_{\mathscr{C}^\omega(U)}:=\sup_{U}|u|+[u]_{\dot{\mathscr{C}}^{\widetilde{\omega}}(U)}.
\end{equation} 
\end{enumerate}
\end{definition}

In relation to Definition~\ref{tFF} a few comments are in order. First, in the context of item {\rm (a)}, 
$[\cdot]_{\dot{\mathscr{C}}^\omega(U)}$ is a semi-norm. Second, the fact that $\omega(t)\to 0$ as $t\to 0^+$ implies that if $u\in\dot{\mathscr{C}}^\omega(U)$ 
then $u$ is a uniformly continuous function on $U$. Finally, we note that the choice $\omega(t):=t^\alpha$ 
for each $t\in(0,\infty)$, with $\alpha\in(0,1)$, yields the classical scale of H\"older spaces of order 
$\alpha$ on $U$.

Going further, it is clear from definitions that if $\omega$ is a growth function on $(0,\infty)$
and $U\subseteq\mathbb{R}^n$ is an arbitrary set, then 
\begin{equation}\label{tr-iTGav.cdF}
\begin{array}{c}
fg\in{\mathscr{C}}^\omega(U)\,\,\text{ and }\,\,
\|fg\|_{{\mathscr{C}}^\omega(U)}\leq\|f\|_{{\mathscr{C}}^\omega(U)}\|g\|_{{\mathscr{C}}^\omega(U)}
\\[4pt]
\text{for any two functions $f,g\in{\mathscr{C}}^\omega(U)$}.
\end{array}
\end{equation} 
Also, for each subset $V$ of $U$, the restriction operators
\begin{equation}\label{tr-iTGav}
\begin{array}{c}
\dot{\mathscr{C}}^\omega(U)\ni u\mapsto u\big|_{V}\in\dot{\mathscr{C}}^\omega(V)
\,\,\text{ and }\,\,
{\mathscr{C}}^\omega(U)\ni u\mapsto u\big|_{V}\in{\mathscr{C}}^\omega(V)
\\[4pt]
\text{are well-defined, linear, and bounded}.
\end{array}
\end{equation} 
In the opposite direction, we have the following extension result.

\begin{lemma}\label{holderclosure}
Let $U\subseteq\mathbb{R}^n$ be an arbitrary set, and suppose $\omega$ is a doubling growth function 
on $(0,\infty)$. Then $\dot{\mathscr{C}}^\omega(U)$ and $\dot{\mathscr{C}}^\omega(\overline{U})$ 
coincide as vector spaces and have equivalent semi-norms. More specifically, the restriction map 
\begin{equation}\label{tgVVVa}
\dot{\mathscr{C}}^\omega(\overline{U})\ni u\longmapsto u\big|_{U}\in\dot{\mathscr{C}}^\omega(U)
\end{equation}
is a linear bijection which, under the canonical identification of functions 
$u\in\dot{\mathscr{C}}^\omega(\overline{U})$ with their restrictions $u\big|_{U}\in\dot{\mathscr{C}}^\omega(U)$, satisfies
\begin{equation}\label{holderclosureeq.xxx}
\begin{array}{c}
[u]_{\dot{\mathscr{C}}^\omega(U)}\leq[u]_{\dot{\mathscr{C}}^\omega(\overline{U})} 
\leq C_\omega^{\rm dou} [u]_{\dot{\mathscr{C}}^\omega(U)}
\\[4pt]
\text{for each function }\,\,u\in\dot{\mathscr{C}}^\omega(U),
\end{array}
\end{equation}
where $C_\omega^{\rm dou}\in[1,\infty)$ is the doubling constant of $\omega$. 

As a corollary of this, \eqref{wnonin-hGFF.cz}, and definitions, whenever 
$U\subseteq\mathbb{R}^n$ is an arbitrary set, and $\omega$ is a doubling growth function 
on $(0,D)$ for some $D\in(0,\infty]$, one may canonically identify 
${\mathscr{C}}^\omega(U)\equiv{\mathscr{C}}^\omega(\overline{U})$ and
\begin{equation}\label{holderclosureeq}
\begin{array}{c}
\|u\|_{{\mathscr{C}}^\omega(U)}\leq\|u\|_{{\mathscr{C}}^\omega(\overline{U})} 
\leq C_\omega^{\rm dou}\|u\|_{{\mathscr{C}}^\omega(U)}
\\[4pt]
\text{for each function }\,\,u\in{\mathscr{C}}^\omega(U).
\end{array}
\end{equation}
\end{lemma}

\begin{proof}
Fix an arbitrary $u\in\dot{\mathscr{C}}^{\omega}(U)$. 
As noted earlier, this membership ensures that $u$ is uniformly continuous, hence $u$ extends uniquely 
to a continuous function $v$ on $\overline{U}$. To show that $v$ actually belongs to 
$\dot{\mathscr{C}}^\omega(\overline{U})$ pick two arbitrary distinct points $x,y\in\overline{U}$.
Then there exist two sequences $\{x_j\}_{j\in{\mathbb{N}}}$, $\{y_j\}_{j\in{\mathbb{N}}}$ of points in $U$ 
such that $x_j\to x$ and $y_j\to y$ as $j\to\infty$. After discarding finitely many terms, there is no 
loss of generality in assuming that $0<|x_j-y_j|<2|x-y|$ for each $j\in{\mathbb{N}}$. 
Relying on the fact that $\omega$ is non-decreasing we may then write 
\begin{align}\label{6f4d-jCF}
|v(x)-v(y)| &=\lim_{j\to\infty}|u(x_j)-u(y_j)|
\leq[u]_{\dot{\mathscr{C}}^{\omega}(U)}\limsup_{j\to\infty}\omega(|x_j-y_j|) 
\nonumber\\[4pt] 
&\leq[u]_{\dot{\mathscr{C}}^{\omega}(U)}\,\omega(2|x-y|) 
\leq C_\omega^{\rm dou}[u]_{\dot{\mathscr{C}}^{\omega}(U)}\,\omega(|x-y|),
\end{align}
where $C_\omega^{\rm dou}\in[1,\infty)$ is the doubling constant of $\omega$. This ultimately shows that 
$v\in\dot{\mathscr{C}}^\omega(\overline{U})$ and $[v]_{\dot{\mathscr{C}}^\omega(\overline{U})}
\leq C_\omega^{\rm dou}[u]_{\dot{\mathscr{C}}^{\omega}(U)}$. All desired conclusions now follow. 
\end{proof}

The following example shows that generalized H{\"o}lder spaces can be distinct from
classical H{\"o}lder spaces.

\begin{example}\label{ex:strict}
Let $U\subset\mathbb{R}^n$ with $D:=\diam(U)<\infty$, and assume that 
\begin{equation}\label{eq:strict-hyp}
\parbox{10.00cm}{there exist a point $x_0\in U$, a rank $k_0\in\mathbb{N}$, two constants $C_0,C_1\in(0,\infty)$
and a sequence $(x_k)_{k\in\mathbb{N}}\subset\mathbb{R}^n$ such that $C_0/k\leq |x_k-x_0|\leq C_1/k$ 
and $x_k\in U$ for every $k\geq k_0$.}
\end{equation}
Informally speaking, we are asking that $U$ contains a sequence convergent to a point in $U$ at a prescribed rate.
Note that this assumption is very mild and it is always satisfied, for instance, whenever $U$ is open or $U$ is Ahlfors regular
(cf. Definition~\ref{treD.iii} below). 

Going further, let $\varphi(t)$ be a growth function on $(0,D)$. Extend $\varphi$ to $[0,D]$ by setting 
$\varphi(0):=0$ and $\displaystyle\varphi(D):=\lim_{t\to D^{-}}\varphi(t)$.
Assume that $\varphi$ is subadditive in $[0,D]$, that is,
\begin{equation}\label{eq:strict-phi}
\varphi(a+b)\leq\varphi(a)+\varphi(b),
\quad\text{for all }a,b\in[0,D]
\text{ with }a+b\in[0,D].
\end{equation} 
Fix $x,y,z\in U$ and assume first that $|x-z|\geq|y-z|$.
Taking $a:=|y-z|$ and $b:=|x-z|-|y-z|$ in \eqref{eq:strict-phi}, we obtain 
\begin{equation}\label{eq:strict-5}
0\leq\varphi(|x-z|)-\varphi(|y-z|)\leq\varphi(|x-z|-|y-z|),
\end{equation}
On the other hand, if $|y-z|\geq|x-z|$, a similar computation yields that
\begin{equation}\label{eq:strict-71}
0\leq\varphi(|y-z|)-\varphi(|x-z|)\leq\varphi(|y-z|-|x-z|).
\end{equation}
Gathering \eqref{eq:strict-5} and \eqref{eq:strict-71}, taking into account that
$\varphi$ is a non-decreasing function on $[0,D]$, and using the reverse triangle inequality,
we obtain
\begin{equation}\label{eq:strict-72}
|\varphi(|x-z|)-\varphi(|y-z|)|\leq\varphi\big(\big||x-z|-|y-z|\big|\big)\leq\varphi(|x-y|),
\end{equation}
for every $x,y,z\in U$. We conclude that:
\begin{equation}\label{eq:strict-claim}
\parbox{10.20cm}{if $\varphi$ is a subadditive growth function, $z\in U$ is a fixed point, and $F:U\to{\mathbb{R}}$ is 
defined as $F(x):=\varphi(|x-z|)$ for each $x\in U$, then $[F]_{\dot{\mathscr{C}}^{\widetilde{\varphi}}(U)}\leq 1$.}
\end{equation}

Continuing onward, fix $\alpha\in(0,1)$ and define now the growth function
\begin{equation}\label{eq:strict-yrF.2}
\omega(t):=t^{\alpha}\,\Big(\alpha^{-1}+\ln(D/t)\Big)
\,\,\text{ for all }\,\,t\in(0,D).
\end{equation}
Note first that this constitutes a particular instance of \eqref{yrF.5} in 
Example~\ref{trfdfd-2}, and consequently, condition \eqref{omega-cond:main.111} holds.
Moreover, for every $\varepsilon\in(0,\alpha)$ there is a constant $C\in(0,\infty)$ such that
\begin{equation}
t^{\alpha}\leq\omega(t)\leq C t^{\alpha-\varepsilon}\,\,\text{ for all }\,\,t\in (0,D),
\end{equation}
and, as a consequence,
\begin{equation}\label{eq:strict-1}
\mathscr{C}^{\alpha}(U)
\subseteq\mathscr{C}^\omega(U)
\subseteq\mathscr{C}^{\alpha-\varepsilon}(U).
\end{equation}
We will show that these inclusions are actually strict.

We start by proving that the second inclusion is strict. Fix $\varepsilon\in(0,\alpha)$ arbitrarily and, with $x_0\in U$ as in 
\eqref{eq:strict-hyp}, define $f(x):=|x-x_0|^{\alpha-\varepsilon}$ for every $x\in U$. From \eqref{eq:strict-claim}
applied with $z:=x_0$ and the subadditive growth function $\varphi:=|\cdot|^{\alpha-\varepsilon}$, it follows that 
$f\in\dot{\mathscr{C}}^{\alpha-\varepsilon}(U)$. As $f$ is bounded on $U$, we have
\begin{equation}\label{eq:strict-claim.BBB}
f\in{\mathscr{C}}^{\alpha-\varepsilon}(U).
\end{equation}
Recall the rank $k_0\in{\mathbb{N}}$, the sequence $(x_k)_{k\in\mathbb{N}}$, and the constants $0<C_0\leq C_1<\infty$ from 
\eqref{eq:strict-hyp}. In particular, there exists a unique integer $N\geq 2$ with the property that 
\begin{equation}\label{eq:kGFFD}
2^{N-2}C_0\leq C_1<2^{N-1}C_0.
\end{equation} 
We also observe that \eqref{eq:strict-hyp} implies that
\begin{equation}\label{eq:strict-prop-x2n}
|x_k-x_{2^Nk}|\leq\frac{5C_1}{4k}\qquad\text{ for every }k\geq k_0.
\end{equation}
Based on these, we may then estimate
\begin{align}\label{eq:strict-3}
[f]_{\dot{\mathscr{C}}^{\widetilde{\omega}}(U)}
&\geq\sup_{k\geq k_0}
\frac{\big||x_k-x_0|^{\alpha-\varepsilon}-|x_{2^Nk}-x_0|^{\alpha-\varepsilon}\big|}
{\widetilde{\omega}(|x_k-x_{2^Nk}|)}
\nonumber\\[8pt]
&\geq\sup_{k \geq k_0}\frac{\left(\frac{C_0}{k}\right)^{\alpha-\varepsilon}
\,(1-2^{-\alpha+\varepsilon})}{\widetilde{\omega}\big(\frac{5C_1}{4k}\big)}
\nonumber\\[8pt]
&\geq\lim_{k\to\infty}\frac{\left(\frac{C_0}{k}\right)^{\alpha-\varepsilon}
\,(1-2^{-\alpha+\varepsilon})}{\widetilde{\omega}\big(\frac{5C_1}{4k}\big)}
\nonumber\\[8pt]
&=\lim_{k\to\infty}\frac{\left(\frac{C_0}{k}\right)^{\alpha-\varepsilon}
\,(1-2^{-\alpha+\varepsilon})}
{\left(\frac{5C_1}{4k}\right)^{\alpha}\cdot\big(\alpha^{-1}+\ln\big(\frac{4kD}{5C_1}\big)\big)}
=\infty.
\end{align}
From this and \eqref{eq:strict-claim.BBB} we conclude that 
$f\in\mathscr{C}^{\alpha-\varepsilon}(U)\setminus\mathscr{C}^\omega(U)$.

To show that the first inclusion in \eqref{eq:strict-1} is also strict, define 
\begin{equation}\label{OIuygfdx}
g(x):=|x-x_0|^{\alpha}\,\bigg(\alpha^{-1}+\ln\bigg(\frac{D}{|x-x_0|}\bigg)\bigg)
\quad\text{ for every }\,\,x\in U\setminus\{x_0\}
\end{equation}
and set $g(x_0):=0$. On the one hand, we claim that 
\begin{equation}\label{eq:strict-claim.CCC}
g\in\mathscr{C}^\omega(U).
\end{equation}
Indeed, since $g$ is bounded on $U$, \eqref{eq:strict-claim.CCC} will follow from \eqref{eq:strict-claim} applied with 
$z:=x_0$ and $\varphi:=\omega$ (the growth function from \eqref{eq:strict-yrF.2}) as soon as we establish that $\omega$ 
is subadditive on $(0,D)$. To show the latter, observe via a direct computation that 
\begin{equation}\label{NGDsdfgh}
\omega''(t)=\alpha t^{\alpha-1}\,\Big[(\alpha-1)\ln(D/t)-1\Big]<0
\,\,\text{ for all }\,\,t\in(0,D).
\end{equation}
Thus, $\omega$ is concave down and we can write
\begin{equation}\label{NGDsdfgh.2}
\omega\big(\lambda t_1+(1-\lambda)t_2\big)\geq\lambda\omega(t_1)+(1-\lambda)\omega(t-2)
\,\,\text{ for all }\,\,t_1,\,t_2\in(0,D),\,\,\lambda\in(0,1).
\end{equation}
Passing to the limit with $t_2\to 0^{+}$ in \eqref{NGDsdfgh.2}, it follows that
\begin{equation}\label{NGDsdfgh.3}
\omega(\lambda t)\geq\lambda\omega(t)\,\,\text{ for all }\,\,t\in(0,D),\,\,\lambda\in(0,1).
\end{equation}
Hence, if $t_1,\,t_2\in(0,D)$ are such that $t_1+t_2\in(0,D)$, then \eqref{NGDsdfgh.3} implies
\begin{align}\label{NGDsdfgh.4}
\omega(t_1)+\omega(t_2)&=\omega\Big(\frac{t_1}{t_1+t_2}(t_1+t_2)\Big)+\omega\Big(\frac{t_2}{t_1+t_2}(t_1+t_2)\Big)
\nonumber\\[4pt]
&\geq \frac{t_1}{t_1+t_2}\omega(t_1+t_2)+\frac{t_2}{t_1+t_2}\omega(t_1+t_2)=\omega(t_1+t_2),
\end{align}
which proves that $\omega$ is subadditive. The justification of the membership in \eqref{eq:strict-claim.CCC} is complete.

On the other hand,
\begin{align}\label{eq:strict-4}
&[g]_{\dot{\mathscr{C}}^{\alpha}(U)}
=\sup_{\substack{x,y\in U\\x\neq y}}\frac{|g(x)-g(y)|}{|x-y|^{\alpha}}
\nonumber\\[8pt]
&\qquad\geq\sup_{k \geq k_0}
\frac{\big||x_k-x_0|^{\alpha}\,\big(\alpha^{-1}+\ln\big(\frac{D}{|x_k-x_0|}\big)\big)-|x_{2^Nk}-x_0|^{\alpha}
\,\big(\alpha^{-1}+\ln\big(\frac{D}{|x_{2^Nk}-x_0|}\big)\big)\big|}{|x_k-x_{2^Nk}|^{\alpha}}
\nonumber\\[8pt]
&\qquad\geq\lim_{k\to\infty}\frac{\left(\frac{C_0}{k}\right)^{\alpha}\,
\left(\alpha^{-1}+\ln\left(\frac{kD}{C_1}\right)\right)-\left(\frac{C_1}{2^Nk}\right)^{\alpha}\,
\left(\alpha^{-1}+\ln\left(\frac{2^NkD}{C_0}\right)\right)}{\left(\frac{5C_1}{4k}\right)^{\alpha}}=\infty.
\end{align}
From this and \eqref{eq:strict-claim.CCC}, we conclude that $g\in\mathscr{C}^\omega(U)\setminus\mathscr{C}^{\alpha}(U)$
as desired.
\end{example}

We continue by making the following definition. 

\begin{definition}\label{tfc}
Let $\Omega\subseteq\mathbb{R}^n$ be a nonempty open set and denote by 
$\mathscr{C}^1(\Omega)$ the space of continuously differentiable functions in $\Omega$. 
Given a growth function $\omega$ on $(0,D)$, with $D\in(0,\infty]$, for each $u\in\mathscr{C}^1(\Omega)$ define
\begin{equation}\label{y7uh}
\norm{u}_{\mathscr{C}^{1,\omega}(\Omega)}:=\sup_{\Omega}|u|+\sup_{\Omega}|\nabla u|
+[\nabla u]_{\dot{\mathscr{C}}^{\widetilde{\omega}}(\Omega)}\in[0,\infty],
\end{equation}
then introduce 
\begin{equation}\label{t4d6h-jNBB}
\mathscr{C}^{1,\omega}(\Omega):=\big\{u\in\mathscr{C}^1(\Omega):\,\norm{u}_{\mathscr{C}^{1,\omega}(\Omega)}<\infty\big\}.
\end{equation} 
\end{definition}

To close this section, we make the following convention. When simply speaking of a  
growth function, we shall understand a growth function $\omega$ defined on some interval 
$(0,D)$, with $D\in(0,\infty]$.

\section{Geometric measure theory}\label{sec:gmt}

As before, $\mathcal{H}^{n-1}$ denotes the $(n-1)$-dimensional Hausdorff measure in $\mathbb{R}^n$, 
and $\mathcal{L}^{n}$ denotes the $n$-dimensional Lebesgue measure in $\mathbb{R}^n$.
Given an open set $\Omega\subseteq{\mathbb{R}}^{n}$ and an aperture 
parameter $\kappa\in(0,\infty)$, define the nontangential approach regions 
\begin{equation}\label{yfrGV}
\Gamma_{\kappa}(x):=\big\{y\in\Omega:\,|y-x|<(1+\kappa)\,{\rm dist}\,(y,\partial\Omega)\big\} 
\,\,\text{ for each }\,\,x\in\partial\Omega,
\end{equation} 
and the nontangentially accessible boundary of $\Omega$ introduced in \cite[Definition~8.8.5, p.\,781]{GHA.I} as
\begin{equation}\label{eq:dNTA}
\partial_{{}_{\rm nta}}\Omega :=
\{x\in\partial\Omega:x\in\overline{\Gamma_\kappa(x)}\,\text{ for each }\,\kappa>0\}.
\end{equation} 
In turn, these regions are used to define the nontangential maximal operator ${\mathcal{N}}_\kappa$, 
acting on each continuous function $u\in{\mathscr{C}}^{\,0}(\Omega)$ according to 
\begin{equation}\label{6rtffd-1654}
\big({\mathcal{N}}_{\kappa}u\big)(x)=\sup\limits_{y\in\Gamma_{\kappa}(x)}|u(y)|
\,\,\text{ for all }\,\,x\in\partial\Omega,
\end{equation} 
with the convention that $\big({\mathcal{N}}_{\kappa}u\big)(x):=0$ whenever $x\in\partial\Omega$ 
is such that $\Gamma_\kappa(x)=\varnothing$. It turns out that 
${\mathcal{N}}_{\kappa}u:\partial\Omega\rightarrow[0,+\infty]$ is a lower-semicontinuous function. 

Continue to assume that $\Omega$ is an arbitrary open, nonempty, proper subset of ${\mathbb{R}}^n$ 
and suppose $u$ is a continuous function defined in $\Omega$. Also, fix an aperture parameter $\kappa>0$ 
and consider a point $x\in\partial_{{}_{\rm nta}}\Omega$. In this context, 
define the $\kappa$-nontangential limit of $u$ at $x$ as 
\begin{equation}\label{abxicnef7}
\big(u\big|^{{}^{\kappa-{\rm n.t.}}}_{\partial\Omega}\big)(x)
:=\lim_{\Gamma_\kappa(x)\ni y\to x}u(y),
\end{equation}
whenever the limit exists. 

Moving on, recall that an ${\mathcal{L}}^n$-measurable set $\Omega\subseteq{\mathbb{R}}^n$ has locally 
finite perimeter if its measure theoretic boundary, i.e., 
\begin{equation}\label{u7gFG-7t.1}
\partial_{*}\Omega:=\Big\{x\in\partial\Omega:\,\limsup_{r\to 0^{+}}\frac{\mathcal{L}^{n}(B(x,r)\cap\Omega)}{r^{n}}> 0, 
\enskip\limsup_{r\to 0^{+}}\frac{\mathcal{L}^{n}(B(x,r)\setminus\Omega)}{r^{n}}>0\Big\},
\end{equation}
satisfies
\begin{equation}\label{u7gFG-7t.2}
{\mathcal{H}}^{n-1}\big(\partial_{*}\Omega\cap K)<+\infty\,\,\text{ for each compact }\,\,K\subseteq{\mathbb{R}}^n
\end{equation}
(cf. \cite[Sections~5.7, 5.8, and 5.11, pp.\,194--226]{EvaGar92}).
Alternatively, an ${\mathcal{L}}^n$-measurable set $\Omega\subseteq{\mathbb{R}}^n$ has locally 
finite perimeter if, with the gradient taken in the sense of distributions in ${\mathbb{R}}^n$,
\begin{equation}\label{2.1.1}
\mu_{\Omega}:=\nabla{\mathbf{1}}_\Omega
\end{equation}
is an ${\mathbb{R}}^n$-valued Borel measure in ${\mathbb{R}}^n$ of locally finite total variation. 

Given a set $\Omega\subseteq{\mathbb{R}}^n$ of locally finite perimeter, fundamental work of De Giorgi-Federer 
(cf., e.g., \cite[Section~5.1, pp.\,166--171]{EvaGar92}) then ensures the following 
polar decomposition of the measure $\mu_\Omega$:
\begin{equation}\label{hmttr-3}
\mu_{\Omega}=\nabla{\mathbf{1}}_\Omega=-\nu\,|\nabla{\mathbf{1}}_\Omega|
\end{equation}
where $|\nabla{\mathbf{1}}_\Omega|$, the total variation measure of the measure $\nabla{\mathbf{1}}_\Omega$, 
is given by 
\begin{equation}\label{Nu-Dj-yygg}
|\nabla{\mathbf{1}}_\Omega|={\mathcal{H}}^{n-1}\lfloor\partial_\ast\Omega,
\end{equation}
and where 
\begin{equation}\label{Nu-Dj}
\begin{array}{l}
\text{$\nu\in\big[L^\infty(\partial_\ast\Omega,{\mathcal{H}}^{n-1})\big]^n$ is an 
${\mathbb{R}}^n$-valued function}
\\[6pt]
\text{satisfying $|\nu(x)|=1$ at ${\mathcal{H}}^{n-1}$-a.e. point $x\in\partial_\ast\Omega$}. 
\end{array}
\end{equation}
We shall refer to $\nu$ above as the {\tt geometric} {\tt measure} {\tt theoretic} {\tt outward} 
{\tt unit} {\tt normal} to $\Omega$.
Note here that by simply eliminating the distribution theory jargon implicit 
in the interpretation of \eqref{hmttr-3} (and using a straightforward limiting argument 
involving a mollifier) one already arrives at the formula 
\begin{equation}\label{WRTP22-1-GF}
\begin{array}{c}
\displaystyle
\int_\Omega{\rm div}\vec{F}\,d{\mathcal{L}}^n
=\int_{\partial_\ast\Omega}\nu\cdot\big(\vec{F}\big|_{\partial\Omega}\big)\,d{\mathcal{H}}^{n-1}
\\[10pt]
\text{for each vector field $\vec{F}\in\big[{\mathscr{C}}^{1}_c({\mathbb{R}}^n)\big]^n$.}
\end{array}
\end{equation}

For a set $\Omega\subseteq{\mathbb{R}}^n$ of locally finite perimeter, we let 
$\partial^*\Omega$ denote the {\tt reduced} {\tt boundary} of $\Omega$, that is,
\begin{equation}\label{PP-WA1}
\parbox{10.00cm}{$\partial^*\Omega$ consists of all points 
$x\in\partial\Omega$ satisfying the following properties: 
$0<\mathcal{H}^{n-1}\big(B(x,r)\cap\partial_\ast\Omega\big)<+\infty$ for each $r\in(0,\infty)$, and
$\lim\limits_{r\to 0^{+}}\fint_{B(x,r)\cap\partial_\ast\Omega}\nu\,d\mathcal{H}^{n-1}=\nu(x)\in S^{n-1}$.}
\end{equation}
For any set $\Omega\subseteq{\mathbb{R}}^n$ of locally finite perimeter we then have (cf. \cite[Lemma ~1, p.\,208]{EvaGar92})
\begin{equation}\label{bdaryproperties}
\partial^{*}\Omega\subseteq\partial_{*}\Omega\subseteq\partial\Omega\,\,\text{ and }\,\,
\mathcal{H}^{n-1}\big(\partial_{*}\Omega\setminus\partial^{*}\Omega\big)=0.
\end{equation}

\begin{definition}\label{treD.iii}
A closed set $E\subseteq\mathbb{R}^{n}$ is called {\tt lower} {\tt Ahlfors} {\tt regular} if there exists a 
constant $c\in(0,\infty)$ such that 
\begin{equation}\label{eq:ADR.a}
c\,r^{n-1}\leq\mathcal{H}^{n-1}\big(B(x,r)\cap E\big)\,\,\text{ for all }\,\,r\in(0,2\diam(E))
\,\,\text{ and all }\,\,x\in E.
\end{equation} 
Also, a closed set $E\subseteq\mathbb{R}^{n}$ is called {\tt upper} {\tt Ahlfors} {\tt regular} if there exists a 
constant $C\in(0,\infty)$ such that 
\begin{equation}\label{eq:ADR.b}
\mathcal{H}^{n-1}\big(B(x,r)\cap E\big)\leq Cr^{n-1}
\,\,\text{ for all }\,\,r\in(0,\infty)\,\,\text{ and all }\,\,x\in{\mathbb{R}}^n.
\end{equation} 
Finally, a closed set $E\subseteq\mathbb{R}^{n}$ is simply called {\tt Ahlfors} {\tt regular} if 
it is both lower Ahlfors regular and upper Ahlfors regular. The constants $c,C$ intervening in 
\eqref{eq:ADR.a}-\eqref{eq:ADR.b} make up the Ahlfors regular character of $E$.
\end{definition}

\begin{remark}\label{RRR.D+M}
In the context of Definition~\ref{treD.iii}, estimate \eqref{eq:ADR.b} self-improves to 
\begin{equation}\label{eq:ADR.b.rrr}
\mathcal{H}^{n-1}\big(\,\overline{B(x,r)}\cap E\big)\leq Cr^{n-1}
\,\,\text{ for all }\,\,r\in(0,\infty)\,\,\text{ and all }\,\,x\in{\mathbb{R}}^n,
\end{equation} 
for the same constant $C\in(0,\infty)$ as in \eqref{eq:ADR.b}.
To justify this claim, given any point $x\in{\mathbb{R}}^n$ and any radius $r\in(0,\infty)$, for each $R\in(r,\infty)$ use \eqref{eq:ADR.b} 
with $R$ in place of $r$ to write $\mathcal{H}^{n-1}\big(\,\overline{B(x,r)}\cap E\big)\leq\mathcal{H}^{n-1}\big(B(x,R)\cap E\big)\leq CR^{n-1}$ 
and then pass to limit $R\searrow r$ to obtain \eqref{eq:ADR.b.rrr}.
\end{remark}

Following \cite[p.\,2572]{HofMitTay10} we make the following definition. 

\begin{definition}\label{utRF-he42}
A nonempty, proper, open subset $\Omega$ of $\mathbb{R}^{n}$ is called an {\tt Ahlfors} {\tt regular} {\tt domain} 
provided $\partial\Omega$ is an Ahlfors regular set and $\mathcal{H}^{n-1}(\partial\Omega\setminus\partial_{*}\Omega)=0$. 
\end{definition}

From \cite[Proposition~2.9, p.\,2588]{HofMitTay10} we know that 
\begin{equation}\label{weaklyaccesible}
\parbox{10.90cm}{if $\Omega\subseteq\mathbb{R}^{n}$ is an Ahlfors regular domain, and 
$\kappa>0$ is an arbitrary aperture parameter, then $x\in\overline{\Gamma_{\kappa}(x)}$ 
(i.e., $x$ is an accumulation point for the nontangential approach region $\Gamma_{\kappa}(x)$)
for ${\mathcal{H}}^{n-1}$-a.e. point $x$ in the topological boundary $\partial\Omega$.}
\end{equation}
In particular, if $\Omega\subseteq\mathbb{R}^{n}$ is an Ahlfors regular domain and 
$u$ is an ${\mathcal{L}}^n$-measurable function defined in $\Omega$, then for any fixed 
aperture parameter $\kappa>0$ it is meaningful to attempt to define the nontangential boundary trace 
$\big(u\big|^{{}^{\kappa-{\rm n.t.}}}_{\partial\Omega}\big)(x)$ at 
${\mathcal{H}}^{n-1}$-a.e. point $x\in\partial\Omega$.

For future reference, it is also useful to remark that (see \cite[Lemma 5.10.9(6), pp.\,466-467]{GHA.I} for a proof)
\begin{equation}\label{76g-h5Fa-ee5ds}
\parbox{11.00cm}{if $\Omega\subset{\mathbb{R}}^n$ is an Ahlfors regular domain 
then $\Omega_{-}:={\mathbb{R}}^n\setminus\overline{\Omega}$ is also 
an Ahlfors regular domain, whose topological boundary coincides with that of $\Omega$,
and whose geometric measure theoretic boundary agrees with that of $\Omega$, i.e., 
$\partial(\Omega_{-})=\partial\Omega$ and $\partial_\ast(\Omega_{-})=\partial_\ast\Omega$. 
Moreover, the geometric measure theoretic outward unit normal to $\Omega_{-}$ 
is $-\nu$ at $\sigma$-a.e. point on $\partial\Omega$.}
\end{equation}

We continue by recalling the notion of countable rectifiability. 

\begin{definition}\label{tdC-yu65}
A closed set $E\subset\mathbb{R}^{n}$ is said to be {\tt countably} {\tt rectifiable} 
{\rm (}of dimension $(n-1)${\rm )} provided 
\begin{equation}\label{64eddd} 
E=\left(\bigcup_{j=1}^{\infty}S_j\right)\cup N, 
\end{equation} 
where $N$ is a null-set for $\mathcal{H}^{n-1}$ and each $S_j$ is the image of a compact subset of $\mathbb{R}^{n-1}$ 
under a Lipschitz map from $\mathbb{R}^{n-1}$ to $\mathbb{R}^{n}$. 
\end{definition}

The following definition is due to G.~David and S.~Semmes (cf. \cite[Definition~I.1.65, p.\,24]{DavSem93}). 

\begin{definition}\label{6feDa99j}
A closed set $E\subseteq{\mathbb{R}}^{n}$ is said to be {\tt uniformly} {\tt rectifiable} 
{\rm(}or simply {\rm UR}{\rm )} if $E$ is Ahlfors regular and there exist $\varepsilon,M\in(0,\infty)$ 
such that for each location $x\in E$ and each scale $R\in(0\,,\,2\,{\rm diam}\,(E)\big)$ it 
there exists a Lipschitz map $\varphi:B_{R}^{n-1}\to{\mathbb{R}}^{n}$ {\rm (}where $B_{R}^{n-1}$ 
is a ball of radius $R$ in ${\mathbb{R}}^{n-1}${\rm )} with Lipschitz constant $\leq M$ and such that
\begin{equation}\label{eq:URDef}
{\mathcal{H}}^{n-1}\big(E\cap B(x,R)\cap\varphi(B_R ^{n-1})\big)\geq\varepsilon R^{n-1}. 
\end{equation} 
\end{definition}

Uniformly rectifiability is a quantitative version of countable rectifiability. Any {\rm UR} set is 
countably rectifiable. We also remark that any Ahlfors regular domain in $\mathbb{R}^n$ has an 
$(n-1)$-dimensional countably rectifiable boundary (cf. \cite[(2-24), p.\,967]{MitMitVer16}).

The following definition first appeared in \cite[Definition~3.7, p.\,2631]{HofMitTay10}. 

\begin{definition}\label{tFCa97}
A nonempty, proper, open subset $\Omega$ of $\mathbb{R}^{n}$ is called a {\tt UR} {\tt domain} provided 
$\partial\Omega$ is a {\rm UR} set and $\mathcal{H}^{n-1}(\partial\Omega\setminus\partial_{*}\Omega)=0$.
\end{definition}

We continue by recording the definition of the class of uniform
domains introduced by O.~Martio and J.~Sarvas in \cite[p.\,387]{MaSa78}.

\begin{definition}\label{tFCa-u6fr}
A nonempty, proper, open subset $\Omega$ of $\mathbb{R}^{n}$ is called a {\tt uniform} {\tt domain} 
if there exist $\varkappa\in[1,\infty)$ such that any two points $x,y\in\Omega$ may be joined in $\Omega$  
by a rectifiable path $\gamma$ satisfying 
\begin{equation}\label{7g4dd}
\begin{array}{c}
\length(\gamma)\leq\varkappa|x-y|\,\,\text{ and, for each $z\in\gamma$,} 
\\[6pt]
\min\big\{\length(\gamma_{x,z})\,,\,\length(\gamma_{y,z})\big\}\leq\varkappa\dist(z,\partial\Omega),
\end{array}
\end{equation}
where $\gamma_{x,z}$ and $\gamma_{y,z}$ are the sub-arcs of $\gamma$ joining $z$ with $x$ and $y$, respectively.
\end{definition}

\begin{definition}\label{6fff-hr}
Let $\omega$ be a growth function. A nonempty, proper, open subset $\Omega$ of $\mathbb{R}^n$ is called a 
{\tt Lyapunov} $\mathscr{C}^{1,\omega}$-{\tt domain} {\rm (}or simply a $\mathscr{C}^{1,\omega}$-{\tt domain}{\rm )} 
if there exist $r,h>0$ such that for every point $x_0\in\partial\Omega$ there exists a coordinate system 
$(x',x_n)\in{\mathbb{R}}^{n-1}\times{\mathbb{R}}$ which is isometric to the canonical one and has $x_0$ 
as its origin, and a function $\varphi\in{\mathscr{C}}^{1,\omega}({\mathbb{R}}^{n-1})$ such that 
\begin{equation}\label{6tgfvc-n54c}
\Omega\cap C(r,h)=\big\{x=(x',x_n)\in{\mathbb{R}}^{n-1}\times\mathbb{R}:\,|x'|<r,\,\,\varphi(x')<x_n<h\big\},
\end{equation}  
where $C(r,h)$ is the cylinder defined as 
\begin{equation}\label{9j8643}
C(r,h)=\big\{x=(x',x_n)\in{\mathbb{R}}^{n-1}\times{\mathbb{R}}:\,|x'|<r,\,\,-h<x_n<h\big\}.
\end{equation}
\end{definition}

\begin{remark}\label{YrafV.a}
Any $\mathscr{C}^{1,\omega}$-domain with compact boundary is, in particular, 
simultaneously a {\rm UR} domain and a uniform domain.
\end{remark}

\begin{remark}\label{YrafV.b}
Analogously to the characterization of $\mathscr{C}^1$ domains given in \cite[Theorem~2.19, p.\,608]{HofMitTay07}, one can 
prove that $\mathscr{C}^{1,\omega}$-domain are those open sets of locally finite perimeter with the 
property that the geometric measure theoretic outward unit normal $\nu$ to $\Omega$, after possibly being 
altered on a set of $\sigma$-measure zero, belongs to $\mathscr{C}^{\omega}(\partial\Omega)$.  
See \cite[(2.74), p.\,608]{HofMitTay07} for the proof in the case $\omega(t)=t^{\alpha}$ for each $t\in(0,\infty)$ 
with $\alpha\in(0,1)$, which is easily adapted to our scenario. 
\end{remark}

We close this section by proving some useful integral estimates on upper Ahlfors regular sets involving growth functions.
Here and elsewhere, for each number $a\in{\mathbb{R}}$ we agree to abbreviate $(a)_{+}:=\max\{a,0\}$.

\begin{lemma}\label{lemma:ijijil}
Suppose $\Sigma$ is closed subset of $\mathbb{R}^n$ satisfying an upper Ahlfors regularity condition with 
constant $C\in(0,\infty)$. Define $\sigma:={\mathcal{H}}^{n-1}\lfloor\Sigma$. Also, assume $\omega$ is a 
growth function on $(0,\infty)$. Finally, fix some $d\in{\mathbb{R}}$ and pick an arbitrary point $x\in\Sigma$. 
Then for each $r\in(0,\infty)$ one has 
\begin{equation}\label{lkmem}
\int_{\Sigma\cap\overline{B(x,r)}}\frac{\omega(|x-y|)}{|x-y|^{n-1+d}}\,d\sigma(y)
\leq C\frac{2^{d_{+}}\cdot 2^{(n-1+d)_{+}}}{\ln 2}\int_{0}^{2\,r}\frac{\omega(s)}{s^d}\frac{ds}{s},
\end{equation}
\begin{equation}\label{dedewad}
\int_{\Sigma\setminus B(x,r)}\frac{\omega(|x-y|)}{|x-y|^{n-1+d}}\,d\sigma(y)
\leq C\frac{2^{d_{+}}\cdot 2^{(n-1+d)_{+}}}{\ln 2}\int_{2\,r}^{4\diam(\Sigma)}\frac{\omega(s)}{s^d}\,\frac{ds}{s}.
\end{equation}
\end{lemma}

\begin{proof}
Fix $r\in(0,\infty)$ and abbreviate $B:=B(x,r)$. Using the fact that $\omega$ is non-decreasing we may estimate
\begin{align}\label{7dfd-DFfd.aaa}
\int_{\Sigma\cap\overline{B}}\frac{\omega(|x-y|)}{|x-y|^{n-1+d}}\,d\sigma(y)
&=\sum_{k=0}^\infty\int_{\Sigma\cap(2^{-k}\overline{B}\setminus 2^{-k-1}\overline{B})}
\frac{\omega(|x-y|)}{|x-y|^{n-1+d}}\,d\sigma(y)
\nonumber\\[4pt]
&\leq\sum_{k=0}^\infty 2^{(1-n-d)_{+}}\frac{\omega(2^{-k}r)}{(2^{-k-1}\,r)^{n-1+d}}\,\sigma(\Sigma\cap 2^{-k}\overline{B}).
\end{align}
To estimate each term in the sum in \eqref{7dfd-DFfd.aaa}, we recall that $\Sigma$ is upper Ahlfors regular
(cf. Remark~\ref{RRR.D+M}) and (by analyzing separately the cases $d=0$, $d>0$, and $d<0$) 
we use that $\omega$ is non-decreasing combined with the fact that $\frac{2^t-1}{t}\geq\ln 2$ for $t\in(0,\infty)$ to write
\begin{align}\label{7dfd-DFfd}
\sum_{k=0}^\infty & 2^{(1-n-d)_{+}}\frac{\omega(2^{-k}r)}{(2^{-k-1}\,r)^{n-1+d}}\,\sigma(\Sigma\cap 2^{-k}\overline{B})
\nonumber\\[4pt]
\leq &C\frac{2^{d_{+}}\cdot 2^{(n-1+d)_{+}}}{\ln 2}\sum_{k=0}^\infty\int_{2^{-k}\,r}^{2^{-k+1}\,r}
\frac{\omega(s)}{s^d}\frac{ds}{s}
\nonumber\\[4pt]
=&C\frac{2^{d_{+}}\cdot 2^{(n-1+d)_{+}}}{\ln 2}\int_{0}^{2\,r}\frac{\omega(s)}{s^d}\frac{ds}{s}.
\end{align}
Now \eqref{lkmem} follows from \eqref{7dfd-DFfd.aaa} and \eqref{7dfd-DFfd}.

Let us turn to \eqref{dedewad}. 
We may assume that $\Sigma\setminus B\neq\varnothing$, otherwise there is nothing to prove. With $\lfloor\cdot\rfloor$ denoting 
the floor function, set $N:=\lfloor\log_2\left(\diam(\Sigma)/r\right)\rfloor\in\mathbb{N}_0\cup\{\infty\}$, 
so that $\Sigma\setminus 2^k B=\varnothing$ for every integer $k>N$. Then, since $\omega$ is non-decreasing and
$\Sigma$ is upper Ahlfors regular, a reasoning similar to the one used in \eqref{7dfd-DFfd.aaa}-\eqref{7dfd-DFfd} allows us to write 
\begin{align}\label{uhh-jrR}
\int_{\Sigma\setminus B}\frac{\omega(|x-y|)}{|x-y|^{n-1+d}}\,d\sigma(y)
&=\sum_{k=0}^N\int_{\Sigma\cap(2^{k+1}\overline{B}\setminus 2^k B)}\frac{\omega(|x-y|)}{|x-y|^{n-1+d}}\,d\sigma(y)
\nonumber\\[4pt]
&\leq\sum_{k=0}^N 2^{(1-n-d)_{+}}\frac{\omega(2^{k+1}\,r)}{(2^k\,r)^{n-1+d}}\,\sigma(\Sigma\cap 2^{k+1}B)
\nonumber\\[4pt]
&\leq C\frac{2^{d_{+}}\cdot 2^{(n-1+d)_{+}}}{\ln 2}\sum_{k=0}^N\int_{2^{k+1}\,r}^{2^{k+2}\,r}\frac{\omega(s)}{s^d}\frac{ds}{s}
\nonumber\\[4pt]
&\leq C\frac{2^{d_{+}}\cdot 2^{(n-1+d)_{+}}}{\ln 2}\int_{2\,r}^{4\diam(\Sigma)}\frac{\omega(s)}{s^d}\frac{ds}{s},
\end{align}
proving \eqref{dedewad}. 
\end{proof}

\section{Singular integrals on generalized H\"older spaces}\label{Sec:6ff}

We first recall a basic result pertaining to the behavior of singular integral operators on {\rm UR} sets
from \cite[Theorem~2.3.2, (2.3.15), (2.3.56), pp.\,348--357]{GHA.III}, 
\cite[Theorem~2.4.1, (2.4.9), pp.\,382--390]{GHA.III}, \cite[Theorem~2.5.1, pp.\,407--408]{GHA.III}. 

\begin{theorem}\label{Clpthm}
Suppose $\Omega\subseteq{\mathbb{R}}^{n}$ is an open set such that $\partial\Omega$ is a {\rm UR} set. 
Abbreviate $\sigma:={\mathcal{H}}^{n-1}\lfloor\partial\Omega$ and denote by $\nu$ the geometric measure theoretic 
outward unit normal to $\Omega$. Also, assume $N=N(n)\in{\mathbb{N}}$ is a sufficiently 
large integer and consider a complex-valued function $k\in{\mathscr{C}}^N\big(\mathbb{R}^{n}\setminus\{0\}\big)$ 
which is odd and positive homogeneous of degree $1-n$. Finally, fix an aperture parameter $\kappa>0$.
In this setting, for each function $f\in L^1\big(\partial\Omega,\tfrac{\sigma(x)}{1+|x|^{n-1}}\big)$ define
\begin{equation}\label{tfd.DDD.1}
\mathcal{T}f(x):=\int_{\partial\Omega}k(x-y)f(y)\,d\sigma(y)\,\,\text{ for each }\,\,x\in\Omega,
\end{equation} 
\begin{equation}\label{TDef}
Tf(x):=\lim_{\varepsilon\to 0^{+}}\int\limits_{\substack{y\in\partial\Omega\\ |x-y|>\varepsilon}}k(x-y)f(y)\,d\sigma(y)
\,\,\text{ for $\sigma$-a.e. }\,\,x\in\partial\Omega.
\end{equation}

Then for each $f\in L^1\big(\partial\Omega,\tfrac{\sigma(x)}{1+|x|^{n-1}}\big)$ 
the limit in \eqref{TDef} exists $\sigma$-a.e. and one has the jump-formula
\begin{equation}\label{jumplp}
(\mathcal{T}f)\big|_{\partial\Omega}^{{}^{\kappa-{\rm n.t.}}}(x)=\frac{1}{2i}\widehat{k}(\nu(x))f(x)+Tf(x),
\end{equation} 
for $\sigma$-a.e. $x\in\partial_\ast\Omega$, where $\widehat{k}$ denotes the Fourier transform 
of $k$ in $\mathbb{R}^n$ and $i:=\sqrt{-1}\in{\mathbb{C}}$.

Also, for each integrability exponent $p\in(1,\infty)$ there exists a finite constant $C>0$ such that
for each function $f\in L^p(\partial\Omega,\sigma)$ one has
\begin{equation}\label{jvgv-hjT}
\norm{\mathcal{N}_\kappa(\mathcal{T}f)}_{L^p(\partial\Omega,\sigma)}\leq C\norm{f}_{L^p(\partial\Omega,\sigma)},
\end{equation} 
\begin{equation}\label{ytff}
\norm{Tf}_{L^p(\partial\Omega,\sigma)}\leq C\norm{f}_{L^p(\partial\Omega,\sigma)}.
\end{equation}
\end{theorem}

We now turn to the task of estimating singular integral operators acting on generalized H\"older spaces. 

\begin{lemma}\label{lemma:3.2-new}
Suppose $\Omega$ is a nonempty, proper, open subset of $\mathbb{R}^n$ with compact boundary
satisfying an upper Ahlfors regularity condition with constant $C\in(0,\infty)$ and abbreviate 
$\sigma:={\mathcal{H}}^{n-1}\lfloor\partial\Omega$. Consider a function 
$k:\Omega\times\partial\Omega\rightarrow\mathbb{R}$ with the property that $k(x,\cdot)$ is a 
$\sigma$-measurable function for each fixed point $x\in\Omega$, and there exists some finite 
constant $C_0>0$ such that 
\begin{equation}\label{eq:3.2-new-hip1}
|k(x,y)|\leq\frac{C_0}{|x-y|^{n-1}}\,\,\text{ for each $x\in\Omega$ and $\sigma$-a.e. }\,\,y\in\partial\Omega.
\end{equation}
Define an integral operator acting on each $f\in L^1(\partial\Omega,\sigma)$ according to 
\begin{equation}\label{6rEE-y77}
{\mathscr{T}}f(x):=\int_{\partial\Omega}k(x,y)f(y)\,d\sigma(y)\,\,\text{ for all }\,\,x\in\Omega,
\end{equation}
and assume 
\begin{equation}\label{gtgttg}
C_1:=\sup_{x\in\Omega}|\mathscr{T}1(x)|<+\infty.
\end{equation}
Finally, with $D\in(0,\infty)$, let $\omega$ be a growth function on $(0,D)$ for which 
\begin{equation}\label{dedqwe}
C_2:=\int_0^{D}\omega(s)\,\frac{ds}{s}<\infty,
\end{equation}
and associate with it the growth function $\widetilde{\omega}$ as in Remark~\ref{rem:wextension}.

Then for every function $f\in\mathscr{C}^{\omega}(\partial\Omega)$ one has
\begin{align}\label{6ggfv-jR}
\sup_{x\in\Omega}|{\mathscr{T}}f(x)| &\leq C\,C_0\,2^{n-1}\,
\Big[\Big(1+\frac{2^{n}}{\ln 4}\log_{+}\big(\tfrac{4\,{\rm diam}\,(\partial\Omega)}{D}\big)\Big)\,\omega(D)+\frac{2^{n}}{\ln 4}\,C_2\Big]
\,[f]_{\dot{\mathscr{C}}^{\widetilde{\omega}}(\partial\Omega)}
\nonumber\\[4pt] 
&\quad+\max\big\{C_1\,,\,C_0\,D^{1-n}\sigma(\partial\Omega)\big\}\cdot\sup_{\partial\Omega}|f|.
\end{align}
\end{lemma}

\begin{proof}
Pick an arbitrary function $f\in{\mathscr{C}}^{\omega}(\partial\Omega)\subseteq L^1(\partial\Omega,\sigma)$ 
and fix some point $x\in\Omega$. If ${\rm dist}(x,\partial\Omega)\geq D$ use \eqref{eq:3.2-new-hip1} to estimate 
\begin{equation}\label{7533-uH}
|\mathscr{T}f(x)|\leq\frac{C_0}{D^{n-1}}\sigma(\partial\Omega)\cdot\sup_{\partial\Omega}|f|.
\end{equation}

Consider next the case when $\mathrm{dist}(x,\partial\Omega)<D$. Pick $x_{*}\in\partial\Omega$ such that
\begin{equation}\label{6fSD-jR}
|x-x_*|=\mathrm{dist}(x,\partial\Omega)=:r\in(0,D),
\end{equation}
and abbreviate $B_{*}:=B(x_{*},r)$. Then we may decompose $\mathscr{T}f(x)={\rm I}+{\rm II}+{\rm III}$, where
\begin{align}\label{6tfff}
{\rm I} &:=\int_{\partial\Omega\cap\,2B_{*}}k(x,y)\,\big[f(y)-f(x_{*})\big]\,d\sigma(y),
\\[6pt]
{\rm II} &:=\int_{\partial\Omega\setminus 2B_{*}}k(x,y)\,\big[f(y)-f(x_*)\big]\,d\sigma(y),
\\[6pt]
{\rm III} &:=(\mathscr{T}1)(x)\,f(x_{*}).
\end{align}
Start by estimating the first term above:
\begin{align}\label{64eed-tRD}
|{\rm I}| &\leq C_0\,[f]_{\dot{\mathscr{C}}^{\widetilde{\omega}}(\partial\Omega)}
\int_{\partial\Omega\cap\,2B_{*}}\frac{\widetilde{\omega}(|y-x_{*}|)}{|x-y|^{n-1}}\,d\sigma(y)
\nonumber\\[4pt]
&\leq C_0\,[f]_{\dot{\mathscr{C}}^{\widetilde{\omega}}(\partial\Omega)}
\frac{\widetilde{\omega}(D)}{r^{n-1}}\,\sigma(\partial\Omega\cap 2\,B_{*})
\nonumber\\[4pt]
&\leq C\,C_0\,2^{n-1}\,\omega(D)\,[f]_{\dot{\mathscr{C}}^{\widetilde{\omega}}(\partial\Omega)}.
\end{align}
where we have used that $\omega$ is non-decreasing, that $r=\dist(x,\partial\Omega)\leq |x-y|$ 
for every $y\in\partial\Omega\cap\,2B_{*}$, and the upper Ahlfors regularity of $\partial\Omega$.

Let us now estimate ${\rm II}$. When $\partial\Omega\setminus 2B_{*}=\varnothing$ we have ${\rm II}=0$. 
When $\partial\Omega\setminus 2B_{*}\neq\varnothing$, for each $y\in\partial\Omega\setminus 2B_{*}$ 
we have $|y-x_{*}|\leq 2\,|y-x|$. This and Lemma~\ref{lemma:ijijil} permit us to estimate 
\begin{align}\label{7t4d4dd}
|{\rm II}| &\leq C_0\,[f]_{\dot{\mathscr{C}}^{\widetilde{\omega}}(\partial\Omega)}
\int_{\partial\Omega\setminus 2B_*}\frac{\widetilde{\omega}(|y-x_*|)}{|x-y|^{n-1}}\,d\sigma(y)
\nonumber\\[4pt]
&\leq 2^{n-1}\,C_0\,[f]_{\dot{\mathscr{C}}^{\widetilde{\omega}}(\partial\Omega)}
\int_{\partial\Omega\setminus 2B_*}\frac{\widetilde{\omega}(|y-x_*|)}{|y-x_*|^{n-1}}\,d\sigma(y)
\nonumber\\[4pt]
&\leq\frac{2^{2\,(n-1)}}{\ln 2}\,C\,C_0 [f]_{\dot{\mathscr{C}}^{\widetilde{\omega}}(\partial\Omega)}
\int_{4\,r}^{4\,{\rm diam}\,(\partial\Omega)}\widetilde{\omega}(s)\frac{ds}{s},
\nonumber\\[4pt]
&\leq\frac{2^{2\,(n-1)}}{\ln 2}\,C\,C_0\,
\big(C_2+\omega(D)\,\log_{+}\big(\tfrac{4\,{\rm diam}\,(\partial\Omega)}{D}\big)\big)\,
[f]_{\dot{\mathscr{C}}^{\widetilde{\omega}}(\partial\Omega)}.
\end{align}

Finally, 
\begin{equation}\label{trffcc}
|{\rm III}|\leq C_1\sup_{\partial\Omega}|f|,
\end{equation} 
so \eqref{6ggfv-jR} follows from \eqref{7533-uH} and \eqref{64eed-tRD}-\eqref{trffcc}.
\end{proof}

Here is a companion result to Lemma~\ref{lemma:3.2-new}, for integral operators whose kernel 
exhibits a stronger singularity at the boundary (compared to \eqref{eq:3.2-new-hip1}).

\begin{lemma}\label{lemma:3.3-new}
Assume $\Omega$ is a nonempty, proper, open subset of $\mathbb{R}^n$ with a compact boundary
satisfying an upper Ahlfors regularity condition with constant $C\in(0,\infty)$, and abbreviate 
$\sigma:={\mathcal{H}}^{n-1}\lfloor\partial\Omega$. Consider a function 
$q:\Omega\times\partial\Omega\rightarrow{\mathbb{R}}$ with the property 
that there exists some finite constant $C_3>0$ such that
\begin{equation}\label{tdcc-QDSD}
|q(x,y)|\leq\frac{C_3}{|x-y|^{n}}\,\,\text{ for each $x\in\Omega$ and $\sigma$-a.e. }\,\,y\in\partial\Omega, 
\end{equation}
and such that $q(x,\cdot)$ is a $\sigma$-measurable function for each fixed point $x\in\Omega$.
Use this to define an integral operator acting on each $f\in L^1(\partial\Omega,\sigma)$ according to 
\begin{equation}\label{yRDXxa}
{\mathscr{Q}}f(x):=\int_{\partial\Omega}q(x,y)f(y)\,d\sigma(y)\,\,\text{ for each }\,\,x\in\Omega.
\end{equation}
Next, fix some $D\in(0,\infty]$ and let $\omega$ be a growth function on $(0,D)$. Associate with 
it the growth function $\widetilde{\omega}$ as in Remark~\ref{rem:wextension} and, further,
the function $W_{\widetilde{\omega}}:(0,\infty)\to(0,\infty)$ defined as in \eqref{W+:general}. 
Finally, set $\rho(x):=\mathrm{dist}(x,\partial\Omega)$ for every $x\in\Omega$, 
and make the assumption that 
\begin{equation}\label{tcwda}
C_4:=\sup_{x\in\Omega}\frac{|{\mathscr{Q}}1(x)|}{W_{\widetilde{\omega}}(\rho(x))}<+\infty.
\end{equation}

Then for each function $f\in{\mathscr{C}}^\omega(\partial\Omega)$ one has
\begin{equation}\label{6fff-j6ii}
\sup_{x\in\Omega}\frac{|{\mathscr{Q}}f(x)|}{W_{\widetilde{\omega}}(\rho(x))}
\leq 2^{n}\,C\,C_3\Big(1+\frac{2^{n+1}}{\ln 2}\Big)[f]_{\dot{\mathscr{C}}^{\widetilde{\omega}}(\partial\Omega)}
+C_4\cdot\sup_{\partial\Omega}|f|.
\end{equation}
\end{lemma}

\begin{proof}
Choose an arbitrary function $f\in \mathscr{C}^\omega(\partial\Omega)$ and fix some point $x\in\Omega$. 
Next, pick $x_{*}\in\partial\Omega$ such that
\begin{equation}\label{yrrffd}
|x-x_{*}|=\mathrm{dist}(x,\partial\Omega)=\rho(x)=:r,
\end{equation}
and abbreviate $B_{*}:=B(x_{*},r)$. Then we may decompose ${\mathscr{Q}}f(x)={\rm I}+{\rm II}+{\rm III}$, where
\begin{align}\label{y6f6ff-nf4DCX}
{\rm I} &:=\int_{\partial\Omega\cap\,2B_{*}}q(x,y)\,\big[f(y)-f(x_*)\big]\,d\sigma(y),
\\[4pt]
{\rm II} &:=\int_{\partial\Omega\setminus 2B_*}q(x,y)\,\big[f(y)-f(x_*)\big]\,d\sigma(y),
\\[4pt]
{\rm III} &:=({\mathscr{Q}}1)(x)\,f(x_*).
\end{align}

We may then estimate 
\begin{align}\label{qssq}
|{\rm I}| &\leq C_3\,[f]_{\dot{\mathscr{C}}^{\widetilde{\omega}}(\partial\Omega)}
\int_{\partial\Omega\cap\,2B_{*}}\frac{\omega(|y-x_{*}|)}{|x-y|^{n}}\,d\sigma(y)
\nonumber\\[4pt]
&\leq C_3\,[f]_{\dot{\mathscr{C}}^{\widetilde{\omega}}(\partial\Omega)}
\frac{\widetilde{\omega}(2r)}{r^{n}}\,\sigma(\partial\Omega\cap\,2B_{*})
\nonumber\\[4pt]
&\leq 2^{n-1}C\,C_3\,[f]_{\dot{\mathscr{C}}^{\widetilde{\omega}}(\partial\Omega)}\frac{\widetilde{\omega}(2r)}{r}
\leq 2^{n}\,C\,C_3\,W_{\widetilde{\omega}}(2r)\,[f]_{\dot{\mathscr{C}}^{\widetilde{\omega}}(\partial\Omega)}
\nonumber\\[4pt]
&\leq 2^{n}\,C\,C_3\,W_{\widetilde{\omega}}(r)\,[f]_{\dot{\mathscr{C}}^{\widetilde{\omega}}(\partial\Omega)},
\end{align}
using that $\omega$ is non-decreasing, that $r=\dist(x,\partial\Omega)\leq|x-y|$ 
for every $y\in\partial\Omega$, the upper Ahlfors regularity of $\partial\Omega$, \eqref{w+:acdszc}, 
and the fact that $W_{\widetilde{\omega}}$ is non-increasing.

Next, note that ${\rm II}=0$ if $\partial\Omega\setminus 2B_{*}=\varnothing$. 
When $\partial\Omega\setminus 2B_{*}\neq\varnothing$ we have $|y-x_{*}|\leq 2\,|x-y|$ for each 
$y\in\partial\Omega\setminus 2B_{*}$. As such, we may estimate 
\begin{align}\label{kyiki}
|{\rm II}| &\leq C_3\,[f]_{\dot{\mathscr{C}}^{\widetilde{\omega}}(\partial\Omega)}
\int_{\partial\Omega\setminus 2B_{*}}\frac{\widetilde{\omega}(|y-x_*|)}{|x-y|^{n}}\,d\sigma(y)
\nonumber\\[4pt]
&\leq 2^{n}\,C_3\,[f]_{\dot{\mathscr{C}}^{\widetilde{\omega}}(\partial\Omega)}
\int_{\partial\Omega\setminus 2B_{*}}\frac{\widetilde{\omega}(|y-x_*|)}{|y-x_*|^{n}}\,d\sigma(y)
\nonumber\\[4pt]
&\leq\frac{2^{2n+1}}{\ln 2}\,C\,C_3\,W_{\widetilde{\omega}}(r)\,[f]_{\dot{\mathscr{C}}^{\widetilde{\omega}}(\partial\Omega)},
\end{align}
where for the second inequality we have used \eqref{dedewad} with $d:=1$, and the fact that $W_{\widetilde{\omega}}$ is non-increasing.
Finally, \eqref{tcwda} gives
\begin{equation}\label{hyhyh}
|{\rm III}|\leq C_4\,W_{\widetilde{\omega}}(r)\cdot\sup_{\partial\Omega}|f|.
\end{equation}
Combining \eqref{qssq}, \eqref{kyiki}, \eqref{hyhyh} then yields \eqref{6fff-j6ii}. 
\end{proof}

From Lemmas~\ref{lemma:3.2-new}-\ref{lemma:3.3-new} we then immediately derive the following result.

\begin{lemma}\label{lemma:pokor}
Let $\Omega$ be a nonempty, proper, open subset of $\mathbb{R}^n$ with a compact boundary 
satisfying an upper Ahlfors regularity condition with constant $C\in(0,\infty)$.
Abbreviate $\sigma:={\mathcal{H}}^{n-1}\lfloor\partial\Omega$ and 
$\rho(x):=\mathrm{dist}(x,\partial\Omega)$ for each $x\in\Omega$. In this setting, consider a 
function $K:\Omega\times\partial\Omega\rightarrow{\mathbb{R}}$, which is continuously differentiable 
in the first argument, with the property that there exists some finite constant $A_0>0$ such that 
\begin{equation}\label{yhtghrt}
\begin{array}{c}
|K(x,y)|+|x-y|\,|\nabla_x K(x,y)|\leq A_0|x-y|^{1-n}
\\[6pt]
\text{for each $x\in\Omega$ and $\sigma$-a.e. }\,\,y\in\partial\Omega,
\end{array}
\end{equation}
and such that $K(x,\cdot)$ is a $\sigma$-measurable function for each fixed point $x\in\Omega$.
Define an integral operator acting on each function $f\in L^1(\partial\Omega,\sigma)$ according to 
\begin{equation}\label{744rfd-ij}
{\mathcal{T}}f(x):=\int_{\partial\Omega}K(x,y)f(y)\,d\sigma(y)\,\,\text{ for each }\,\,x\in\Omega.
\end{equation}
Let $\omega$ be a growth function on $(0,D)$, where $D\in(0,\infty)$.
Associate with it the growth function $\widetilde{\omega}$ as in Remark~\ref{rem:wextension}, 
and bring in the function $W_{\widetilde{\omega}}$ associated with $\widetilde{\omega}$ as in \eqref{W+:general}.
In relation to these, assume 
\begin{equation}\label{aadede}
A_1:=\int_0^{D}\omega(s)\,\frac{ds}{s}<+\infty,
\end{equation}
and 
\begin{equation}\label{nhgfnh}
A_2:=\sup_{x\in\Omega}|({\mathcal{T}}1)(x)|
+\sup_{x\in\Omega}\frac{|\nabla({\mathcal{T}}1)(x)|}{W_{\widetilde{\omega}}(\rho(x))}<+\infty.
\end{equation} 

Then for every $f\in{\mathscr{C}}^\omega(\partial\Omega)$ one has
\begin{align}\label{754ew3D}
&\sup_{x\in\Omega}|({\mathcal{T}}f)(x)|+\sup_{x\in\Omega}\frac{|\nabla({\mathcal{T}}f)(x)|}{W_{\widetilde{\omega}}(\rho(x))}
\nonumber\\[6pt]
&\quad\leq C\,A_0\,2^{n-1}\,\Big[\Big(1+\frac{2^{n}}{\ln 4}\log_{+}\big(\tfrac{4\,{\rm diam}\,(\partial\Omega)}{D}\big)\Big)\,\omega(D)
+\frac{2^{n}}{\ln 4}\,A_1+2+\frac{2^{n+2}}{\ln 2}\Big]\,[f]_{\dot{\mathscr{C}}^{\widetilde{\omega}}(\partial\Omega)}
\nonumber\\[6pt] 
&\qquad+\big(2\,A_2+A_0 D^{1-n}\sigma(\partial\Omega)\big)\sup_{\partial\Omega}|f|.
\end{align}

As a consequence, there exists a finite constant $C_{\Omega,n,\omega}>0$ with the property that 
for each function $f\in{\mathscr{C}}^\omega(\partial\Omega)$ one has
\begin{equation}\label{mlmmlm}
\sup_{x\in\Omega}|({\mathcal{T}}f)(x)|
+\sup_{x\in\Omega}\frac{|\nabla(\mathcal{T} f)(x)|}{W_{\widetilde{\omega}}(\rho(x))}
\leq C_{\Omega,n,\omega}\,\big(A_0\,(1+A_1)+A_2\big)\,\|f\|_{{\mathscr{C}}^{\omega}(\partial\Omega)}.
\end{equation}
\end{lemma}

We continue by estimating the generalized H\"older norm of a ${\mathscr{C}}^1$ function 
defined in a uniform domain (cf. Definition~\ref{tFCa-u6fr}).

\begin{lemma}\label{lemma:fokogtr}
Suppose $\Omega\subset{\mathbb{R}}^n$ is a uniform domain. Recall the constant $\varkappa\in[1,\infty)$ 
appearing in \eqref{7g4dd} and abbreviate $\rho(x):=\mathrm{dist}(x,\partial\Omega)$ for each $x\in\Omega$. 
In addition, fix some $D\in(0,\infty]$ and consider a growth function $\omega$ on $(0,D)$ which is doubling 
and Dini. Associate with it $\widetilde{\omega}$ as in Remark~\ref{rem:wextension}, and associate with 
$\widetilde{\omega}$ the function $W_{\widetilde{\omega}}$ as in \eqref{W+:general}. 

Then there exists a finite constant $C=C(\omega,D,\varkappa)>0$ with the property 
that for each function $u\in{\mathscr{C}}^1(\Omega)$ one has
\begin{equation}\label{eq:deddoo}
\norm{u}_{\mathscr{C}^{\omega_{\!Z}}(\Omega)}\leq C\bigg(\sup_{x\in\Omega}|u(x)|
+\sup_{x\in\Omega}\frac{|(\nabla u)(x)|}{W_{\widetilde{\omega}}(\rho(x))}\bigg).
\end{equation}
\end{lemma}

\begin{proof}
Fix two distinct points $x,y\in\Omega$ and assume first that $|x-y|<D$. Since $\Omega$ is a uniform 
domain with constant $\varkappa\in[1,\infty)$, there exists a rectifiable curve $\gamma:[0,L]\to\Omega$  
(written using the arc-length parametrization) joining $x$ with $y$ within $\Omega$ and satisfying 
$L=\length(\gamma)\leq\varkappa|x-y|$ as well as
\begin{equation}\label{eq:lmkkokm}
\min\{s,L-s\}\leq\varkappa\,\mathrm{dist}(\gamma(s),\partial\Omega)=\varkappa\,\rho(\gamma(s))
\,\,\text{ for each }\,\,s\in[0,L].
\end{equation}
Since $\gamma$ is a.e. differentiable and $|d\gamma/ds|=1$ for almost every $s\in(0,L)$, 
for any given function $u\in\mathscr{C}^1(\Omega)$ we may write
\begin{align}\label{yuhguhg}
|u(x)-u(y)| &=\left|\int_{0}^L\frac{d}{ds}\big[u(\gamma(s))]\,ds\right|
\leq\int_0^L |(\nabla u)(\gamma(s))|\,ds
\nonumber\\[4pt] 
&\leq\sup_{x\in\Omega}\frac{|(\nabla u)(x)|}{W_{\widetilde\omega}(\rho(x))}\,
\int_0^L W_{\widetilde\omega}\big(\rho(\gamma(s))\big)\,ds.
\end{align}
On the other hand, since $W_{\widetilde\omega}$ is non-increasing, on account of \eqref{eq:lmkkokm} we may write
\begin{align}\label{eeerer}
\int_0^L W_{\widetilde\omega}\big(\rho(\gamma(s))\big)\,ds
&\leq\int_0^L W_{\widetilde\omega}\big(\varkappa^{-1}\,\min\{s,L-s\}\big)\,ds
\nonumber\\[6pt]
&=2\,\int_0^{\frac{L}{2}}W_{\widetilde\omega}(\varkappa^{-1}\,s)\,ds
=2\,\varkappa\,\int_0^{\frac{L}{2\varkappa}}W_{\widetilde\omega}(s)\,ds
\nonumber\\[6pt]
&\leq 2\,\varkappa\,\int_0^{|x-y|}W_{\widetilde\omega}(s)\,ds\leq C\omega_Z(|x-y|).
\end{align}
For the first inequality in \eqref{eeerer} we have also used the fact that $L\leq\varkappa|x-y|$, while the last 
inequality in \eqref{eeerer} is based on \eqref{dewdead-ttt.xxx}. Combining \eqref{yuhguhg} with \eqref{eeerer} yields 
\begin{equation}\label{lemeq1}
|u(x)-u(y)|\leq C\,\Big(\sup_{x\in\Omega}
\frac{|(\nabla u)(x)|}{W_{\widetilde{\omega}}(\rho(x))}\Big)\,\omega_Z(|x-y|)
\end{equation}
for all distinct points $x,y\in\Omega$ such that $|x-y|<D$. If $D\in(0,\infty)$, then for any $x,y\in\Omega$ 
satisfying $|x-y|>D$ we may write 
\begin{equation}\label{lemeq2}
\frac{|u(x)-u(y)|}{\widetilde{\omega_Z}(|x-y|)}\leq\frac{2}{\omega_Z(D)}\sup_{z\in\Omega}|u(z)|.
\end{equation}
From \eqref{lemeq1} and \eqref{lemeq2} the desired estimate now follows. 
\end{proof}

\section{Proof of Theorem~\ref{theor:uewmp.222}}
\label{KJB.D+M}

The proof of Theorem~\ref{theor:uewmp.222}, presented at the end of this section, relies on the material developed thus far, 
but additional work is required. First, in Theorem~\ref{thm:GHA-T1} below, we establish a geometrically sharper form of the boundedness 
statement appearing in item~\textit{(i)} of Theorem~\ref{theor:uewmp.222}. Second, Theorem~\ref{thm:GHA-T1.JUMP}, stated shortly thereafter, 
anticipates the boundary trace result asserted in item~\textit{(ii)} of Theorem~\ref{theor:uewmp.222}.

\begin{theorem}\label{thm:GHA-T1}
Let $n\in{\mathbb{N}}$ be such that $n\geq 2$. Assume $\Omega\subset\mathbb{R}^n$ is an open set with a compact upper Ahlfors 
regular boundary, satisfying ${\mathcal{H}}^{n-1}(\partial\Omega\setminus\partial_\ast\Omega)=0$.
Abbreviate $\sigma:={\mathcal{H}}^{n-1}\lfloor\partial\Omega$ and denote by $\nu$ the geometric 
measure theoretic outward unit normal to $\Omega$. Consider a vector-valued function 
$\vec{k}\in\big[{\mathscr{C}}^2({\mathbb{R}}^n\setminus\{0\})\big]^n$ which is odd, divergence-free, and 
positive homogeneous of degree $1-n$ and associate with it the complex number $\vartheta$ as in \eqref{kjshgh.FF.D+M}.
Suppose $\omega:\big(0,\diam(\partial\Omega)\big)\to(0,\infty)$ is a doubling Dini growth function, and 
associate with it its Zygmund pair, $\omega_Z$ as in \eqref{omega-cond:main}.

Then for each $f\in{\mathscr{C}}^{\omega}(\partial\Omega)$ the generalized double layer potential operator
\begin{align}\label{16778.bbb.222.GDL.D+M}
Tf(x):=\lim_{\varepsilon\to 0^{+}}\int\limits_{\substack{y\in\partial\Omega\\ |x-y|>\varepsilon}}
\langle\nu(y),\vec{k}(x-y)\rangle f(y)\,d\sigma(y)\,\,\text{ for $\sigma$-a.e. }\,\,x\in\partial\Omega,
\end{align}
is meaningfully defined and coincides, outside of a $\sigma$-nullset, with a function in ${\mathscr{C}}^{\omega_Z}(\partial\Omega)$.
Thus interpreted, $T$ induces a well-defined, linear, and bounded mapping 
\begin{equation}\label{Jgsag.TTT.222.D+M}
T:{\mathscr{C}}^{\omega}(\partial\Omega)\longrightarrow{\mathscr{C}}^{\omega_Z}(\partial\Omega).
\end{equation}
In particular, under the assumption that $\omega$ is a Zygmund modulus of continuity, i.e., there exists $C\in(0,\infty)$ for which 
\begin{equation}\label{omega-cond:main.333.D+M}
\int_0^{t}\omega(s)\frac{ds}{s}+t\,\int_t^{\diam(\partial\Omega)}\frac{\omega(s)}{s}\,\frac{ds}{s}
\leq C\omega(t)\,\,\text{ for each }\,\,t\in\big(0,\diam(\partial\Omega)\big),
\end{equation}
it follows that 
\begin{equation}\label{Jgsag.TTT.333.D+M}
T:{\mathscr{C}}^{\omega}(\partial\Omega)\longrightarrow{\mathscr{C}}^{\omega}(\partial\Omega)
\end{equation}
is a well-defined, linear, and bounded mapping.
\end{theorem}

\begin{proof}
Recall that an exterior domain is defined as the complement of a compact set in $\mathbb{R}^n$. Since 
$n\geq 2$ and $\partial\Omega$ is compact, by \cite[Lemma~5.10.10, p.\,469]{GHA.I}, $\Omega$ is an exterior domain 
if and only if $\Omega$ is unbounded. With $\vartheta$ as in \eqref{kjshgh.FF.D+M}, \cite[Proposition~2.5.16, pp.\,438--439]{GHA.III} 
ensures that, at $\sigma$-a.e. point on $\partial \Omega$,
\begin{equation}\label{eq:double-T1}
T1=
\left\lbrace
\begin{array}{ll}
-\vartheta/2, & \text{if } \Omega \text{ is bounded}, 
\\[8pt]
+\vartheta/2, & \text{if } \Omega \text{ is an exterior domain}. 
\end{array}
\right.
\end{equation}

Going further, set $D:=\diam(\partial\Omega)\in(0,\infty)$ and let $\widetilde{\omega}$ be associated with $\omega$ as in 
Remark~\ref{rem:wextension}. Also, fix some $f\in{\mathscr{C}}^\omega(\partial\Omega)$.
From the properties of $\vec{k}$ and \eqref{lkmem}-\eqref{dedewad} we see that there exists some $C\in(0,\infty)$, which depends 
only on $\vec{k}$, the upper Ahlfors regularity constant of $\partial\Omega$ and the doubling constant
of $\omega$, such that for every $x\in\partial\Omega$ we have
\begin{align}\label{eq:sio-deadpo} 
\int_{\partial\Omega} &\left|\langle\nu(y),\vec{k}(x-y)\rangle(f(y)-f(x))\right|\,d\sigma(y)
\leq C \int_{\partial\Omega}\frac{|f(y)-f(x)|}{|x-y|^{n-1}}\,d\sigma(y)
\nonumber\\[4pt] 
&\qquad\leq C\,[f]_{\dot{\mathscr{C}}^{\widetilde{\omega}}(\partial\Omega)}\,
\int_{\partial\Omega}\frac{\widetilde{\omega}(|x-y|)}{|x-y|^{n-1}}\,d\sigma(y)
\leq C\,[f]_{\dot{\mathscr{C}}^{\widetilde{\omega}}(\partial\Omega)}
\int_{0}^{4D}\widetilde{\omega}(s)\,\frac{ds}{s}
\nonumber\\[4pt]
&\qquad\leq C\,[f]_{\dot{\mathscr{C}}^{\widetilde{\omega}}(\partial\Omega)}
\int_{0}^{D}\omega(s)\,\frac{ds}{s}<+\infty,
\end{align}
where in the penultimate inequality we have made the change of variables $s\mapsto 4s$ and used the doubling 
property of $\omega$ and \eqref{wnonin-hGFF.cz}. Granted this, for $\sigma$-a.e. $x\in\partial\Omega$ we may write 
(based on Lebesgue's Dominated Convergence Theorem and \eqref{eq:double-T1}) 
\begin{align}\label{eq:sio-ffrc}
&\lim_{\varepsilon\to 0^{+}}\,\int\limits_{\substack{y\in\partial\Omega\\ |x-y|>\varepsilon}}
\langle\nu(y),\vec{k}(x-y)\rangle f(y)\,d\sigma(y),
\nonumber\\[4pt]
&\quad
=\lim_{\varepsilon\to 0^{+}}\,\int\limits_{\substack{y\in\partial\Omega\\|x-y|>\varepsilon}}
\langle\nu(y),\vec{k}(x-y)\rangle(f(y)-f(x))\,d\sigma(y)\mp\frac{\vartheta}{2}\,f(x)
\nonumber\\[4pt]
&\quad
=\int\limits_{\partial\Omega}
\langle\nu(y),\vec{k}(x-y)\rangle(f(y)-f(x))\,d\sigma(y)\mp\frac{\vartheta}{2}\,f(x),
\end{align}
where the sign of $\frac{\vartheta}{2}\,f(x)$ is minus if $\Omega$ is bounded and plus if $\Omega$ is unbounded. 
This implies that for $\sigma$-a.e. $x\in\partial\Omega$ the limit in the second line of \eqref{eq:sio-ffrc} exists. 
Thus, up to redefining $T$ on a nullset for $\sigma$, we may assume that at {\it every} point $x\in\partial\Omega$ we have  
\begin{equation}\label{eq:sio-abfhtj}
Tf(x)=\mp\frac{\vartheta}{2}\,f(x)+\,
\int_{\partial\Omega}\langle\nu(y),\vec{k}(x-y)\rangle(f(y)-f(x))\,d\sigma(y),
\end{equation}
where $\mp$ depends on $\Omega$ being bounded or unbounded, as explained above.

Our next goal is to show that the operator \eqref{eq:sio-abfhtj} is well-defined and bounded from
$\mathscr{C}^\omega(\partial\Omega)$ into $\mathscr{C}^\omega(\partial\Omega)$. To this end, fix two distinct 
points $x_1,x_2\in\partial\Omega$ and use \eqref{eq:sio-abfhtj} to decompose
\begin{equation}\label{eq:sio-qtfew}
Tf(x_1)-Tf(x_2)={\rm I}+{\rm II}, 
\end{equation}
where
\begin{equation}\label{eq:sio-reecd}
{\rm I}:=\mp\frac{\vartheta}{2}\,(f(x_1)-f(x_2)),
\end{equation}
and
\begin{equation}\label{eq:sio-alpp}
{\rm II}:=\int_{\partial\Omega}\big(\langle\nu(y),\vec{k}(x_1-y)\rangle(f(y)-f(x_1))
-\langle\nu(y),\vec{k}(x_2-y)\rangle(f(y)-f(x_2))\big)\,d\sigma(y).
\end{equation}
Set $r:=|x_1-x_2|\in(0,D)$ and note that this and \eqref{wnonin-hGFF.cz.UU} entails 
\begin{equation}\label{eq:sio-qppe}
|{\rm I}|\leq\frac{|\vartheta|}{2}\,[f]_{\dot{\mathscr{C}}^{\widetilde{\omega}}(\partial\Omega)}\,
\omega(r)\leq C\,[f]_{\dot{\mathscr{C}}^{\widetilde{\omega}}(\partial\Omega)}\,
\omega_Z(r),
\end{equation}
where $C\in(0,\infty)$ depends only on $\vartheta$ and the doubling constant of $\omega$.
Also, we may bound 
\begin{equation}\label{eq:sio-oioeew}
|{\rm II}|\leq C\big({\rm II}_1+{\rm II}_2+{\rm II}_3\big),
\end{equation}
where
\begin{align}\label{eq:sio-alpp***}
{\rm II}_1 &:=\bigg|\,\int\limits_{\substack{y\in\partial\Omega\\ |x_1-y|>2r}}
\big(\langle\nu(y),\vec{k}(x_1-y)\rangle(f(y)-f(x_1))
\nonumber\\[4pt] 
&\qquad\qquad -\langle\nu(y),\vec{k}(x_2-y)\rangle(f(y)-f(x_2))\,\big)\,d\sigma(y)\bigg|,
\end{align}
\begin{equation}\label{eq:sio-654r4e-uyt.1}
{\rm II}_2:=[f]_{\dot{\mathscr{C}}^{\widetilde{\omega}}(\partial\Omega)}\,
\int\limits_{\substack{y\in\partial\Omega\\ |x_1-y|\leq 2r}}\frac{\widetilde{\omega}(|x_1-y|)}{|x_1-y|^{n-1}}\,d\sigma(y),
\end{equation}
and
\begin{align}\label{eq:sio-654r4e-uyt.2}
{\rm II}_3 &:=\,[f]_{\dot{\mathscr{C}}^{\widetilde{\omega}}(\partial\Omega)}\,
\int\limits_{\substack{y\in\partial\Omega\\ |x_1-y|\leq 2r}}\frac{\widetilde{\omega}(|x_2-y|)}{|x_2-y|^{n-1}}\,d\sigma(y)
\nonumber\\[4pt] 
&\leq[f]_{\dot{\mathscr{C}}^{\widetilde{\omega}}(\partial\Omega)}\,
\int\limits_{\substack{y\in\partial\Omega\\ |x_2-y|\leq 3r}}\frac{\widetilde{\omega}(|x_2-y|)}{|x_2-y|^{n-1}}\,d\sigma(y).
\end{align}
The inequality in \eqref{eq:sio-654r4e-uyt.2} is a consequence of the fact that 
$|x_2-y|\leq|x_2-x_1|+|x_1-y|\leq 3\,r$ whenever $|x_1-y|\leq 2r$. We may then use \eqref{lkmem},
the fact that $\omega$ is doubling, and \eqref{wnonin-hGFF.cz} to conclude that 
\begin{align}\label{eq:sio-ootr}
{\rm II}_2+{\rm II}_3 
&\leq C\,[f]_{\dot{\mathscr{C}}^{\widetilde{\omega}}(\partial\Omega)}
\,\int_0^{6r}\widetilde{\omega}(s)\frac{ds}{s}
\leq C\,[f]_{\dot{\mathscr{C}}^{\widetilde{\omega}}(\partial\Omega)}
\,\int_0^{r}\omega(s)\frac{ds}{s}
\nonumber\\[6pt]
&\leq C[f]_{\dot{\mathscr{C}}^{\widetilde{\omega}}(\partial\Omega)} \,\omega_Z(r),
\end{align}
for $C\in(0,\infty)$ which depends only on $n$, the upper Ahlfors regularity constant of $\partial\Omega$, 
and the doubling constant of $\omega$.

From \cite[(2.5.151), p.\,435]{GHA.III} we know that there exists a constant $C\in(0,\infty)$, which depends only 
on $\vec{k}$, the dimension $n$, and the upper Ahlfors regularity constant of $\partial\Omega$, with the property that
\begin{equation}\label{Cau-C2btg590.FF.2.NEW!}
\sup_{x\in{\mathbb{R}}^n}\sup_{r\in(0,\infty)}\Big|\int_{\partial\Omega\setminus\overline{B(x,r)}}
\big\langle\nu(y),\vec{k}(x-y)\big\rangle\,d\sigma(y)\Big|\leq C.
\end{equation}
Let us turn our attention to ${\rm II}_1$ which we further decompose and estimate by
\begin{equation}\label{eq:sio-ndnf}
{\rm II}_1\leq{\rm II}_{11}+{\rm II}_{12},
\end{equation}
where
\begin{align}\label{eq:sio-bvgev.AAA}
{\rm II}_{11} &:=\bigg|\,\int_{\partial\Omega\setminus\overline{B(x_1,2r)}}
\langle\nu(y),\vec{k}(x_1-y)\rangle\big(f(x_2)-f(x_1)\big)\,d\sigma(y)\bigg|,
\\[4pt]
{\rm II}_{12} &:=\bigg|\,\int_{\partial\Omega\setminus\overline{B(x_1,2r)}}
\langle\nu(y),\vec{k}(x_1-y)-\vec{k}(x_2-y)\rangle\big(f(y)-f(x_2)\big)\,d\sigma(y)\bigg|.
\label{eq:sio-bvgev.BBB}
\end{align}
Making use of \eqref{Cau-C2btg590.FF.2.NEW!} and \eqref{wnonin-hGFF.cz.UU}, we have
\begin{align}\label{eq:sio-bvgev}
{\rm II}_{11} &\leq
[f]_{\dot{\mathscr{C}}^{\widetilde{\omega}}(\partial\Omega)}\,
\omega(r)\,\bigg|\int_{\partial\Omega\setminus\overline{B(x_1,2r)}}\langle\nu(y),\vec{k}(x_1-y)\rangle\,d\sigma(y)\bigg|
\nonumber\\[4pt]
&\leq C\,[f]_{\dot{\mathscr{C}}^{\widetilde{\omega}}(\partial\Omega)}
\,\omega(r)
\leq C\,[f]_{\dot{\mathscr{C}}^{\widetilde{\omega}}(\partial\Omega)}
\omega_Z(r),
\end{align}
for some $C\in(0,\infty)$ which depends only on $\vec{k}$, $n$ and the upper Ahlfors regularity constant of $\partial\Omega$.

To estimate ${\rm II}_{12}$ observe that, by the Mean Value Theorem and the homogeneity of $\vec{k}$,
\begin{equation}\label{POHgf.kk}
\big|\vec{k}(x_1-y)-\vec{k}(x_2-y)\big|\leq C\frac{r}{|x_2-y|^n}\,\,\text{ uniformly for }\,
y\in\partial\Omega\setminus\overline{B(x_1,2r)},
\end{equation}
with $C\in(0,\infty)$ depending only on $\vec{k}$. This, the fact that 
$\partial\Omega\setminus\overline{B(x_1,2r)}\subseteq \partial\Omega\setminus B(x_2,r)$, 
estimate \eqref{dedewad} with $d:=1$, the doubling property of $\omega$, \eqref{wnonin-hGFF.cz}, 
and \eqref{omega-cond:main}, imply
\begin{align}\label{eq:sio-ppqqp}
{\rm II}_{12} &\leq\int_{\partial\Omega\setminus\overline{B(x_1,2r)}}
\big|\vec{k}(x_1-y)-\vec{k}(x_2-y)\big|\,\big|f(y)-f(x_2)\big|\,d\sigma(y)
\nonumber\\[4pt]
&\leq C\,[f]_{\dot{\mathscr{C}}^{\widetilde{\omega}}(\partial\Omega)}\,r
\int_{\partial\Omega\setminus B(x_2,r)}\frac{\widetilde{\omega}(|x_2-y|)}{|x_2-y|^n}\,d\sigma(y)
\nonumber\\[4pt]
&\leq C\,[f]_{\dot{\mathscr{C}}^{\widetilde{\omega}}(\partial\Omega)}\,r
\int_{2r}^{4D}\frac{\widetilde{\omega}(s)}{s}\,\frac{ds}{s}
\leq C\,[f]_{\dot{\mathscr{C}}^{\widetilde{\omega}}(\partial\Omega)}\,r
\int_{r/2}^{D}\frac{\omega(s)}{s}\,\frac{ds}{s}
\nonumber\\[4pt]
&\leq C\,[f]_{\dot{\mathscr{C}}^{\widetilde{\omega}}(\partial\Omega)}\,\omega_Z(r/2)
\leq C\,[f]_{\dot{\mathscr{C}}^{\widetilde{\omega}}(\partial\Omega)}\,\omega_Z(r),
\end{align}
for some $C\in(0,\infty)$, which depends only on $\vec{k}$, $n$, the upper Ahlfors regularity constant 
of $\partial\Omega$, and the doubling constant of $\omega$.

In concert, \eqref{eq:sio-qtfew}, \eqref{eq:sio-qppe}, \eqref{eq:sio-oioeew}, \eqref{eq:sio-ootr}, \eqref{eq:sio-ndnf}, 
\eqref{eq:sio-bvgev}, and \eqref{eq:sio-ppqqp}, and recalling that $r=|x_1-x_2|$, allow us to conclude that there 
exists some $C\in(0,\infty)$, depending only on $\vec{k}$, $n$, $\omega$, $D$, and the upper regularity 
constant of $\partial\Omega$ with the property that 
\begin{equation}\label{eq:sio-awdfe}
\begin{array}{c}
\big|Tf(x_1)-Tf(x_2)\big|
\leq C\,[f]_{\dot{\mathscr{C}}^{\widetilde{\omega}}(\partial\Omega)}\,
\omega_Z(|x_1-x_2|)
\\[6pt]
\text{for all distinct points }\,\,x_1,x_2\in\partial\Omega.
\end{array}
\end{equation}
Since \eqref{eq:sio-abfhtj} and \eqref{eq:sio-deadpo} also imply 
\begin{equation}\label{eq:sio-7654ee.A}
\sup_{\partial\Omega}|Tf|\leq
C\,[f]_{\dot{\mathscr{C}}^{\widetilde{\omega}}(\partial\Omega)}
\cdot\Big(\int_{0}^{D}\omega(s)\,\frac{ds}{s}\Big)+\frac{|\vartheta|}{2}\sup_{\partial\Omega}|f|,
\end{equation}
we ultimately obtain 
\begin{equation}\label{eq:sio-7654ee.B}
\|Tf\|_{{\mathscr{C}}^{\omega_{\!Z}}(\partial\Omega)}\leq 
C\big(\sup_{\partial\Omega}|f|+[f]_{\dot{\mathscr{C}}^{\widetilde{\omega}}(\partial\Omega)}\big)
=C\|f\|_{{\mathscr{C}}^{\omega}(\partial\Omega)},
\end{equation}
where $C\in(0,\infty)$ depends on $\vec{k}$, $n$, $\omega$, $D$, and the upper regularity constant 
of $\partial\Omega$. Therefore, we conclude that the operator \eqref{Jgsag.TTT.222.D+M} is bounded.

Finally, since by \eqref{omega-cond:main.333.D+M} we have $\omega_Z(t)\leq C\omega(t)$ for every $t\in(0,\diam(\partial\Omega))$, 
that $T$ is bounded in the setting \eqref{Jgsag.TTT.333.D+M} whenever \eqref{omega-cond:main.333.D+M} holds 
follows at once from the boundedness of the operator \eqref{Jgsag.TTT.222.D+M}.
\end{proof}

Here is the theorem, advertised at the beginning of this section, which prefigures 
the boundary trace result claimed in item~\textit{(ii)} of Theorem~\ref{theor:uewmp.222}.

\begin{theorem}\label{thm:GHA-T1.JUMP}
Let $n\in{\mathbb{N}}$ be such that $n\geq 2$ and assume that $\Omega\subset\mathbb{R}^n$ is an Ahlfors regular domain 
with compact boundary. Abbreviate $\sigma:={\mathcal{H}}^{n-1}\lfloor\partial\Omega$ and denote by $\nu$ the geometric 
measure theoretic outward unit normal to $\Omega$. Consider a vector-valued function 
$\vec{k}\in\big[{\mathscr{C}}^2({\mathbb{R}}^n\setminus\{0\})\big]^n$ which is odd, divergence-free, and 
positive homogeneous of degree $1-n$. Associate with $\vec{k}$ the operators $\mathcal{T}$ and $T$ as in \eqref{GDL-TH.it.1} and
\eqref{16778.bbb.222.GDL}, respectively. Finally, define the complex number $\vartheta$ as in \eqref{kjshgh.FF.D+M}.

Then the following jump-formula {\rm (}where $I$ denotes the identity operator, 
and $\kappa>0$ is an arbitrary fixed aperture parameter{\rm )} is valid:
\begin{equation}\label{2.3.21BBBB.NEW!}
\begin{array}{c}
\text{for every given function $f\in{\mathscr{C}}^{\omega}(\partial\Omega)$ one has}
\\[6pt]
\big(\mathcal{T}f\big)\Big|^{{}^{\kappa-{\rm n.t.}}}_{\partial\Omega}
=(-\tfrac{\vartheta}{2}I+T)f\,\text{ at $\sigma$-a.e. point on }\,\partial\Omega,
\end{array}
\end{equation}
whenever $\omega:\big(0,\diam(\partial\Omega)\big)\to(0,\infty)$ is a doubling Dini growth function. 
\end{theorem}

\begin{proof}
We being by noting that the action of the integral operator ${\mathcal{T}}$ is presently meaningful on the 
generalized H\"older space ${\mathscr{C}}^{\omega}(\partial\Omega)$ (since this is contained in $L^1(\partial\Omega,\sigma)$). 
Also, given that the boundary of $\Omega$ is presently assumed to be an Ahlfors regular set satisfying 
$\sigma(\partial\Omega\setminus\partial_\ast\Omega)=0$, we may invoke \cite[Proposition~8.8.6(iii), p.\,782]{GHA.I} which 
ensures that 
\begin{equation}\label{eq:jdhdnd.D+M}
\sigma(\partial\Omega\setminus\partial_{{}_{\rm nta}}\Omega)=0,
\end{equation}
where $\partial_{{}_{\rm nta}}\Omega$ denotes the nontangentially accessible boundary of $\Omega$ (see \eqref{eq:dNTA}). 
As such, it is meaningful to consider the nontangential boundary trace in the left-hand side of the equality in \eqref{2.3.21BBBB.NEW!} 
in the present setting. Fix now some H\"older function $f\in{\mathscr{C}}^{\omega}(\partial\Omega)$ and observe from \eqref{GDL-TH.it.1} that 
${\mathcal{T}}f$ is continuous in $\Omega$. Also, recall from \eqref{eq:sio-ffrc} that there exists a $\sigma$-measurable set 
$A\subseteq\partial\Omega$ satisfying $\sigma(A)=0$ and with the property that the limit 
\begin{align}\label{TD15-AAAAA.NEW!}
\lim_{\varepsilon\to 0^{+}}\int\limits_{\substack{|x-y|>\varepsilon\\ y\in\partial\Omega}}
\big\langle\nu(y),\vec{k}(x-y)\big\rangle f(y)\,d\sigma(y)\,\text{ exists for each }\,x\in\partial\Omega\setminus A.
\end{align}
Fix an arbitrary point
\begin{align}\label{TD14.abd.NEW!}
x\in\big(\partial\Omega\setminus A\big)\cap\partial_{{}_{\rm nta}}\Omega.
\end{align}
After possibly augmenting $A$ by a $\sigma$-nullset in $\partial\Omega$, we can assume (thanks to \cite[Lemma~2.5.10, p.\,431]{GHA.III} 
used with $\Omega$ replaced by ${\mathbb{R}}^n\setminus\Omega$) that there exists some ${\mathcal{L}}^1$-nullset $N_x\subset(0,\infty)$ 
with the property that for each $\varepsilon\in(0,\infty)\setminus N_x$ the set $\partial B(x,\varepsilon)\setminus\Omega$ 
is ${\mathcal{H}}^{n-1}$-measurable and 
\begin{equation}\label{Cau-C2.6a7.NEW!.XXX}
\lim_{\substack{\varepsilon\in(0,\infty)\setminus N_x\\ \varepsilon\to 0}}\,\,
\int_{\substack{|x-y|=\varepsilon\\ y\in{\mathbb{R}}^n\setminus\Omega}}
\Big\langle\frac{y-x}{\varepsilon}\,,\,\vec{k}(x-y)\Big\rangle\,d{\mathcal{H}}^{n-1}(y)=-\frac{\vartheta}{2}. 
\end{equation}
Furthermore, after additionally augmenting $N_x$ by an ${\mathcal{L}}^1$-nullset in $(0,\infty)$, there is no loss of generality 
in assuming that for each $\varepsilon\in(0,\infty)\setminus N_{x}$ one has
\begin{equation}\label{ia-GHGa}
\begin{array}{c}
\displaystyle
\int_{\substack{|x-y|\leq\varepsilon\\ y\in\partial\Omega}}\langle\nu,\vec{F}\rangle\,d\sigma
=\int_{\substack{|x-y|=\varepsilon\\ y\in{\mathbb{R}}^n\setminus\Omega}}
\Big\langle\frac{y-x}{\varepsilon},\vec{F}(y)\Big\rangle\,d{\mathcal{H}}^{n-1}(y),
\\[20pt]
\text{if the vector field $\vec{F}\in\big[{\mathscr{C}}^1({\mathbb{R}}^n)\big]^n$ is divergence-free in 
$B(x,\varepsilon)\setminus\Omega$}.
\end{array}
\end{equation}
Indeed, this is implied by \cite[Lemma~5.7.2, (5.7.26), pp.\,412--413]{GHA.I} presently used with $E:={\mathbb{R}}^n\setminus\Omega$
and observing that the outward unit normal to this $E$ is $-\nu$ (cf. \eqref{76g-h5Fa-ee5ds}).

For some arbitrary $\kappa>0$ fixed, write
\begin{align}\label{TD14.NEW!}
\lim\limits_{\substack{z\in\Gamma_\kappa(x)\\ z\to x}}{\mathcal{T}}f(z) 
&=\lim_{\substack{\varepsilon\in(0,\infty)\setminus N_x\\ \varepsilon\to 0}}
\,\,\lim\limits_{\substack{z\in\Gamma_\kappa(x)\\ z\to x}}
\int\limits_{\substack{|x-y|>\varepsilon\\ y\in\partial\Omega}}
\big\langle\nu(y),\vec{k}(z-y)\big\rangle f(y)\,d\sigma(y)
\nonumber\\[4pt]
&\quad+\lim_{\substack{\varepsilon\in(0,\infty)\setminus N_x\\ \varepsilon\to 0}}
\,\,\lim\limits_{\substack{z\in\Gamma_\kappa(x)\\ z\to x}}
\int\limits_{\substack{|x-y|\leq\varepsilon\\ y\in\partial\Omega}}
\big\langle\nu(y),\vec{k}(z-y)\big\rangle(f(y)-f(x))\,d\sigma(y)
\nonumber\\[4pt]
&\quad+\left(\lim_{\substack{\varepsilon\in(0,\infty)\setminus N_x\\ \varepsilon\to 0}}
\,\,\lim\limits_{\substack{z\in\Gamma_\kappa(x)\\ z\to x}}
\int\limits_{\substack{|x-y|\leq\varepsilon\\ y\in\partial\Omega}}
\big\langle\nu(y),\vec{k}(z-y)\big\rangle\,d\sigma(y)\right)f(x)
\nonumber\\[4pt]
&=:{\rm I}_1+{\rm I}_2+{\rm I}_3.
\end{align}
For each fixed threshold $\varepsilon>0$, Lebesgue's Dominated Convergence Theorem applies to the 
limit as $\Gamma_\kappa(x)\ni z\to x$ in ${\rm I}_1$ and yields
\begin{align}\label{TD15.NEW!}
{\rm I}_1 &=\lim_{\substack{\varepsilon\in(0,\infty)\setminus N_x\\ \varepsilon\to 0}}
\int\limits_{\substack{|x-y|>\varepsilon\\ y\in\partial\Omega}}
\big\langle\nu(y),\vec{k}(x-y)\big\rangle f(y)\,d\sigma(y)
\nonumber\\[6pt]
&=\lim_{\varepsilon\to 0^{+}}\int\limits_{\substack{|x-y|>\varepsilon\\ y\in\partial\Omega}}
\big\langle\nu(y),\vec{k}(x-y)\big\rangle f(y)\,d\sigma(y)=Tf(x),
\end{align}
where the second equality is a consequence of \eqref{TD15-AAAAA.NEW!}, and the last equality 
comes from \eqref{16778.bbb.222.GDL.D+M}.

To handle ${\rm I}_2$, we first observe that for every $y\in\partial\Omega$ 
and $z\in\Gamma_\kappa(x)$ we may estimate 
\begin{align}\label{TD16.NEW!}
|x-y| & \leq |z-y|+|z-x|\leq |z-y|+(1+\kappa)\,{\rm dist}(z,\partial\Omega)
\nonumber\\[4pt]
& \leq|z-y|+(1+\kappa)|z-y|=(2+\kappa)|z-y|.
\end{align}
Hence, since $f$ belongs to ${\mathscr{C}}^{\omega}(\partial\Omega)$,
\begin{align}\label{TD17.NEW!}
&\left|\big\langle\nu(y),\vec{k}(z-y)\big\rangle\right||f(y)-f(x)| 
\nonumber\\[4pt]
&\qquad\quad\leq(2+\kappa)^{n-1}\Big(\sup_{S^{n-1}}|\vec{k}|\Big)\|f\|_{\dot{\mathscr{C}}^{\omega}(\partial\Omega)}\frac{\omega(|x-y|)}{|x-y|^{n-1}}.
\end{align}
Based on this, \eqref{lkmem}, the upper Ahlfors regularity of $\partial\Omega$, 
and once again Lebesgue's Dominated Convergence Theorem, we obtain that
\begin{align}\label{TD17.tfd.NEW!}
\big|{\rm I}_2\big| &=\Bigg|\lim_{\substack{\varepsilon\in(0,\infty)\setminus N_x\\ \varepsilon\to 0}}
\int\limits_{\substack{|x-y|\leq\varepsilon\\ y\in\partial\Omega}}
\big\langle\nu(y),\vec{k}(x-y)\big\rangle(f(y)-f(x))\,d\sigma(y)\Bigg|
\nonumber\\[6pt]
&\leq(2+\kappa)^{n-1}\Big(\sup_{S^{n-1}}|\vec{k}|\Big)\cdot\|f\|_{\dot{\mathscr{C}}^{\omega}(\partial\Omega)}\,
\limsup_{\varepsilon\to 0^{+}}\int\limits_{\substack{|x-y|\leq\varepsilon\\ y\in\partial\Omega}}\frac{\omega(|x-y|)}{|x-y|^{n-1}}\,d\sigma(y)
\nonumber\\[6pt]
&\leq C\cdot\|f\|_{\dot{\mathscr{C}}^{\omega}(\partial\Omega)}\,\limsup_{\varepsilon\to 0^{+}}\int_{0}^{2\varepsilon}\omega(s)\frac{ds}{s}=0,
\end{align}
thanks to the fact that the growth function $\omega$ is Dini. Thus, 
\begin{equation}\label{TD18.NEW!}
{\rm I}_2=0.
\end{equation}

As for ${\rm I}_3$ in \eqref{TD14.NEW!}, for each fixed $z\in\Gamma_\kappa(x)$ pick 
$r\in\big(0\,,\,{\rm dist}(z,\partial\Omega)\big)$ and consider a scalar function 
$\psi\in{\mathscr{C}}^{\infty}({\mathbb{R}}^n)$ satisfying 
\begin{equation}\label{y6tgGG-y65tg.NEW!}
\text{$\psi\equiv 0$ in $B(z,r/2)$ and $\psi\equiv 1$ in ${\mathbb{R}}^n\setminus B(z,r)$}.
\end{equation}
The idea is now to use \eqref{ia-GHGa} for the vector field 
\begin{equation}\label{y6tgGG.NEW!}
\vec{F}\in\big[{\mathscr{C}}^{2}({\mathbb{R}}^n)\big]^n,\quad
\vec{F}(y):=\psi(y)\vec{k}(z-y)\,\,\text{ for all }\,\,y\in{\mathbb{R}}^n.
\end{equation}
Note that since $\vec{k}(z-\cdot)$ is divergence-free in ${\mathbb{R}}^n\setminus\{z\}$ and $\psi$ is identically $1$ in 
${\mathbb{R}}^n\setminus B(z,r)$, it follows that $\vec{F}$ is divergence-free in ${\mathbb{R}}^n\setminus\overline{B(z,r)}$. 
Bearing in mind that the latter region contains $\overline{B(x,\varepsilon)}\setminus\Omega$ for every $\varepsilon>0$, 
formula \eqref{ia-GHGa} applies to this vector field and gives that for each $\varepsilon\in(0,\infty)\setminus N_x$ we have
\begin{equation}\label{TD19.NEW!}
\int\limits_{\substack{|x-y|\leq\varepsilon\\ y\in\partial\Omega}}
\big\langle\nu(y),\vec{k}(z-y)\big\rangle\,d\sigma(y)
=\int\limits_{\substack{|x-y|=\varepsilon\\ y\in{\mathbb{R}}^n\setminus\Omega}}
\big\langle(y-x)/\varepsilon\,,\,\vec{k}(z-y)\big\rangle\,d{\mathcal{H}}^{n-1}(y).
\end{equation}
From \eqref{TD19.NEW!} and \eqref{Cau-C2.6a7.NEW!.XXX} we therefore conclude that
\begin{align}\label{TD20.NEW!}
&\lim_{\substack{\varepsilon\in(0,\infty)\setminus N_x\\ \varepsilon\to 0}}
\,\,\lim\limits_{\substack{z\in\Gamma_\kappa(x)\\ z\to x}}
\int\limits_{\substack{|x-y|\leq\varepsilon\\ y\in\partial\Omega}}
\big\langle\nu(y),\vec{k}(z-y)\big\rangle\,d\sigma(y)
\\[4pt]
&\qquad\qquad\qquad\qquad
=\lim_{\substack{\varepsilon\in(0,\infty)\setminus N_x\\ \varepsilon\to 0}}
\int\limits_{\substack{|x-y|=\varepsilon\\ y\in{\mathbb{R}}^n\setminus\Omega}}
\Big\langle\frac{y-x}{/\varepsilon}\,,\,\vec{k}(x-y)\Big\rangle\,d{\mathcal{H}}^{n-1}(y)=-\frac{\vartheta}{2}.
\nonumber
\end{align}
A combination of \eqref{TD14.NEW!}, \eqref{TD15.NEW!}, \eqref{TD18.NEW!}, and \eqref{TD20.NEW!} shows that the 
limit in the left-hand side of \eqref{TD14.NEW!} exists and matches $(-\tfrac{\vartheta}{2}I+T)f(x)$. 
In view of \eqref{TD14.abd.NEW!} and \eqref{eq:jdhdnd.D+M}, this proves that the formula in the second line 
of \eqref{2.3.21BBBB.NEW!} holds for each function $f\in{\mathscr{C}}^{\omega}(\partial\Omega)$ 
at $\sigma$-a.e. point $x\in\partial\Omega$.
\end{proof}

We have now assembled all the tools required to carry out the proof of Theorem~\ref{theor:uewmp.222}.

\vskip 0.08in
\begin{proof}[Proof of Theorem~\ref{theor:uewmp.222}]
All claims in item {\it (i)} are directly implied by Theorem~\ref{thm:GHA-T1}. 
In dealing with item {\it (ii)}, work under the additional assumption that $\Omega$ is also a uniform domain. 
For starters, we recall from \cite[Lemma~2.5.14, pp.\,436--437]{GHA.III} that $\mathcal{T}$ maps constant functions to constant functions. 
Specifically, with the piece of notation introduced in \eqref{kjshgh.FF.D+M}, for every point in $\Omega$ we have
\begin{equation}\label{JUfsf.D+M.1}
\mathcal{T}1=
\left\lbrace
\begin{array}{ll}
\displaystyle
-\vartheta, & \text{if }\,\,\Omega\,\,\text{ is bounded}, 
\\[8pt]
\displaystyle 
0, & \text{if }\,\,\Omega \text{ is an exterior domain}.
\end{array}
\right.
\end{equation}
To proceed, for each $x\in\Omega$ set $\rho(x):=\mathrm{dist}(x,\partial\Omega)$. We claim that 
there exists a constant $C\in(0,\infty)$, depending on $\Omega$, $n$, $\omega$ and $\vec{k}$, with the property that
\begin{equation}\label{ustgik.D+M}
\|{\mathcal{T}}f\|_{\mathscr{C}^{\omega_{\!Z}}(\Omega)}\leq C\bigg(\sup_{x\in\Omega}|({\mathcal{T}}f)(x)|
+\sup_{x\in\Omega}\frac{|\nabla(\mathcal{T} f)(x)|}{W_{\widetilde{\omega}}(\rho(x))}\bigg)
\leq C\|f\|_{{\mathscr{C}}^{\omega}(\partial\Omega)},
\end{equation}
for each $f\in{\mathscr{C}}^{\omega}(\partial\Omega)$. Indeed, the first inequality in \eqref{ustgik.D+M} is directly 
implied by Lemma~\ref{lemma:fokogtr}, since we are presently assuming that $\Omega$ is a uniform domain. The second 
inequality in \eqref{ustgik.D+M} is a consequence of Lemma~\ref{lemma:pokor}, whose applicability is ensured by 
\eqref{GDL-TH.it.000}, the assumption that the growth function $\omega$ is Dini (cf. \eqref{wnonin-hGFF.DDD}), and 
\eqref{JUfsf.D+M.1}. Having established \eqref{ustgik.D+M}, we conclude that the operator in \eqref{Jgsag.TTT} 
is, as claimed, well defined, linear, and bounded.

Moreover, in view of \eqref{wnonin-hGFF.cz.UU} and Lemma~\ref{holderclosure}, for each $f\in{\mathscr{C}}^{\omega}(\partial\Omega)$ 
we may canonically extend ${\mathcal{T}}f$ to a function in $\mathscr{C}^{\omega_{\!Z}}(\overline{\Omega})$, for which we retain 
the same symbol. With this in mind, if we now fix some background aperture parameter $\kappa\in(0,\infty)$, it follows 
from \eqref{2.3.21BBBB.NEW!} that 
\begin{equation}\label{2.3.21BBBB.NEW!.dddmmm}
\big(\mathcal{T}f\big)\Big|_{\partial\Omega}
=\big(\mathcal{T}f\big)\Big|^{{}^{\kappa-{\rm n.t.}}}_{\partial\Omega}
=(-\tfrac{\vartheta}{2}I+T)f\,\text{ at $\sigma$-a.e. point on }\,\partial\Omega.
\end{equation}
Since ${\mathcal{T}}f$ belongs to $\mathscr{C}^{\omega_{\!Z}}(\Omega)\equiv\mathscr{C}^{\omega_{\!Z}}(\overline{\Omega})$, 
it follows that $\big(\mathcal{T}f\big)\Big|_{\partial\Omega}$ belongs to $\mathscr{C}^{\omega_{\!Z}}(\partial\Omega)$, 
hence this is a continuous function on $\partial\Omega$ (given that $\omega_Z$ is a growth function; cf. \eqref{wnonin-hGFF.cz.UU}). 
From Theorem~\ref{thm:GHA-T1} we also know that $(-\tfrac{\vartheta}{2}I+T)f$ is a continuous function on $\partial\Omega$ 
(again, bearing in mind that $\omega_Z$ is a growth function). 
Since, as seen from \eqref{2.3.21BBBB.NEW!.dddmmm}, these two continuous functions coincide $\sigma$-a.e. on the Ahlfors regular 
set $\partial\Omega$, we ultimately conclude that they coincide everywhere on $\partial\Omega$. This establishes 
\eqref{eq:-knbmdnjm.D+M}.
Finally, \eqref{Jgsag.TTT.D+M} is a direct consequence of \eqref{Jgsag.TTT} and \eqref{omega-cond:main.333}.
\end{proof}

We close this section by noting that \eqref{ustgik.D+M} may be regarded as a significant generalization 
of a classical result due to Hardy-Littlewood.

\section{Clifford algebras and Cauchy-Clifford operators}\label{SEc:5rff}

The Clifford algebra with $n$ imaginary units is the minimal enlargement of ${\mathbb{R}}^n$ to a 
unitary real algebra $({\mathcal{C}}\!\ell_n,+,\odot)$, which is not generated as an algebra by any 
proper subspace of $\mathbb{R}^n$ and such that 
\begin{equation}\label{g6c4dDaf.1}
x\odot x=-|x|^2\,\,\text{ for every }\,\,x\in\mathbb{R}^n\hookrightarrow{\mathcal{C}}\!\ell_n. 
\end{equation}
In particular, if $\{e_j\}_{j=1}^n$ is the standard orthonormal basis in $\mathbb{R}^n$ then
\begin{equation}\label{g6c4dDaf.2}
e_j\odot e_j=-1\,\,\text{ and }\,\,e_j\odot e_k=-e_k\odot e_j\,\,\text{ whenever }\,\,1\leq j\neq k\leq n.
\end{equation}
This allows us to define a canonical embedding $\mathbb{R}^n\hookrightarrow{\mathcal{C}}\!\ell_n$ via
\begin{equation}\label{identificationrncln}
{\mathbb{R}}^n\ni x=(x_1,\dots,x_n)\equiv\sum_{j=1}^n x_j e_j\in{\mathcal{C}}\!\ell_n.
\end{equation}
As a concrete model for the Clifford algebra ${\mathcal{C}}\!\ell_n$, one can take a suitable sub-algebra 
of the matrix algebra ${\mathbb{R}}^{2^n\times 2^n}$. Thus, one should think (and at times there are distinct advantages of doing so) 
of the standard orthonormal basis in $\mathbb{R}^n$ as being identified with a suitably chosen family $\{e_j\}_{j=1}^n$ of matrices 
in ${\mathbb{R}}^{2^n\times 2^n}$. See \cite[pp.\,522-524]{GHA.I} for further details; here we only wish to remark that in this scenario, 
the Clifford algebra product $\odot$ simply becomes matrix operation. 

Going further, any element $u\in{\mathcal{C}}\!\ell_n$ has a unique representation of the form
\begin{equation}\label{eq:u-written}
u=\sum_{\ell=0}^n\sideset{}{'}\sum_{|I|=\ell}u_I e_I,\qquad u_I\in{\mathbb{R}},
\end{equation}
where $\sum'$ indicates that the sum is performed only over strictly increasing multi-indices 
$I$, i.e., $I=(i_1,i_2,\dots,i_\ell)$ with $1\leq i_1<i_2<\cdots<i_\ell\leq n$, and $e_I$ denotes 
the product $e_I:=e_{i_1}\odot e_{i_2}\odot\cdots\odot e_{i_\ell}$. We agree to set 
$e_0:=e_{\varnothing}:=1$. Throughout, we shall endow ${\mathcal{C}}\!\ell_n$ with the natural 
Euclidean metric. Hence, for each $u$ represented as above we have
\begin{equation}\label{hTfc}
|u|:=\left(\sum_I|u_I|^2\right)^{1/2}.
\end{equation} 
If $I=(i_1,\dots,i_\ell)$ with $1\leq i_1<i_2<\cdots<i_\ell\leq n$, define the conjugate 
$\overline{e_I}$ of $e_I$ as $\overline{e_I}=(-1)^{l}e_{i_l}\odot\cdots\odot e_2\odot e_1$. 
More generally, for an arbitrary element $u\in{\mathcal{C}}\!\ell_n$ represented as above we then define 
\begin{equation}\label{6ereddd}
\overline{u}:=\sum_{\ell=0}^n\sideset{}{'}\sum_{|I|=\ell}u_I\overline{e_I}.
\end{equation}
Note that $\overline{x}=-x$ for every $x\in{\mathbb{R}}^n\hookrightarrow{\mathcal{C}}\!\ell_n$, 
and $|u|=|\overline{u}|$ for every $u\in{\mathcal{C}}\!\ell_n$. One can also check that 
\begin{equation}\label{eq:clifford2n2}
|u\odot v|\leq 2^{n/2}|u||v|\,\,\text{ for all }\,\,u,v\in{\mathcal{C}}\!\ell_n,
\end{equation} 
and, in fact, $|u\odot v|=|u||v|$ whenever $u\in{\mathbb{R}}^n\hookrightarrow{\mathcal{C}}\!\ell_n$ 
or $v\in{\mathbb{R}}^n\hookrightarrow{\mathcal{C}}\!\ell_n$. For further details on Clifford 
algebras, the reader is referred to \cite[Chapter~1, pp.\,1--15]{Mit94}.

To study the boundedness of Cauchy-Clifford operators acting on Clifford algebra-valued 
functions, we need appropriate norms on the spaces to which the said functions belong. 
If $(X,\norm{\cdot}_X)$ is a Banach space then $X\otimes\mathcal{C}\!\ell_n$ will denote 
the Banach space consisting of elements of the form
\begin{equation}\label{t5fff-j632}
u=\sum_{\ell=0}^n\sideset{}{'}\sum_{|I|=\ell}u_I e_I,\qquad u_I\in X,
\end{equation}
equipped with the norm
\begin{equation}\label{t5fff-j632.B}
X\otimes\mathcal{C}\!\ell_n\ni u=\sum_{\ell=0}^n\sideset{}{'}\sum_{|I|=\ell}u_I e_I
\mapsto\norm{u}_{X\otimes\mathcal{C}\!\ell_n}:=\sum_{\ell=0}^n\sideset{}{'}\sum_{|I|=\ell}\norm{u_I}_{X}.
\end{equation}

Finally, we recall the definition of the Dirac operator $D:=\sum_{j=1}^n e_j\partial_j$. 
Given an open set $\Omega\subseteq{\mathbb{R}}^n$, we shall denote by $D_L$ and $D_R$ 
the action of $D$ on functions $u\in{\mathscr{C}}^1(\Omega)\otimes{\mathcal{C}}\!\ell_n$ 
from the left and from the right, respectively. More precisely, if $u$ is written as in \eqref{eq:u-written} 
then
\begin{equation}\label{111-6ffccc}
D_L u:=\sum_{\ell=0}^n\sideset{}{'}\sum_{|I|=\ell}\sum_{j=1}^n\frac{\partial u_I}{\partial x_j}\,e_j\odot e_I,
\quad
D_R u:=\sum_{\ell=0}^n\sideset{}{'}\sum_{|I|=\ell}\sum_{j=1}^n\frac{\partial u_I}{\partial x_j}\,e_I\odot e_j.
\end{equation}
If $\Omega\subseteq{\mathbb{R}}^n$ is an Ahlfors regular domain, $\sigma:={\mathcal{H}}^{n-1}\lfloor\partial\Omega$, 
and $\nu=(\nu_1,\dots,\nu_n)$ is the geometric measure theoretic outward unit normal to $\Omega$, we agree to identify 
the latter vector field with the Clifford algebra valued function defined at $\sigma$-a.e. point $x\in\partial\Omega$ as
$\nu(x)=\nu_1(x)e_1+\cdots+\nu_n(x)e_n\in{\mathcal{C}}\!\ell_n$. In this context, for any 
$u,v\in{\mathscr{C}}^1_c({\mathbb{R}}^n)\otimes{\mathcal{C}}\!\ell_n$ formula \eqref{WRTP22-1-GF} implies that
the following integration by parts formula holds:
\begin{equation}\label{DiracProperty2}
\int_{\partial\Omega}u(x)\odot\nu(x)\odot v(x)\,d\sigma(x) 
=\int_{\Omega}\Big((D_R u)(x)\odot v(x)+u(x)\odot(D_L v)(x)\Big)\,dx.
\end{equation}
The following lemma is proved in \cite[Lemma~4.2, p.\,982]{MitMitVer16}. To state it, we set
\begin{equation}\label{111-6ffc.4445}
[x]_s:=x_s\,\,\text{ for each }\,\,x=(x_1,\dots,x_n)\in{\mathbb{R}}^n\,\,\text{  and }\,\,s\in\{1,\dots,n\}.
\end{equation}

\begin{lemma}\label{lemmapol}
Fix $n\geq 2$ and consider an odd, harmonic, homogeneous polynomial $P(x)$, with $x\in{\mathbb{R}}^n$, 
of degree $\ell\geq 3$. Associate with it the family $P_{rs}(x)$ with $1\leq r,s\leq n$ of harmonic 
homogeneous polynomials of degree $\ell-2$ given by the formula 
\begin{equation}\label{PrsDef}
P_{rs}(x):=\frac{1}{\ell(\ell-1)}(\partial_r\partial_s P)(x),\quad x\in{\mathbb{R}}^n.
\end{equation} 

Then there exists a family of odd, $\mathscr{C}^{\infty}$ functions 
\begin{equation}\label{tfCCa}
k_{rs}:{\mathbb{R}}^n\setminus\{0\}\to{\mathbb{R}}^n\hookrightarrow{\mathcal{C}}\!\ell_n,\qquad 1\leq r,s\leq n,
\end{equation} 
which are homogeneous of degree $1-n$, such that for each $r,s\in\{1,\dots,n\}$ and each 
$x\in{\mathbb{R}}\setminus\{0\}$ one has 
\begin{equation}\label{lemmapoleq2}
\frac{P(x)}{|x|^{n-1+\ell}}=\sum_{r,s=1}^{n}\big[k_{rs}(x)\big]_s\,\,\,{\rm (}\text{cf. \eqref{111-6ffc.4445}}{\rm )},
\end{equation} 
\begin{equation}\label{lemmapoleq3}
(D_R k_{rs})(x)=\frac{\ell-1}{n+\ell-3}\frac{\partial}{\partial x_r}\left(\frac{P_{rs}(x)}{|x|^{n+\ell-3}}\right).
\end{equation} 
Moreover, there exists a finite dimensional constant $c_n>0$ such that 
\begin{equation}\label{lemmapoleq1}
\max_{1\leq r,s\leq n}\sup_{S^{n-1}}|k_{rs}|+\max_{1\leq r,s\leq n}\sup_{S^{n-1}}|\nabla k_{rs}|
\leq c_n 2^\ell\norm{P}_{L^1(S^{n-1},{\mathcal{H}}^{n-1})}.
\end{equation}
\end{lemma}

\begin{remark}\label{tfCF-6755}
Let $P$ be an odd, harmonic homogeneous polynomial of degree $\ell\geq 1$. Then for each multi-index 
$\gamma\in{\mathbb{N}}_0^n$ there exists a finite constant $c_{n,\gamma}>0$ such that 
\begin{align}\label{boundpol}
\norm{\partial^{\gamma}P}_{L^{\infty}(S^{n-1},{\mathcal{H}}^{n-1})} &\leq c_{n,\gamma}\int_{B(0,2)}|P(x)|\,dx 
\nonumber\\[4pt] 
&=c_{n,\gamma}\int_{S^{n-1}}|P(\upsilon)|\Big(\int_{0}^2 r^{n-1+\ell}\,dr\Big)\,d{\mathcal{H}}^{n-1}(\upsilon) 
\nonumber\\[4pt] 
&=c_{n,\gamma}\frac{2^\ell}{n+\ell}\norm{P}_{L^1(S^{n-1},{\mathcal{H}}^{n-1})},
\end{align} 
where we have used interior estimates for the harmonic function $P$ and its homogeneity of degree $\ell$.
\end{remark}

We continue to assume that $\Omega\subseteq{\mathbb{R}}^n$ is an Ahlfors regular domain, and abbreviate 
$\sigma:={\mathcal{H}}^{n-1}\lfloor\partial\Omega$. As before, we identify the geometric measure theoretic 
outward unit normal $\nu=(\nu_1,\dots,\nu_n)$ to $\Omega$ with the Clifford algebra valued function 
defined at $\sigma$-a.e. point $x\in\partial\Omega$ as $\nu(x)=\nu_1(x)e_1+\cdots+\nu_n(x)e_n\in{\mathcal{C}}\!\ell_n$. 
With $\varpi_{n-1}$ denoting the surface area of the unit sphere in $\mathbb{R}^n$, define
the vector field $\vec{k}=(k_j)_{1\leq j\leq n}$ as
\begin{equation}\label{eq:kforcc}
k_j(x):=\frac{1}{\varpi_{n-1}}\frac{x}{|x|^n}\odot e_j,\,\,\text{ for every }\,\,j\in\{1,\dots,n\}\,\,\text{ and }\,\,x\in\mathbb{R}^n\setminus\{0\}.
\end{equation}
Clearly, we have 
\begin{equation}\label{GDL-TH.it.000-cc}
\vec{k}\in\big[{\mathscr{C}}^N({\mathbb{R}}^n\setminus\{0\})\big]^n\,\,\text{ is odd, positive homogeneous of degree $1-n$}. 
\end{equation}
Moreover, one can prove (cf. \cite[Example~5.1.3, p.\,552-553]{GHA.IV}) that
\begin{equation}\label{eq:cc-div.vartheta}
\div\vec{k}=0\,\,\text{ in }\,\,\mathbb{R}^n\setminus\{0\}\,\,\text{ and }\,\,
\vartheta=\int_{S^{n-1}}\langle x,\vec{k}(x)\rangle\,d\mathcal{H}^{n-1}(x)=-1.
\end{equation}

Let us also remark that, if we regard the Clifford algebra ${\mathcal{C}}\!\ell_n$ as sub-algebra of the matrix algebra ${\mathbb{R}}^{2^n\times 2^n}$
(a scenario in which the Clifford algebra product $\odot$ simply becomes matrix operation), then the components $(k_j)_{1\leq j\leq n}$
of the vector field $\vec{k}$ become matrix-valued functions. This is relevant in the context of Remark~\ref{KJag.GFF.D+M}.
Specifically, the generalized double layer operator $\mathcal{T}$ defined as in \eqref{GDL-TH.it.1} 
for this choice of $\vec{k}$ becomes the boundary-to-domain Cauchy-Clifford operator, that is, the operator acting on any function 
$f\in L^1\big(\partial\Omega,\frac{\sigma(x)}{1+|x|^{n-1}}\big)\otimes{\mathcal{C}}\!\ell_n$ as
\begin{equation}\label{CauchyCliffordDomain}
(\mathcal{C}f)(x):=\frac1{\varpi_{n-1}}\int_{\partial\Omega}\frac {x-y}{|x-y|^n}\odot\nu(y)\odot f(y)\,d\sigma(y)\,\,\text{ for each }\,\,x\in\Omega.
\end{equation}
Similarly, the operator $T$ defined in \eqref{16778.bbb.222.GDL} becomes the Cauchy-Clifford
principal value operator, given by
\begin{equation}\label{CauchyCliffordBoundary}
{\mathbf{C}}f(x):=\lim_{\varepsilon\to 0^{+}}\frac{1}{\varpi_{n-1}}\,
\int\limits_{\substack{y\in\partial\Omega\\ |x-y|>\varepsilon}}\frac{x-y}{|x-y|^n}\odot\nu(y)\odot f(y)\,d\sigma(y).
\end{equation}

There is also a remarkable Cauchy Reproducing Formula involving the above boundary-to-domain Cauchy-Clifford operator. 
Concretely, with $\Omega$ as above, if the function $u\in{\mathscr{C}}^\infty(\Omega)\otimes{\mathcal{C}}\!\ell_n$ 
satisfies (for some fixed aperture parameter $\kappa>0$) 
\begin{equation}\label{DiracProperty1.a}
\begin{array}{c}
Du=0\,\,\text{ in }\,\,\Omega,\quad
{\mathcal{N}}_{\kappa}u\in L^1\big(\partial\Omega,\frac{\sigma(x)}{1+|x|^{n-1}}\big),
\,\,\,\text{ and}
\\[6pt]
u\big|^{{}^{\kappa-{\rm n.t.}}}_{\partial\Omega}\,\,\text{ exists at $\sigma$-a.e. point on }\,\,\partial\Omega,
\end{array}
\end{equation}
then (cf. \cite[Theorem~1.2.2, pp.\,41-42]{GHA.III}),
\begin{equation}\label{DiracProperty1.b}
u=\mathcal{C}\big(u\big|_{\partial\Omega }^{{}^{\kappa-{\rm n.t.}}}\big)\,\,\text{ in }\,\,\Omega.
\end{equation}

We next arrive at a central result in this paper, regarding the action of the Cauchy-Clifford operator 
on generalized H\"older spaces in Ahlfors regular domains with compact boundaries.

\begin{theorem}\label{theor:ijiji}
Let $\Omega\subseteq{\mathbb{R}}^n$ be an Ahlfors regular domain whose boundary is compact. 
Set $\sigma:={\mathcal{H}}^{n-1}\lfloor\partial\Omega$ and fix an aperture parameter $\kappa>0$. 
Also, suppose $\omega$ is a growth function on $(0,\diam(\partial\Omega))$ which is doubling 
and Dini, and associate with it the function $\omega_Z$ as in \eqref{omega-cond:main}.
Then the following statements are true.

\begin{enumerate}
\item[(a)] For every given function $f\in{\mathscr{C}}^\omega(\partial\Omega)\otimes{\mathcal{C}}\!\ell_n$ 
the limit in \eqref{CauchyCliffordBoundary} exists for $\sigma$-a.e. $x\in\partial\Omega$, and 
the operator ${\mathbf{C}}$ thus defined induces {\rm (}modulo redefining functions on $\sigma$-nullsets{\rm )}  
a well-defined, linear, and bounded mapping 
\begin{equation}\label{mmkm.xxx}
{\mathbf{C}}:{\mathscr{C}}^\omega(\partial\Omega)\otimes{\mathcal{C}}\!\ell_n
\longrightarrow{\mathscr{C}}^{\omega_{\!Z}}(\partial\Omega)\otimes{\mathcal{C}}\!\ell_n,
\end{equation}
with operator norm controlled in terms of $n$, $\omega$, $\diam(\partial\Omega)$, and the upper Ahlfors regularity 
constant of $\partial\Omega$. Furthermore, for each $f\in{\mathscr{C}}^\omega(\partial\Omega)\otimes{\mathcal{C}}\!\ell_n$
the following jump-formula holds 
\begin{equation}\label{eawdde}
\left(\mathcal{C}f\big|_{\partial\Omega}^{{}^{\kappa-{\rm n.t.}}}\right)(x)
=\tfrac{1}{2}f(x)+({\mathbf{C}}f)(x)\,\,\text{ for $\sigma$-a.e. }\,\,x\in\partial\Omega.
\end{equation}
 
Moreover, under the additional assumption that $\omega$ satisfies \eqref{omega-cond:main.111}, 
\begin{equation}\label{mmkm}
{\mathbf{C}}:{\mathscr{C}}^\omega(\partial\Omega)\otimes{\mathcal{C}}\!\ell_n
\longrightarrow{\mathscr{C}}^\omega(\partial\Omega)\otimes{\mathcal{C}}\!\ell_n
\end{equation}
is a well-defined, linear, and bounded operator, and one has {\rm (}with $I$ denoting the identity operator{\rm )}
\begin{equation}\label{pmapc}
{\mathbf{C}}^2=\tfrac{1}{4}\,I\,\,\text{ on }\,\,{\mathscr{C}}^\omega(\partial\Omega)\otimes{\mathcal{C}}\!\ell_n.
\end{equation}

\item[(b)] Make the additional assumption that $\Omega$ is a uniform domain. 
Then the boundary-to-domain Cauchy-Clifford operator introduced in \eqref{CauchyCliffordDomain} 
is well-defined, linear, and bounded in the context
\begin{equation}\label{pnmq.xxx}
\mathcal{C}:{\mathscr{C}}^\omega(\partial\Omega)\otimes{\mathcal{C}}\!\ell_n\rightarrow
{\mathscr{C}}^{\omega_{\!Z}}(\Omega)\otimes{\mathcal{C}}\!\ell_n\equiv
{\mathscr{C}}^{\omega_{\!Z}}(\overline{\Omega})\otimes{\mathcal{C}}\!\ell_n,
\end{equation}
with operator norm controlled in terms of $n$, $\omega$, and $\Omega$. Also, there exists some finite constant 
$C_{\Omega,n,\omega}>0$ with the property that 
\begin{equation}\label{C1conseq}
\sup_{x\in\Omega}|({\mathcal{C}}f)(x)|+\sup_{x\in\Omega}\frac{|\nabla({\mathcal{C}}f)(x)|}{W_{\widetilde{\omega}}(\rho(x))}
\leq C_{\Omega,n,\omega}\,\|f\|_{{\mathscr{C}}^{\omega}(\partial\Omega)}
\end{equation} 
for each function $f\in{\mathscr{C}}^{\omega}(\partial\Omega)$. Furthermore, 
given any $f\in{\mathscr{C}}^{\omega}(\partial\Omega)\otimes{\mathcal{C}}\!\ell_n$, under the canonical identification of ${\mathcal{C}}f$ 
with a function in ${\mathscr{C}}^{\omega_Z}(\overline{\Omega})\otimes{\mathcal{C}}\!\ell_n$, one has the boundary trace formula
\begin{equation}\label{eq:-knbmdnjm.D+M.WACO}
\mathcal{C}f\big|_{\partial\Omega}=\tfrac{1}{2}f+{\mathbf{C}}f\,\,\text{ everywhere on }\,\,\partial\Omega.
\end{equation} 

Lastly, as a corollary of \eqref{pnmq.xxx}, if $\omega$ is also assumed to satisfy \eqref{omega-cond:main.111} then 
the boundary-to-domain Cauchy-Clifford operator introduced in \eqref{CauchyCliffordDomain} is well-defined, linear, and bounded in the context
\begin{equation}\label{pnmq}
\mathcal{C}:{\mathscr{C}}^\omega(\partial\Omega)\otimes{\mathcal{C}}\!\ell_n
\longrightarrow{\mathscr{C}}^\omega(\overline{\Omega})\otimes{\mathcal{C}}\!\ell_n.
\end{equation}
\end{enumerate}
\end{theorem}

\begin{proof}
All claims, save for \eqref{pmapc}, are obtained by specializing Theorem~\ref{theor:uewmp.222}
(more precisely, the results of the previous section) to the case when the vector field $\vec{k}=(k_j)_{1\leq j\leq n}$ 
has components as in \eqref{eq:kforcc}. Indeed, as explained in \eqref{eq:kforcc}-\eqref{CauchyCliffordBoundary}, this gives 
rise to a boundary-to-domain generalized double layer that is precisely the boundary-to-domain Cauchy-Clifford operator defined 
in \eqref{CauchyCliffordDomain}, and to a boundary-to-boundary generalized double layer that is precisely the boundary-to-boundary 
Cauchy-Clifford operator defined in \eqref{CauchyCliffordBoundary}. Note that making the aforementioned identifications means that 
now we need a ``matrix-vector'' version of Theorem~\ref{theor:uewmp.222} of the sort discussed in Remark~\ref{KJag.GFF.D+M}.

Thus, in order to complete the proof of the current theorem there remains to justify \eqref{pmapc}
(assuming $\omega$ also satisfies \eqref{omega-cond:main.111}). To this end, given any 
$f\in{\mathscr{C}}^\omega(\partial\Omega)\otimes{\mathcal{C}}\!\ell_n$, define $u:=\mathcal{C}f$. 
By design, $u$ belongs to $\mathscr{C}^\infty(\Omega)\otimes{\mathcal{C}}\!\ell_n$ and $D_Lu=0$ in $\Omega$. 
Also, \eqref{C1conseq} (or the boundedness of the operator \eqref{pnmq.xxx}) implies
\begin{equation}\label{hdesCC.EW}
\sup_{x\in\Omega}|u(x)|\leq C_{\Omega,n,\omega}
\|f\|_{{\mathscr{C}}^\omega(\partial\Omega)\otimes{\mathcal{C}}\!\ell_n}<+\infty.
\end{equation}
Keeping in mind that ${\mathcal{N}}_{\kappa}u$ is a lower-semicontinuous function, from \eqref{hdesCC.EW} we 
conclude that ${\mathcal{N}}_{\kappa}u\in L^\infty(\partial\Omega,\sigma)\subset L^1(\partial\Omega,\sigma)$
since $\partial\Omega$ is compact and therefore has finite measure. Also, the jump-formula
\eqref{eawdde} gives
\begin{equation}\label{fmqp}
u\big|_{\partial\Omega}^{{}^{\kappa-{\rm n.t.}}}
=\big(\mathcal{C}f\big|_{\partial\Omega}^{{}^{\kappa-{\rm n.t.}}}\big)
=\big(\tfrac{1}{2}\,I+{\mathbf{C}}\big)f\,\,\text{ at $\sigma$-a.e. point in }\,\,\partial\Omega.
\end{equation}
From this and \eqref{DiracProperty1.a}-\eqref{DiracProperty1.b} we then obtain 
\begin{equation}\label{fmdae}
u={\mathcal{C}}\big(u\big|_{\partial\Omega}^{{}^{\kappa-{\rm n.t.}}}\big)
={\mathcal{C}}\big(\big(\tfrac{1}{2}\,I+{\mathbf{C}}\big)f\big)\,\,\text{ in }\,\,\Omega.
\end{equation}
Notice that $\big(\tfrac{1}{2}\,I+{\mathbf{C}}\big)f\in{\mathscr{C}}^\omega(\partial\Omega)\otimes{\mathcal{C}}\!\ell_n$ 
by \eqref{mmkm}. As such, we may once again employ the jump-formula \eqref{eawdde} for this function. 
Together with \eqref{fmqp} and \eqref{fmdae} this yields  
\begin{equation}\label{fmdaw}
\big(\tfrac{1}{2}\,I+{\mathbf{C}}\big)f
=u\big|_{\partial\Omega}^{{}^{\kappa-{\rm n.t.}}}
={\mathcal{C}}\Big(\big(\tfrac{1}{2}\,I+{\mathbf{C}}\big)f\Big)\Big|_{\partial\Omega}^{{}^{\kappa-{\rm n.t.}}}
=\big(\tfrac{1}{2}\,I+{\mathbf{C}}\big)\big(\tfrac{1}{2}\,I+{\mathbf{C}}\big)f, 
\end{equation}
at $\sigma$-a.e. point on $\partial\Omega$. Now \eqref{pmapc} follows from this and simple algebra, 
completing the proof of Theorem~\ref{theor:ijiji}.
\end{proof}

In the last portion of this section we briefly review the harmonic single potential operator 
and highlight its relationship to the Cauchy-Clifford integral operator. To set the stage, 
assume $\Omega\subseteq{\mathbb{R}}^n$ is an Ahlfors regular domain whose boundary is compact. 
Abbreviate $\sigma:={\mathcal{H}}^{n-1}\lfloor\partial\Omega$ and denote by $\nu$ the geometric 
measure theoretic outward unit normal to $\Omega$. In this setting, define the action of the harmonic 
single layer operator on each $f\in L^1(\partial\Omega,\sigma)\otimes{\mathcal{C}}\!\ell_n$ as 
\begin{equation}\label{Rad-1}
{\mathcal{S}}f(x):=\int_{\partial\Omega}E(x-y)f(y)\,d\sigma(y)\,\,\text{ for each }\,\,x\in\Omega,
\end{equation} 
where $E$ denotes the standard fundamental solution for the Laplacian. That is, for each 
$x\in{\mathbb{R}}^n\setminus\{0\}$, 
\begin{equation}\label{laplacianfundamentalsol}
E(x):=\left\lbrace
\begin{array}{ll}
\displaystyle
\frac{1}{\varpi_{n-1}(2-n)}\frac{1}{|x|^{n-2}}, &\text{ if }\,\,n\geq 3, 
\\[8pt]
\displaystyle 
\frac{1}{2\pi}\ln|x|, &\text{ if }\,\,n=2. 
\end{array}
\right.
\end{equation}
From definitions and the fact that $\nu\odot\nu=-1$ at $\sigma$-a.e. point on $\partial\Omega$, it follows 
that for each $f\in L^1(\partial\Omega,\sigma)\otimes{\mathcal{C}}\!\ell_n$ we have
\begin{equation}\label{yrDD0u42}
D_L{\mathcal{S}}f=-{\mathcal{C}}(\nu\odot f)\,\,\text{ in }\,\,\Omega. 
\end{equation}
In particular, in light of the identification in \eqref{identificationrncln}, 
\begin{equation}\label{singlelayergrad}
\nabla({\mathcal{S}}1)\equiv D_L{\mathcal{S}}1=-{\mathcal{C}}\nu\,\,\text{ in }\,\,\Omega.
\end{equation}

\section{Proof of Theorem~\ref{theor:uewmp}}\label{treED}

In broad outline, the proof proceeds along the original approach in \cite[pp.\,998-1005]{MitMitVer16}, where the 
authors have dealt with the classical growth function $\omega(t):=t^{\alpha}$ for each $t\in(0,\infty)$.

\begin{proof}[Proof of Theorem~\ref{theor:uewmp}]
Throughout, abbreviate $\rho(x):=\dist(x,\partial\Omega)$ for each $x\in\Omega$ and 
set $D:=\diam(\partial\Omega)$. Also, extend $\omega$ to $\widetilde{\omega}$, defined 
as in Remark~\ref{rem:wextension}, i.e., $\widetilde{\omega}(t):=\omega(\min\{t,D\})$ for each $t\in(0,\infty)$.
Finally, associate with $\widetilde{\omega}$ the function $W_{\widetilde{\omega}}$ as in \eqref{W+:general}. 

\vskip 0.08in
\noindent\underline{\it Proof of $(a)\Rightarrow(f)$, and of estimate \eqref{mainthmeq1}}. 
Work under the assumption that $\Omega\subseteq{\mathbb{R}}^n$ is a $\mathscr{C}^{1,\omega}$-domain with 
compact boundary (hence also a {\rm UR} domain and a uniform domain, as observed in Remark~\ref{YrafV.a}). 
Let $P$ be an odd homogeneous polynomial of degree $\ell\geq 1$ and suppose ${\mathbb{T}}:={\mathbb{T}}_{+}$ 
is associated with $P$ as in \eqref{defT}. For now, make the additional assumption that $P$ is harmonic. 
The goal is to prove that for every function $f\in{\mathscr{C}}^{\omega}(\partial\Omega)$ we have
\begin{equation}\label{inductiongoal}
\sup_{x\in\Omega}|({\mathbb{T}}f)(x)|
+\sup_{x\in\Omega}\frac{|\nabla(\mathbb{T}f)(x)|}{W_{\widetilde{\omega}}(\rho(x))} 
\leq C^{\ell}2^{\ell^2}\norm{P}_{L^1(S^{n-1},{\mathcal{H}}^{n-1})}\norm{f}_{{\mathscr{C}}^{\omega}(\partial\Omega)},
\end{equation} 
where $C\in[1,\infty)$ depends only on the dimension $n$, $\omega$, $D$, $\norm{\nu}_{\mathscr{C}^{\omega}(\partial\Omega)}$, 
and the upper Ahlfors regularity constant of $\partial\Omega$. 

The strategy is to prove \eqref{inductiongoal} by induction on the degree 
$\ell\in 2{\mathbb{N}}-1$. If $\ell=1$ then $P(x)=\sum_{j=1}^n a_j x_j$ for every 
$x=(x_1,\dots,x_n)\in{\mathbb{R}}^n$. Hence, from \eqref{boundpol}, 
\begin{equation}\label{maxprop}
\max_{1\leq j\leq n}|a_j|\leq\norm{P}_{L^{\infty}(S^{n-1},{\mathcal{H}}^{n-1})}
\leq c_n\norm{P}_{L^1(S^{n-1},{\mathcal{H}}^{n-1})}.
\end{equation} 
The idea is to use Lemma~\ref{lemma:pokor} to deal with the case $\ell=1$ of \eqref{inductiongoal}. 
To check that Lemma~\ref{lemma:pokor} is indeed applicable, note that $\nu$ is currently known to belong 
to ${\mathscr{C}}^\omega(\partial\Omega)\otimes{\mathcal{C}}\!\ell_n$ hence, from \eqref{C1conseq},
\begin{equation}\label{maineqC1}
\sup_{x\in\Omega}|({\mathcal{C}}\nu)(x)|
+\sup_{x\in\Omega}\frac{|\nabla(\mathcal{C}\nu)(x)|}{W_{\widetilde{\omega}}(\rho(x))}<+\infty.
\end{equation}
Upon observing from \eqref{Rad-1} that in the present case ${\mathbb{T}}$ may be expressed as 
\begin{equation}\label{7433W}
{\mathbb{T}}=\varpi_{n-1}\sum_{j=1}^n a_j\partial_j{\mathcal{S}},
\end{equation} 
and using \eqref{singlelayergrad} we obtain, at each $x\in\Omega$, 
\begin{align}\label{bdxyy.1}
|({\mathbb{T}}1)(x)|&\leq \varpi_{n-1}\Big(\max_{1\leq j\leq n}|a_j|\Big)|({\mathcal{C}}\nu)(x)|\,\,\text{ and}
\\[4pt]
\big|\nabla(\mathbb{T}1)(x)\big|&\leq \varpi_{n-1}\Big(\max_{1\leq j\leq n}|a_j|\Big)\big|\nabla({\mathcal{C}}\nu)(x)\big|.
\label{bdxyy.2}
\end{align}
In concert, \eqref{bdxyy.1}, \eqref{bdxyy.2}, \eqref{maineqC1}, and \eqref{maxprop} imply that \eqref{nhgfnh} holds for ${\mathbb{T}}$ 
with the constant $A_2$ in \eqref{nhgfnh} depending linearly on $\norm{P}_{L^1(S^{n-1},{\mathcal{H}}^{n-1})}$.
The other hypotheses in Lemma~\ref{lemma:pokor} are guaranteed by the structure of 
${\mathbb{T}}$ and \eqref{54332e90}. In summary, there exists $C\in[1,\infty)$ such that \eqref{inductiongoal} holds for $\ell=1$, 
by Lemma~\ref{lemma:pokor}.

Next, fix $\ell\geq 3$ and assume \eqref{inductiongoal} is satisfied when the polynomial entering 
the definition of ${\mathbb{T}}$ has degree $\leq\ell-2$. Pick an arbitrary odd harmonic homogeneous 
polynomial $P$ of degree $\ell$ and associate with it the operator ${\mathbb{T}}$ as in \eqref{defT}.
Also, for $r,s\in\{1,\dots,n\}$, let $P_{rs}(x)$ and $k_{rs}$ be as in Lemma~\ref{lemmapol} for this choice of $P$. 
For each $r,s\in\{1,\dots,n\}$ set 
\begin{equation}\label{Defkrs}
k^{rs}(x):=\frac{P_{rs}(x)}{|x|^{n+\ell-3}},\qquad x\in{\mathbb{R}}^n\setminus\{0\},
\end{equation} 
then use this as a kernel to define an integral operator acting on every  
$g\in{\mathscr{C}}^{\omega}(\partial\Omega)\otimes{\mathcal{C}}\!\ell_n$, say 
$g=\sum_{l=0}^n\sum_{|I|=l}^{'}g_I e_I$ with each $g_I$ in ${\mathscr{C}}^{\omega}(\partial\Omega)$,
according to 
\begin{align}\label{defTrs}
{\mathbb{T}}^{\,rs}g(x) &:=\int_{\partial\Omega}k^{rs}(x-y)g(y)\,d\sigma(y) 
\nonumber\\[4pt]
& =\sum_{l=0}^n\sideset{}{'}\sum_{|I|=l}
\left(\int_{\partial\Omega}k^{rs}(x-y)g_I(y)\,d\sigma(y)\right)e_I,\quad x\in\Omega.
\end{align}
Based on the induction hypothesis (used component-wise), \eqref{eq:clifford2n2}, \eqref{PrsDef}, and \eqref{boundpol} 
we may estimate 
\begin{align}\label{proofmainthmeq1}
\sup_{x\in\Omega}|({\mathbb{T}}^{\,rs}g)(x)|
+ &\sup_{x\in\Omega}\frac{|\nabla({\mathbb{T}}^{\,rs}g)(x)|}{W_{\widetilde{\omega}}(\rho(x))} 
\nonumber\\[6pt]
&\leq 2C^{\ell-2}2^{(\ell-2)^2}\norm{P_{rs}}_{L^1(S^{n-1},{\mathcal{H}}^{n-1})} 
\norm{g}_{{\mathscr{C}}^{\omega}(\partial\Omega)\otimes{\mathcal{C}}\!\ell_n} 
\nonumber\\[4pt]
&\leq c_n C^{\ell-2}2^{(\ell-2)^2}2^{\ell}\norm{P}_{L^1(S^{n-1},{\mathcal{H}}^{n-1})} 
\norm{g}_{{\mathscr{C}}^{\omega}(\partial\Omega)\otimes{\mathcal{C}}\!\ell_n}.
\end{align}

For each $r,s\in\{1,\dots,n\}$ and 
$g\in{\mathscr{C}}^{\omega}(\partial\Omega)\otimes{\mathcal{C}}\!\ell_n$ let us also define 
\begin{equation}\label{tr46-TR}
({\mathbb{T}}_{rs}g)(x):=\int_{\partial\Omega}k_{rs}(x-y)\odot g(y)\,d\sigma(y),\qquad\forall\,x\in\Omega.
\end{equation} 
By \eqref{lemmapoleq2}, for any real-valued function $f\in{\mathscr{C}}^{\omega}(\partial\Omega)\hookrightarrow
{\mathscr{C}}^{\omega}(\partial\Omega)\otimes{\mathcal{C}}\!\ell_n$ we then have
\begin{equation}\label{proofmainthmeq2}
({\mathbb{T}}f)(x)=\sum_{r,s=1}^{n}[{\mathbb{T}}_{rs}f(x)]_s\,\,\text{ for each }\,\,x\in\Omega,
\end{equation}
where, as usual, $[\cdot]_s$ singles out the $s$-th component of a vector in ${\mathbb{R}}^n$. 
If the set $\Omega$ is unbounded, fix $x\in\Omega$ along with $R_1\in(0,\dist(x,\partial\Omega))$ 
and $R_2>\dist(x,\partial\Omega)+D$, then define 
$\Omega_{R_1, R_2}:=\big(B(x,R_2)\setminus\overline{B(x,R_1)}\,\big)\cap\Omega$, 
which is a bounded ${\mathscr{C}}^{1,\omega}$-domain such that 
$\partial\Omega_{R_1,R_2}=\partial B(x,R_2)\cup\partial B(x,R_1)\cup\partial\Omega$. 
Denoting by $\nu$ the outward unit normal to $\Omega_{R_1,R_2}$ and by $\sigma$ the surface measure on
$\partial\Omega_{R_1,R_2}$, for each $r,s\in\{1,\dots,n\}$ we have
\begin{align}\label{gf6fDad}
\int_{\partial\Omega_{R_1,R_2}}k_{rs}(x-y)\odot\nu(y)\,d\sigma(y) 
&=-\int_{\Omega_{R_1,R_2}}(D_R k_{rs})(x-y)\,dy 
\nonumber\\[4pt] 
&=\frac{\ell-1}{n+\ell-3}\int_{\Omega_{R_1,R_2}}
\frac{\partial}{\partial y_r}\left(\frac{P_{rs}(x-y)}{|x-y|^{n+\ell-3}}\right)\,dy 
\nonumber\\[4pt] 
&=\frac{\ell-1}{n+\ell-3}\int_{\partial\Omega_{R_1,R_2}}k^{rs}(x-y)\nu_r(y)\,d\sigma(y),
\end{align}
where the first equality follows from \eqref{DiracProperty2}, the second from \eqref{lemmapoleq3}, 
and the third from \eqref{Defkrs} and the Divergence Theorem. Therefore, for each $r,s\in\{1,\dots,n\}$ we may write 
\begin{align}\label{mainthmlongeq}
({\mathbb{T}}_{rs} &\nu)(x)=\int_{\partial\Omega}k_{rs}(x-y)\odot\nu(y)\,d\sigma(y)
\nonumber\\[4pt] 
&=\int_{\partial\Omega_{R_1, R_2}}k_{rs}(x-y)\odot\nu(y)\,d{\mathcal{H}}^{n-1}(y)
-\int_{\partial B(x,R_1)}k_{rs}(x-y)\odot\frac{x-y}{|x-y|}\,d{\mathcal{H}}^{n-1}(y)
\nonumber\\[4pt] 
&\quad+\int_{\partial B(x,R_2)}k_{rs}(x-y)\odot\frac{x-y}{|x-y|}\,d{\mathcal{H}}^{n-1}(y)
\nonumber\\[4pt] 
&=\frac{\ell-1}{n+\ell-3}\int_{\partial\Omega_{R_1,R_2}}k^{rs}(x-y)\nu_r(y)\,d{\mathcal{H}}^{n-1}(y) 
-\int_{S^{n-1}}k_{rs}(\upsilon)\odot\upsilon\,d{\mathcal{H}}^{n-1}(\upsilon)
\nonumber\\[4pt] 
&\quad+\int_{S^{n-1}}k_{rs}(\upsilon)\odot\upsilon\,d{\mathcal{H}}^{n-1}(\upsilon)
\nonumber\\[4pt] 
&=\frac{\ell-1}{n+\ell-3}\left(\int_{\partial\Omega}k^{rs}(x-y)\nu_r(y)\,d\sigma(y)\right.
-\int_{\partial B(x,R_1)}k^{rs}(x-y)\frac{x_r-y_r}{|x-y|}\,d{\mathcal{H}}^{n-1}(y) 
\nonumber\\[4pt] 
&\quad+\left. \int_{\partial B(x,R_2)}k^{rs}(x-y)\frac{x_r-y_r}{|x-y|}\,d{\mathcal{H}}^{n-1}(y)\right)
\nonumber\\[4pt] 
&=\frac{\ell-1}{n+\ell-3}\left(({\mathbb{T}}^{\,rs}\nu_r)(x)
-\int_{S^{n-1}}k^{rs}(\upsilon)\upsilon_r\,d{\mathcal{H}}^{n-1}(\upsilon )\right.
+\left. \int_{S^{n-1}}k^{rs}(\upsilon)\upsilon_r\,d{\mathcal{H}}^{n-1}(\upsilon)\right)
\nonumber\\[4pt] 
&=\frac{\ell-1}{n+\ell-3}({\mathbb{T}}^{\,rs}\nu_r)(x).
\end{align}
From \eqref{mainthmlongeq} and \eqref{proofmainthmeq1} used with $f:=\nu_r\in{\mathscr{C}}^{\omega}(\partial\Omega)$, 
we obtain, for $r,s\in\{1,\dots,n\}$, 
\begin{align}\label{proofmainthmeq5}
\sup_{x\in\Omega}|({\mathbb{T}}_{rs}\nu)(x)| 
&+\sup_{x\in\Omega}\frac{|\nabla({\mathbb{T}}_{rs}\nu)(x)|}{W_{\widetilde{\omega}}(\rho(x))}
\nonumber\\[4pt] 
&\leq\sup_{x\in\Omega}|({\mathbb{T}}^{\,rs}\nu_r)(x)|
+\sup_{x\in\Omega}\frac{|\nabla({\mathbb{T}}^{\,rs}\nu_r)(x)|}{W_{\widetilde{\omega}}(\rho(x))}
\nonumber\\[4pt] 
&\leq c_n C^{\ell-2}2^{(\ell-2)^2}2^\ell\norm{P}_{L^1(S^{n-1},{\mathcal{H}}^{n-1})}
\norm{\nu}_{{\mathscr{C}}^{\omega}(\partial\Omega)}.
\end{align}

Assume now that $\Omega$ is bounded. Fix a point $x\in\Omega$ and define $\Omega_{R_1}:=\Omega\setminus\overline{B(x,R_1)}$ 
with $R_1$ as before, so that $\partial\Omega_{R_1}=\partial B(x,R_1)\cup\partial\Omega$. 
Proceeding as in the past, in place of \eqref{mainthmlongeq} we now obtain 
\begin{align}\label{proofmainthmeq3}
({\mathbb{T}}_{rs}\nu)(x) &=\frac{\ell-1}{n+\ell-3}({\mathbb{T}}^{\,rs}\nu_r)(x) 
-\frac{\ell-1}{n+\ell-3}\int_{S^{n-1}}k^{rs}(\upsilon)\upsilon_r\,d{\mathcal{H}}^{n-1}(\upsilon)
\nonumber\\[4pt] 
&\quad-\int_{S^{n-1}}k_{rs}(\upsilon)\odot\upsilon\,d{\mathcal{H}}^{n-1}(\upsilon),
\end{align}
for each $r,s\in\{1,\dots,n\}$. From \eqref{Defkrs}, \eqref{PrsDef}, \eqref{boundpol}, and \eqref{lemmapoleq1} we see that
\begin{equation}\label{proofmainthmeq4}
\norm{k^{rs}}_{L^{\infty}(S^{n-1},{\mathcal{H}}^{n-1})} 
+\norm{k_{rs}}_{L^{\infty}(S^{n-1},{\mathcal{H}}^{n-1})}\leq c_n 2^\ell\norm{P}_{L^{1}(S^{n-1},{\mathcal{H}}^{n-1})}.
\end{equation}
In turn, from \eqref{proofmainthmeq3}, \eqref{proofmainthmeq4}, and the fact that 
$\norm{\nu}_{{\mathscr{C}}^{\omega}(\partial\Omega)}\geq\sup_{\partial\Omega}|\nu|=1$ , we conclude that 
\eqref{proofmainthmeq5} also holds (with a possibly different dimensional constant $c_n\in(0,\infty)$) in the case 
when $\Omega$ is bounded.

Pressing on, fix $r,s\in\{1,\dots,n\}$ arbitrary and, for each 
$g\in{\mathscr{C}}^{\omega}(\partial\Omega)\otimes{\mathcal{C}}\!\ell_n$, set 
\begin{equation}\label{6f4rded}
{\widetilde{\mathbb{T}}}_{rs}g(x):=\int_{\partial\Omega}k_{rs}(x-y)\odot\nu(y)\odot g(y)\,d\sigma(y), 
\qquad\forall\,x\in\Omega.
\end{equation} 
Note that ${\widetilde{\mathbb{T}}}_{rs}1={\mathbb{T}}_{rs}\nu$, hence \eqref{proofmainthmeq5} yields 
\begin{align}\label{t5rdd9jG}
\sup_{x\in\Omega}|({\widetilde{\mathbb{T}}}_{rs}1)(x)| 
&+\sup_{x\in\Omega}\frac{|\nabla({\widetilde{\mathbb{T}}}_{rs}1)(x)|}{W_{\widetilde{\omega}}(\rho(x))}
\nonumber\\[6pt]
&\leq c_n C^{\ell-2}2^{(\ell-2)^2}2^\ell\norm{P}_{L^1(S^{n-1},{\mathcal{H}}^{n-1})} 
\norm{\nu}_{{\mathscr{C}}^{\omega}(\partial\Omega)}.
\end{align}
Thanks to the homogeneity properties of $k_{rs}$ and \eqref{lemmapoleq1} we have
\begin{equation}\label{743rtgP}
|k_{rs}(x-y)\odot\nu(y)|\leq c_n\frac{\norm{k_{rs}}_{L^{\infty}(S^{n-1},{\mathcal{H}}^{n-1})}}{|x-y|^{n-1}} 
\leq\frac{c_n 2^\ell\norm{P}_{L^{1}(S^{n-1},{\mathcal{H}}^{n-1})}}{|x-y|^{n-1}},
\end{equation} 
\begin{equation}\label{OGR-tt}
|\nabla_x(k_{rs}(x-y)\odot\nu(y))|\leq c_n\frac{\norm{\nabla k_{rs}}_{L^{\infty}(S^{n-1},{\mathcal{H}}^{n-1})}}{|x-y|^{n}} 
\leq\frac{c_n 2^\ell\norm{P}_{L^{1}(S^{n-1},{\mathcal{H}}^{n-1})}}{|x-y|^{n}},
\end{equation} 
so that we may invoke Lemma~\ref{lemma:pokor} to obtain that, for each
$g\in{\mathscr{C}}^{\omega}(\partial\Omega)\otimes{\mathcal{C}}\!\ell_n$,
\begin{align}\label{86432wdC}
\sup_{x\in\Omega}| &({\widetilde{\mathbb{T}}}_{rs}g)(x)| 
+\sup_{x\in\Omega}\frac{|\nabla({\widetilde{\mathbb{T}}}_{rs}g)(x)|}{W_{\widetilde{\omega}}(\rho(x))} 
\nonumber\\[6pt] 
&\leq C_{\Omega,n,\omega}2^\ell\Big(1+C_\omega\,\omega(D)+C^{\ell-2}2^{(\ell-2)^2}
\norm{\nu}_{{\mathscr{C}}^{\omega}(\partial\Omega)}\Big)\norm{P}_{L^1(S^{n-1},{\mathcal{H}}^{n-1})}  
\norm{g}_{{\mathscr{C}}^{\omega}(\partial\Omega)\otimes{\mathcal{C}}\!\ell_n}
\nonumber\\[6pt] 
&\leq C_{\Omega,n,\omega}2^\ell\Big(C^{\ell-2}2^{(\ell-2)^2}\norm{\nu}_{{\mathscr{C}}^{\omega}(\partial\Omega)}+1\Big) 
\norm{P}_{L^1(S^{n-1},{\mathcal{H}}^{n-1})}\norm{g}_{{\mathscr{C}}^{\omega}(\partial\Omega)\otimes{\mathcal{C}}\!\ell_n}.
\end{align} 
By specializing this to the case when $g:=\nu\odot f$ with $f\in{\mathscr{C}}^{\omega}(\partial\Omega)$ arbitrary, 
and keeping in mind that $\nu\odot\nu=-1$, we arrive at the conclusion that 
\begin{align}\label{teWD}
\sup_{x\in\Omega}| &({\mathbb{T}}_{rs}f)(x)|
+\sup_{x\in\Omega}\frac{|\nabla({\mathbb{T}}_{rs}f)(x)|}{W_{\widetilde{\omega}}(\rho(x))}
\\[6pt]
&\leq C_{\Omega,n,w}2^\ell\big(C^{\ell-2}2^{(\ell-2)^2}\norm{\nu}_{{\mathscr{C}}^{\omega}(\partial\Omega)}+1\big) 
\norm{P}_{L^1(S^{n-1},{\mathcal{H}}^{n-1})}\norm{\nu}_{{\mathscr{C}}^{\omega}(\partial\Omega)} 
\norm{f}_{{\mathscr{C}}^{\omega}(\partial\Omega)}
\nonumber
\end{align} 
for every $f\in{\mathscr{C}}^{\omega}(\partial\Omega)$. In turn, from \eqref{teWD} and \eqref{proofmainthmeq2} we obtain 
\begin{align}\label{proofmainthmeq6}
\sup_{x\in\Omega}| &(\mathbb{T}f)(x)|
+\sup_{x\in\Omega}\frac{|\nabla({\mathbb{T}}f)(x)|}{W_{\widetilde{\omega}}(\rho(x))}
\\[6pt] 
&\leq C_{\Omega,n,w}2^\ell\big(C^{\ell-2}2^{(\ell-2)^2}\norm{\nu}_{{\mathscr{C}}^{\omega}(\partial\Omega)}+1\big) 
\norm{P}_{L^1(S^{n-1},{\mathcal{H}}^{n-1})}\norm{\nu}_{{\mathscr{C}}^{\omega}(\partial\Omega)} 
\norm{f}_{{\mathscr{C}}^{\omega}(\partial\Omega)},
\nonumber
\end{align} 
for every $f\in{\mathscr{C}}^{\omega}(\partial\Omega)$. The current working hypothesis is that $\ell\in{\mathbb{N}}$ 
is odd and $\ell\geq 3$, hence $2^{(\ell-2)^2}2^\ell\leq 2^{\ell^2}$ and $2^\ell\leq C^{\ell-2}2^{\ell^{2}}$ 
if $C\geq 1$. Therefore, in such a scenario, there exists some $C(n,\omega,\Omega)\in(0,\infty)$ such that 
\begin{equation}\label{proofmainthmeq8}
C_{\Omega,n,w}2^\ell\big(C^{\ell-2}2^{(\ell-2)^2}\norm{\nu}_{{\mathscr{C}}^{\omega}(\partial\Omega)}+1\big) 
\norm{\nu}_{{\mathscr{C}}^{\omega}(\partial\Omega)}\leq C(n,\omega,\Omega)C^{\ell-2}2^{\ell^2},
\end{equation} 
where $C_{n,\omega,\Omega}$ is the constant appearing in \eqref{proofmainthmeq6}. This suggests 
that, to begin with, we take $C\geq\max\big\{1\,,\,\sqrt{C(n,\omega,\Omega)}\,\big\}$ which then 
ensures both that $C\geq 1$ and that $C(n,\omega,\Omega)C^{\ell-2}2^{\ell^2}\leq C^\ell 2^{\ell^2}$. 
Ultimately, this choice, together with \eqref{proofmainthmeq6} and \eqref{proofmainthmeq8}, 
proves \eqref{inductiongoal} and therefore the induction is complete.

Moving on, we claim that the additional assumption that $P$ is harmonic may be eliminated. 
The starting point in the justification of this claim is the observation that any homogeneous polynomial 
$P$ in ${\mathbb{R}}^n$ may be written as $P(x)=P_0(x)+|x|^2Q_0(x)$ for every $x\in{\mathbb{R}}^n$, where 
$P_0$ and $Q_0$ are homogeneous polynomials in ${\mathbb{R}}^n$, and $P_0$ is harmonic (cf. \cite[p.\,69]{Ste70}). 
If $P$ has degree $\ell=2N+1$ for some $N\in{\mathbb{N}}_0$, by iterating this process finitely many steps 
we conclude that for each $j\in\{0,1,\dots,N\}$ there exists a harmonic homogeneous polynomial 
$P_j$ of degree $\ell-2j$ such that 
\begin{equation}\label{Pdecomposition}
P(x)=\sum_{j=0}^{N}|x|^{2j}P_j(x),\qquad\forall\,x\in{\mathbb{R}}^n.
\end{equation}
Since the restrictions to the unit sphere of any two homogeneous harmonic polynomials of different 
degrees are orthogonal in $L^2(S^{n-1},{\mathcal{H}}^{n-1})$ (cf. \cite[p.\,69]{Ste70}), we have
\begin{equation}\label{uteD}
\norm{P}_{L^2(S^{n-1},{\mathcal{H}}^{n-1})}^2=\sum_{j=0}^{N}\norm{P_j}_{L^2(S^{n-1},{\mathcal{H}}^{n-1})}^2.
\end{equation} 
Thus, for every $j\in\{0,1,\dots,N\}$, 
\begin{equation}\label{63421ef}
\norm{P_j}_{L^1(S^{n-1},{\mathcal{H}}^{n-1})}
\leq c_n\norm{P_j}_{L^2(S^{n-1},{\mathcal{H}}^{n-1})} 
\leq c_n\norm{P}_{L^2(S^{n-1},{\mathcal{H}}^{n-1})}.
\end{equation} 
The upshot of \eqref{Pdecomposition} is that we may now express the action of the operator ${\mathbb{T}}$ 
on any function $f\in{\mathscr{C}}^{\omega}(\partial\Omega)$ as 
\begin{equation}\label{trDCa-765}
\mathbb{T}f(x)=\sum_{j=0}^{N}\int_{\partial\Omega}\frac{P_j(x-y)}{|x-y|^{n-1+(\ell-2j)}}f(y)\,d\sigma(y),
\qquad\forall\,x\in\Omega.
\end{equation} 
and an estimate like \eqref{inductiongoal} is valid for each integral operator in the above sum since 
each polynomial $P_j$ is odd, homogeneous, and harmonic. From this and \eqref{63421ef} we then conclude 
that, in the current general case, for each function $f\in{\mathscr{C}}^{\omega}(\partial\Omega)$ we have
\begin{equation}\label{y5r4df3}
\sup_{x\in\Omega}|({\mathbb{T}}f)(x)|
+\sup_{x\in\Omega}\frac{|\nabla({\mathbb{T}}f)(x)|}{W_{\widetilde{\omega}}(\rho(x))} 
\leq c_n\ell C^\ell 2^{\ell^2}\norm{P}_{L^2(S^{n-1},{\mathcal{H}}^{n-1})} 
\norm{f}_{{\mathscr{C}}^{\omega}(\partial\Omega)}.
\end{equation} 
By choosing again $C$ big enough to begin with, matters may be arranged so that 
$c_n\ell\leq C^\ell$ for $\ell\geq 1$. Assume this is the case and rename $C^2$ as $C$. 
From \eqref{y5r4df3} and Lemma~\ref{lemma:fokogtr} we then conclude that 
the operator ${\mathbb{T}}_{+}$ from \eqref{defT} maps $\mathscr{C}^{\omega}(\partial\Omega)$ 
continuously into $\mathscr{C}^{\omega}(\Omega_{+})$, and that the corresponding estimate claimed 
in \eqref{mainthmeq1} holds. Finally, $\Omega_{-}$ is also a $\mathscr{C}^{1,\omega}$-domain with 
compact boundary, so the same argument applies for the operator ${\mathbb{T}}_{-}$.

\vskip 0.08in
\noindent\underline{{\it Proof of} $(f)\Rightarrow(e)$}. This is a direct consequence 
of the observation that the operators ${\mathscr{R}}_j^{\pm}$ considered in item $(e)$ 
are particular cases of those considered in item $(f)$.

\vskip 0.08in
\noindent\underline{{\it Proof of} $(e)\Rightarrow(a)$}. Since $\Omega$ is a {\rm UR} domain, 
the jump-formulas from Theorem~\ref{Clpthm} hold. In view of \eqref{76g-h5Fa-ee5ds} this implies 
that for each $f\in L^1(\partial\Omega,\sigma)$, each $j\in\{1,\dots,n\}$, and $\sigma$-a.e. 
$x\in\partial\Omega$, we have
\begin{equation}\label{y5rdwW}
\left({\mathscr{R}}_j^{\pm}f\big|_{\partial\Omega_{\pm}}^{{}^{\kappa-{\rm n.t.}}}\right)(x) 
=\mp\frac{1}{2}\nu_j(x)f(x)+\lim_{\varepsilon\to 0^{+}}\int_{\partial\Omega\setminus B(x,\varepsilon)}
(\partial_j E)(x-y)f(y)\,d\sigma(y),
\end{equation} 
where $E$ is the fundamental solution for the Laplacian defined in \eqref{laplacianfundamentalsol}. 
From this, \eqref{mainthmRinCw}, Lemma~\ref{holderclosure}, and \eqref{tr-iTGav} we then see that 
for $j\in\{1,\dots,n\}$ we have
\begin{equation}\label{6543ee-U}
\nu_j={\mathscr{R}}_{j}^{-}1\big|_{\partial\Omega_{-}} 
-{\mathscr{R}}_{j}^{+}1\big|_{\partial\Omega_{+}}\in{\mathscr{C}}^{\omega}(\partial\Omega),
\end{equation} 
where the first equality holds $\sigma$-a.e. on $\partial\Omega$. 
Thus, $\nu\in{\mathscr{C}}^{\omega}(\partial\Omega)$ which ultimately goes to 
show that $\Omega$ is a ${\mathscr{C}}^{1,\omega}$-domain.

\vskip 0.08in
\noindent\underline{\it Proof of $(a)\Rightarrow (d)$, and of estimate \eqref{mainthmeq2}}. 
Suppose $\Omega\subseteq{\mathbb{R}}^n$ is a ${\mathscr{C}}^{1,\omega}$-domain 
with compact boundary and fix an odd homogenous polynomial $P$ of degree $\ell\geq 1$ in ${\mathbb{R}}^n$. 
Under these hypotheses, we have already proved (in the implication $(a)\Rightarrow(f)$) that the integral 
operators associated with $P$ as in \eqref{defT} map $\mathscr{C}^{\omega}(\partial\Omega)$ continuously 
into $\mathscr{C}^{\omega}(\Omega_{\pm})$ and that the estimates in \eqref{mainthmeq1} hold. 
Assume first that $P$ is harmonic and define 
\begin{equation}\label{trewW}
k(x):=\frac{P(x)}{|x|^{n-1+\ell}},\qquad\forall\,x\in{\mathbb{R}}^n\setminus\{0\}.
\end{equation} 
From \cite[p.\,73]{Ste70} we know that its Fourier transform is given by 
\begin{equation}\label{fourierk}
{\widehat{k}}(\xi)=\gamma_{n,\ell}\frac{P(\xi)}{|\xi|^{\ell+1}},\qquad\forall\,\xi\in{\mathbb{R}}^n\setminus\{0\},
\end{equation}
where $\gamma_{n,\ell}=O(\ell^{-(n-2)/2})$ if $n$ is even, and $\gamma_{n,\ell}=O(\ell^{-(n-4)/2})$ if 
$n$ is odd as $\ell\to\infty$.

Fix two arbitrary distinct points $x,y\in\partial\Omega$. In the case when $|\nu(x)-\nu(y)|\geq 1/2$ we have
$\omega(|x-y|)\geq(2\norm{\nu}_{{\mathscr{C}}^{\omega}(\partial\Omega)})^{-1}$, hence
\begin{align}\label{proofmainthmeq9}
\frac{|P(\nu(x))-P(\nu(y))|}{\omega(|x-y|)} &\leq 4\norm{\nu}_{\mathscr{C}^{\omega}(\partial\Omega)}
\norm{P}_{L^{\infty}(S^{n-1},{\mathcal{H}}^{n-1})} 
\nonumber\\[4pt] 
&\leq c_n 2^\ell\norm{\nu}_{{\mathscr{C}}^{\omega}(\partial\Omega)}
\norm{P}_{L^{1}(S^{n-1},{\mathcal{H}}^{n-1})},
\end{align} 
where the last inequality above is a consequence of \eqref{boundpol}. Consider next the case when $|\nu(x)-\nu(y)|\leq 1/2$.
In such a scenario, the line segment $[\nu(x),\nu(y)]$ is contained in the annulus $\overline{B(0,1)}\setminus B(0,1/2)$.
Based on this, the Mean Value Theorem, the homogeneity of $P$, and \eqref{boundpol}, we obtain 
\begin{align}\label{proofmainthmeq10}
\frac{|P(\nu(x))-P(\nu(y))|}{\omega(|x-y|)} &\leq\left(\sup_{z\in[\nu(x),\nu(y)]}|(\nabla P)(z)|\right) 
\norm{\nu}_{{\mathscr{C}}^{\omega}(\partial\Omega)}
\nonumber\\[4pt] 
&\leq\norm{\nabla P}_{L^{\infty}(S^{n-1},{\mathcal{H}}^{n-1})} 
\norm{\nu}_{{\mathscr{C}}^{\omega}(\partial\Omega)}
\nonumber\\[4pt] 
&\leq c_n 2^\ell\norm{P}_{L^{1}(S^{n-1},{\mathcal{H}}^{n-1})}
\norm{\nu}_{{\mathscr{C}}^{\omega}(\partial\Omega)}.
\end{align} 
Moreover, using \eqref{boundpol} and the fact that 
$\norm{\nu}_{{\mathscr{C}}^{\omega}(\partial\Omega)}\geq\sup_{\partial\Omega}|\nu|=1$, 
\begin{align}\label{proofmainthmeq11}
\sup_{x\in\partial\Omega}|P(\nu(x))| &\leq\norm{P}_{L^{\infty}(S^{n-1},{\mathcal{H}}^{n-1})}
\leq c_n 2^\ell\norm{P}_{L^{1}(S^{n-1},{\mathcal{H}}^{n-1})}
\nonumber\\[4pt]
&\leq c_n 2^\ell\norm{P}_{L^{1}(S^{n-1},{\mathcal{H}}^{n-1})} 
\norm{\nu}_{{\mathscr{C}}^{\omega}(\partial\Omega)}.
\end{align}
From \eqref{proofmainthmeq9}, \eqref{proofmainthmeq10}, \eqref{proofmainthmeq11}, and the 
fact that $\widehat{k}(\nu(x))=\gamma_{n,\ell}P(\nu(x))$ for each $x\in\partial\Omega$ 
(as seen from \eqref{fourierk}), we can conclude that the mapping 
$\partial\Omega\ni x\mapsto{\widehat{k}}(\nu(x))\in{\mathbb{C}}$ belongs to ${\mathscr{C}}^{\omega}(\partial\Omega)$ 
and 
\begin{equation}\label{6532ss43df}
\|{\widehat{k}}(\nu(\cdot))\|_{{\mathscr{C}}^{\omega}(\partial\Omega)} 
=\gamma_{n,\ell}\norm{P(\nu(\cdot))}_{{\mathscr{C}}^{\omega}(\partial\Omega)} 
\leq c_n 2^\ell\norm{\nu}_{{\mathscr{C}}^{\omega}(\partial\Omega)}\norm{P}_{L^1(S^{n-1},{\mathcal{H}}^{n-1})},
\end{equation} 
where the last inequality is based on the decay properties of $\gamma_{n,\ell}$ in the parameter $\ell$. 
Given that $\Omega$ is a {\rm UR} domain (cf. Remark~\ref{YrafV.a}), Theorem~\ref{Clpthm} is applicable. 
Fix an aperture parameter $\kappa>0$. Based on the jump-formula \eqref{jumplp}, \eqref{mainthmeq1}, 
Lemma~\ref{holderclosure}, \eqref{tr-iTGav}, \eqref{6532ss43df}, and \eqref{tr-iTGav.cdF} we may 
then estimate 
\begin{align}\label{rEEE.15}
\norm{Tf}_{{\mathscr{C}}^{\omega}(\partial\Omega)} 
&\leq\norm{\tfrac{1}{2i}{\widehat{k}}(\nu(\cdot))f+Tf}_{{\mathscr{C}}^{\omega}(\partial\Omega)} 
+\norm{\tfrac{1}{2i}{\widehat{k}}(\nu(\cdot))f}_{{\mathscr{C}}^{\omega}(\partial\Omega)}
\nonumber\\[4pt] 
&\leq\Big\|{\mathbb{T}f}\big|_{\partial\Omega}^{{}^{\kappa-{\rm n.t.}}}\Big\|_{{\mathscr{C}}^{\omega}(\partial\Omega)} 
+\tfrac{1}{2}\|{\widehat{k}}(\nu(\cdot))\|_{{\mathscr{C}}^{\omega}(\partial\Omega)} 
\norm{f}_{{\mathscr{C}}^{\omega}(\partial\Omega)}
\nonumber\\[4pt] 
&\leq\norm{{\mathbb{T}}f}_{{\mathscr{C}}^{\omega}(\overline{\Omega})} 
+c_n 2^\ell\norm{\nu}_{{\mathscr{C}}^{\omega}(\partial\Omega)}\norm{P}_{L^1(S^{n-1},{\mathcal{H}}^{n-1})} 
\norm{f}_{{\mathscr{C}}^{\omega}(\partial\Omega)}
\nonumber\\[4pt] 
&\leq\norm{{\mathbb{T}}f}_{{\mathscr{C}}^{\omega}(\overline{\Omega})} 
+c_n 2^\ell\norm{\nu}_{{\mathscr{C}}^{\omega}(\partial\Omega)}\norm{P}_{L^2(S^{n-1},{\mathcal{H}}^{n-1})} 
\norm{f}_{{\mathscr{C}}^{\omega}(\partial\Omega)}
\nonumber\\[4pt] 
&\leq\left(C^\ell 2^{\ell^2}+c_n 2^\ell\norm{\nu}_{{\mathscr{C}}^{\omega}(\partial\Omega)}\right) 
\norm{P}_{L^2(S^{n-1},{\mathcal{H}}^{n-1})}\norm{f}_{{\mathscr{C}}^{\omega}(\partial\Omega)}
\nonumber\\[4pt] 
&\leq(2C)^\ell 2^{\ell^2}\norm{P}_{L^2(S^{n-1},{\mathcal{H}}^{n-1})}\norm{f}_{{\mathscr{C}}^{\omega}(\partial\Omega)},
\end{align} 
assuming, without loss of generality, that 
$C\geq\max\big\{1\,,\,c_n\norm{\nu}_{{\mathscr{C}}^{\omega}(\partial\Omega)}\big\}$. 
This ultimately proves that the singular integral operator $T$ associated with the polynomial $P$ as in \eqref{trdcc-765}
maps $\mathscr{C}^{\omega}(\partial\Omega)$ boundedly into itself in the case when $P$ is harmonic.

In the case when the polynomial $P$ is not necessarily harmonic, we decompose $P$ as in  \eqref{Pdecomposition} 
and for each $f\in{\mathscr{C}}^{\omega}(\partial\Omega)$ write 
\begin{equation}\label{555.Jdd}
Tf(x)=\sum_{j=0}^{N}\lim_{\varepsilon\to 0^{+}}\int\limits_{\substack{y\in\partial\Omega\\ |x-y|>\varepsilon}}
\frac{P_j(x-y)}{|x-y|^{n-1+(\ell-2j)}}f(y)\,d\sigma(y)\,\,\text{ for $\sigma$-a.e. }\,\,x\in\partial\Omega.
\end{equation} 
Each term in the above sum may be regarded as the action of an integral operator on the function $f$, 
of the sort described in \eqref{trdcc-765}, though now associated with a harmonic homogeneous odd polynomial. 
As such, what we have proved up to this point applies and, on account of \eqref{63421ef}, gives that 
\begin{equation}\label{yrdd-kjt4}
\norm{Tf}_{{\mathscr{C}}^{\omega}(\partial\Omega)}\leq\ell(2C)^\ell 2^{\ell^2}
\norm{P}_{L^2(S^{n-1},{\mathcal{H}}^{n-1})}\norm{f}_{{\mathscr{C}}^{\omega}(\partial\Omega)}.
\end{equation} 
By choosing again $C$ big enough so that $\ell(2C)^\ell\leq C^\ell(2C)^\ell$ for $\ell\geq 1$ and renaming 
$2C^2$ as $C$, the claim in item $(d)$ and the estimate in \eqref{mainthmeq2} follow.

\vskip 0.08in
\noindent\underline{{\it Proof of} $(d)\Rightarrow (c)$}. This is a direct consequence of the observation that the 
operators $R_j$ considered in item $(c)$ are particular cases of those considered in item $(d)$.

\vskip 0.08in
\noindent\underline{{\it Proof of} $(c)\Rightarrow (b)$}. 
Work under the assumption that $\Omega$ is a {\rm UR} domain and fix an arbitrary $j\in\{1,\dots,n\}$. Observe that 
$1\in{\mathscr{C}}^{\omega}(\partial\Omega)\subseteq L^2(\partial\Omega,\sigma)$, since 
$\partial\Omega$ is assumed to be bounded. We claim that the functional 
$R_j 1\in\big({\mathscr{C}}^{\omega}(\partial\Omega)\big)^{*}$, originally defined as in 
\eqref{Rjdistributional1}, is given by the principal-value integral \eqref{Rieszpv} with $f:=1$. 
Once this claim is justified, item $(b)$ becomes a direct consequence of item $(c)$.

To see why the claim is true, pick an arbitrary 
$g\in{\mathscr{C}}^{\omega}(\partial\Omega)$, write \eqref{Rjdistributional1} for this $g$ and with $f=1$, 
then apply Lebesgue's Dominated Convergence Theorem and Fubini's Theorem, to obtain
\begin{align}\label{Rjdis-yr}
\langle R_j 1,g\rangle &=\frac{1}{2\varpi_{n-1}}\int_{\partial\Omega}
\int_{\partial\Omega}\frac{x_j-y_j}{|x-y|^n}\big(g(x)-g(y)\big)\,d\sigma(y)\,d\sigma(x)
\nonumber\\[4pt]
&=\lim_{\varepsilon\to 0^{+}}\frac{1}{2\varpi_{n-1}}\hskip 0.20in\int\hskip -0.35in
\int\limits_{\substack{(x,y)\in\partial\Omega\times\partial\Omega\\ 
|x-y|>\varepsilon}}\frac{x_j-y_j}{|x-y|^n}\big(g(x)-g(y)\big)\,d\sigma(y)\,d\sigma(x)
\nonumber\\[4pt]
&=\lim_{\varepsilon\to 0^{+}}\int_{\partial\Omega}\Bigg(
\frac{1}{\varpi_{n-1}}
\int\limits_{\substack{y\in\partial\Omega\\ |x-y|>\varepsilon}}\frac{x_j-y_j}{|x-y|^n}\,d\sigma(y)\Bigg)g(x)\,d\sigma(x).
\end{align}
There remains to show that 
\begin{align}\label{Rjdis-yr.BB}
\lim_{\varepsilon\to 0^{+}}\int_{\partial\Omega} &\Bigg(
\frac{1}{\varpi_{n-1}}
\int\limits_{\substack{y\in\partial\Omega\\ |x-y|>\varepsilon}}\frac{x_j-y_j}{|x-y|^n}\,d\sigma(y)\Bigg)g(x)\,d\sigma(x)
\nonumber\\[4pt]
&=\int_{\partial\Omega}\lim_{\varepsilon\to 0^{+}}
\Bigg(\frac{1}{\varpi_{n-1}}
\int\limits_{\substack{y\in\partial\Omega\\ |x-y|>\varepsilon}}\frac{x_j-y_j}{|x-y|^n}\,d\sigma(y)\Bigg)g(x)\,d\sigma(x)
\nonumber\\[4pt]
&=\Bigg\langle\lim_{\varepsilon\to 0^{+}}\frac{1}{\varpi_{n-1}}
\int\limits_{\substack{y\in\partial\Omega\\ |\cdot-y|>\varepsilon}}\frac{(\cdot-y)_j}{|\cdot-y|^n}\,d\sigma(y)
\,,\,g\Bigg\rangle.
\end{align}
The first equality in \eqref{Rjdis-yr.BB} is a consequence of Lebesgue Dominated Convergence Theorem since the existence of the 
respective pointwise limit is guaranteed by the existence of the principal value Riesz transform acting on $1$, while the 
desired uniform bound is provided by the fact that the maximal operator $R_j^{\rm max}$, acting on 
each $h\in L^2(\partial\Omega,\sigma)$ by
\begin{equation}\label{oigfdcvb}
\big(R_j^{\rm max}h\big)(x):=\sup_{\varepsilon>0}
\Bigg|\int\limits_{\substack{y\in\partial\Omega\\ |x-y|>\varepsilon}}\frac{x_j-y_j}{|x-y|^n}h(y)\,d\sigma(y)\Bigg|\,\,
\text{ for each }\,x\in\partial\Omega,
\end{equation}
is bounded from $L^2(\partial\Omega,\sigma)$ into $L^2(\partial\Omega,\sigma)$ (cf. e.g., \cite[Theorem~5.10.3, pp.\,458-460]{GHA.I}).
Finally, the second equality in \eqref{Rjdis-yr.BB} is immediate from definitions. This completes the proof of $(c)\Rightarrow (b)$.

\vskip 0.08in
\noindent\underline{{\it Proof of} $(b)\Rightarrow (a)$}. Assume $R_j 1\in{\mathscr{C}}^{\omega}(\partial\Omega)$ 
for every $j\in\{1,\dots,n\}$. Since we trivially have ${\mathscr{C}}^{\omega}(\partial\Omega)\subset 
L^{\infty}(\partial\Omega,\sigma)\subset\BMO(\partial\Omega,\sigma)$, from the discussion in 
Section~\ref{sec:intro} it follows that each distributional Riesz transform $R_j$ extends to a 
bounded linear operator on $L^2(\partial\Omega,\sigma)$, given by \eqref{Rieszpv}. Also, \eqref{Mabb88} guarantees that
$\partial\Omega$ is a {\rm UR} set. Granted this, we may invoke \cite[Proposition~5.1, p.\,986]{MitMitVer16}
or \cite[Propositions~2.5.29 and 2.5.32, pp.\,462--467]{GHA.III}
to conclude that the operator $\mathbf{C}$ is well-defined and bounded on $L^2(\partial\Omega,\sigma)\otimes{\mathcal{C}}\!\ell_n$, and 
\begin{equation}\label{identityClp}
{\mathbf{C}}^2 =\tfrac{1}{4}\,I\,\,\text{ on }\,\,L^2(\partial\Omega,\sigma)\otimes{\mathcal{C}}\!\ell_n.
\end{equation}
Then at $\sigma$-a.e. point on $\partial\Omega$ we may write 
\begin{equation}\label{proofmainthmeq7}
\frac{1}{4}\nu={\mathbf{C}}({\mathbf{C}}\nu)=-{\mathbf{C}}\Big(\sum_{j=1}^n(R_j 1)e_j\Big).
\end{equation}
The first equality above uses \eqref{identityClp}. The second equality in \eqref{proofmainthmeq7} uses the 
fact that at $\sigma$-a.e. point on $\partial\Omega$ we have ${\mathbf{C}}\nu=-\sum_{j=1}^n(R_j 1)e_j$, 
itself a consequence of the definition of ${\mathbf{C}}$ plus the fact that $\nu\odot\nu=-1$ and 
$x-y=\sum_{j=1}^n(x_j-y_j)e_j$ for each $x,y\in{\mathbb{R}}^n$. Thanks to \eqref{mainthmeq3} and 
Theorem~\ref{theor:ijiji}, the function in the right hand-side of \eqref{proofmainthmeq7} belongs to 
${\mathscr{C}}^{\omega}(\partial\Omega)\otimes{\mathcal{C}}\!\ell_n$, so ultimately 
$\nu\in{\mathscr{C}}^{\omega}(\partial\Omega)$ (hence $\Omega$ is a ${\mathscr{C}}^{1,\omega}$-domain, 
by Remark~\ref{YrafV.b}). 

This concludes the proof of Theorem~\ref{theor:uewmp}.
\end{proof}

\noindent{\bf Acknowledgments.}  
The second author acknowledges partial financial support from \hfill\\
MCIN/AEI/10.13039/501100011033 through grants CEX2023-001347-S and 
PID2022-141354NBI00. The third author has been supported in part by the Simons Foundation grant $\#\,$2505191.
The fourth author has been supported in part by the Simons Foundation grant $\#\,$6374811.



\begin{thebibliography}{10}
%
\bibitem{ABMMZ} R.\,Alvarado, D.\,Brigham, V.\,Maz'ya, M.\,Mitrea, and E.\,Ziad\'e, {\it On the regularity 
of domains satisfying a uniform hour-glass condition and a sharp version of the Hopf-Oleinik Boundary Point 
Principle}, Journal of Mathematical Sciences, 176 (2011), no.~3, 281--360.
%
\bibitem{BaSa1965} A.A.\,Babaev and V.V.\,Salaev, {\it On an analogue of the Plemelj-Privalov theorem in the case of 
nonsmooth curves, and its applications}, Dokl. Akad. Nauk SSSR, 161 (1965), 267--269.
%
\bibitem{BaSt1956} N.K.\,Bari and S.B.\,Stechkin, {\it Best approximations and differentiability properties of two conjugate 
functions}, Trudy Moskow Mat. Obshch., 5 (1956), 483--522.
%
\bibitem{BenSha88} C.\,Bennett and R.\,Sharpley, {\it Interpolation of Operators}, Pure and Applied
Mathematics, Vol.~129, Academic Press Inc., Boston, MA, 1988.
%
\bibitem{CMM22} M.\,Cao, J.J.\,Mar\'in, and J.M.\,Martell, {\it Extrapolation on function and modular spaces, and applications}, 
Adv. Math., 406 (2022), Paper No. 108520, 87 pp.
%
\bibitem{Casey2024} E.\,Casey, {\it Quantitative control on the Carleson $\varepsilon$-function determines regularity}, 
arXiv:2410.18422v2 [math.CA] 29 Nov. 2024.
%
\bibitem{Chr90} M.\,Christ, {\it Lectures on Singular Integral Operators}, CBMS Regional
Conference Series in Mathematics, Vol.~77, American Mathematical Society, Providence, RI, 1990.
%
\bibitem{CUMP} D.\,Cruz-Uribe, J.M.\,Martell, and C.\,P\'erez, {\it Weights, extrapolation and the theory of Rubio de Francia},
Oper. Theory Adv. Appl., 215 Birkhäuser/Springer Basel AG, Basel, 2011.
%
\bibitem{DavSem93} G.\,David and S.\,Semmes, {\it Analysis of and on Uniformly Rectifiable Sets},
Mathematical Surveys and Monographs, Vol.\,38, American Mathematical Society, Providence, RI, 1993.
%
\bibitem{Dav1949} N.A.\,Davydov, {\it Continuity of an integral of Cauchy type on a closed domain}, 
Dokl. Akad. Nauk SSSR, 64 (1949), 759-762. 
%
\bibitem{Dyn1979} E.M.\,Dyn'kin, {\it Smoothness of Cauchy type}, Zap. Nauchn. Sem. Leningrad, Otdel. Mat., Inst. Steklov, 
92 (1979), 115--133.

\bibitem{EvaGar92} L.\,Evans and R.\,Gariepy, {\it Measure Theory and Fine Properties of Functions}, 
Studies in Advanced Mathematics, CRC Press, Boca Raton, FL, 1992.
%
\bibitem{Gus1992} E.G.\,Guseinov, {\it The Plemelj-Privalov theorem for generalized H\"older classes}, 
Russian Acad. Sci. Sb. Math., 75 (1993), 165--182.
%
\bibitem{HilPhi57} E.\,Hille and R.S. Phillips, {\it Functional analysis and semi-groups}, third printing of the revised edition of 1957, 
American Mathematical Society Colloquium Publications, Vol. XXXI, Amer. Math. Soc., Providence, RI, 1974.
%
\bibitem{HofMitTay07} S.\,Hofmann, M.\,Mitrea, and M.\,Taylor, {\it Geometric and transformational 
properties of Lipschitz domains, Semmes-Kenig-Toro domains, and other classes of finite perimeter 
domains}, J. Geom. Anal., 17 (2007), no.~4, 593--647.
%
\bibitem{HofMitTay10} S.\,Hofmann, M.\,Mitrea, and M.\,Taylor, {\it Singular integrals and elliptic 
boundary problems on regular Semmes-Kenig-Toro domains}, Int. Math. Res. Not. IMRN (2010), no.~14,
2567--2865.
%
\bibitem{Magda1947} L.G.\,Magnaradze, {\it On a generalization of the Plemelj-Privalov theorem}, 
Soobshch. Akad. Nauk Gruzin. SSR, 8 (1947), 509--516.
%
\bibitem{Mas} A.\,Mas, {\it Variation for singular integrals on Lipschitz graphs: 
$L^p$ and end-point estimates}, Trans. Amer. Math. Soc., 365 (2011), no.~11, 5759--5781.
%
\bibitem{MaSa78} O.\,Martio and J.\,Sarvas, {\it Injectivity theorems in plane
and space}, Ann. Acad. Sci. Fenn. Ser. A I Math. 4, (1978--79), 383--401.
%
\bibitem{Mattila} P.\,Mattila, {\it Geometry of Sets and Measures in Euclidean 
Spaces: Fractals and Rectifiability}, Cambridge University Press, Vol.\,44, London, 1995.
%
\bibitem{GHA.I} D.\,Mitrea, I.\,Mitrea, and M.\,Mitrea, {\it Geometric Harmonic Analysis\,-\,Volume~I: A Sharp Divergence Theorem with
Nontangential Pointwise Traces}, Developments in Mathematics, Vol.\,72, Springer, 2022.
%
\bibitem{GHA.II} D.\,Mitrea, I.\,Mitrea, M.\,Mitrea, 
{\it Geometric Harmonic Analysis\,-\,Volume~II: Function Spaces Measuring Size and Smoothness on Rough Sets}, 
Developments in Mathematics, Vol.\,73, Springer, 2022.
%
\bibitem{GHA.III} D.\,Mitrea, I.\,Mitrea, and M.\,Mitrea, {\it Geometric Harmonic Analysis\,-\,Volume~III: Integral Representations,
Calder\'on-Zygmund Theory, Fatou Theorems, and Applications to Scattering}, Developments in Mathematics, Vol.\,74, Springer, 2023.
%
\bibitem{GHA.IV} D.\,Mitrea, I.\,Mitrea, and M.\,Mitrea, {\it Geometric Harmonic 
Analysis\,-\,Volume~IV: Boundary Layer Potentials in Uniformly Rectifiable Domains, and 
Applications to Complex Analysis}, Developments in Mathematics, Vol.\,75, Springer, 2023.
%
\bibitem{GHA.V} D.\,Mitrea, I.\,Mitrea, M.\,Mitrea, 
{\it Geometric Harmonic Analysis\,-\,Volume~V: Fredholm Theory and Finer Estimates for Integral Operators, 
with Applications to Boundary Problems}, Developments in Mathematics, Vol.\,76, Springer, 2023.
%
\bibitem{MitMitVer16} D.\,Mitrea, M.\,Mitrea, and J.\,Verdera, {\it Characterizing regularity of
domains via the Riesz transforms on their boundaries}, Anal. PDE, 9 (2016), no.~4, 955--1018.
%
\bibitem{Mit94} M.\,Mitrea, {\it Clifford Wavelets, Singular Integrals, and Hardy Spaces},
Lecture Notes in Mathematics, Vol.~1575, Springer-Verlag, Berlin, 1994.
%
\bibitem{Mush1953} N.I.\,Muskhelishvili, {\it Singular Integral Equations}, 2nd edition, Fizmatgiz, Moscow, 1962; 
English transl. of 1st edition Noordhoph Groningen, 1953.
%
\bibitem{NazTolVol14} F.\,Nazarov, X.\,Tolsa, and A.\,Volberg, {\it On the uniform rectifiability of
AD-regular measures with bounded Riesz transform operator: the case of codimension 1}, Acta Math., 
213 (2014), no.~2, 237--321.
%
\bibitem{Plemelj1908} J.\,Plemelj, {\it Ein Erg\"anzungssatz zur Cauchyschen Integraldarstellung 
analytischer Funktionen, Randwerte betreffend}, Monatshefte f\"ur Mathematik und Physik, 19 (1908), 205--210. 
%
\bibitem{Privalov1916} I.\,Privalov, {\it Sur les fonctions conjugu\'ees}, Bull. Soc. Math., 44 (1916), 100--103. 
%
\bibitem{Privalov1939} I.\,Privalov, {\it Sur les int\'egrales du type de Cauchy}, C.R. (Dokl.) Acad. Sci. URSS, 230 (1939), 859--863. 
%
\bibitem{Salaev1976} V.V.\,Salaev, {\it Direct and inverse estimates for a singular Cauchy integral over a closed curve}, 
Mat. Zametki, 19 (1976), 365--380.
%
\bibitem{Ste70} E.M.\,Stein, {\it Singular Integrals and Differentiability Properties of Functions}, 
Princeton Mathematical Series, No.~30, Princeton University Press, Princeton, N.J., 1970.
%
\bibitem{Zyg1923} A.\,Zygmund, {\it Sur le module de continuit\'e de la somme de la s\'erie conjugu\'ee de la s\'eries de Fourier}, 
Prace Mat.-Fiz., 33 (1923), 125--132.
\end{thebibliography}
\end{document}